\documentclass[onefignum,onetabnum]{siamart190516}
\usepackage{amsmath,amssymb,amsfonts,amsopn}
\usepackage{url}
\usepackage{dsfont}
\usepackage{mathtools}
\usepackage{todonotes}
\usepackage{subfig}
\usepackage{color}
\usepackage{graphicx}
\usepackage{tikz}
\usepackage[english]{babel}
\usepackage[cm]{sfmath}
\usepackage{pdfpages}
\usepackage{algorithm}               
\usepackage[noend]{algorithmic}      
\usepackage{tabularx,booktabs}
\usepackage{caption}

\usetikzlibrary{mindmap,trees,shapes,arrows,backgrounds,topaths}

\algsetup{indent=1.5em}

\newcommand{\R}{\mathds{R}}
\newcommand{\Z}{\mathds{Z}}
\newcommand{\N}{\mathds{N}}

\newcommand{\C}{\mathds{C}}
\newcommand{\F}{\mathds{F}}

\newcommand{\M}{\mathcal{M}}
\renewcommand{\l}{$\ell_{1}$}

\newcommand{\st}{\quad\text{s.t.}\quad}
\newcommand{\norm}[1]{\lVert {#1} \rVert}
\newcommand{\abs}[1]{\lvert {#1} \rvert}
\newcommand{\suchthat}{\,:\,}
\newcommand{\define}{\coloneqq}

\newcommand{\NP}{\textsf{NP}}
\renewcommand{\P}{\textsf{P}}
\newcommand{\ones}{\mathds{1}}

\newcommand{\cH}{\mathcal{H}}
\newcommand{\eps}{\varepsilon}

\newcommand{\cardmin}[1]{$\ell_{0}$-\textsc{min}({#1})} % X
\newcommand{\cardcons}[3]{$\ell_{0}$-\textsc{cons}({#1},\,{#2},\,{#3})} % f, k, X
\newcommand{\cardreg}[2]{$\ell_{0}$-\textsc{reg}({#1},\,{#2})} % rho, X
\newcommand{\tagcardmin}[1]{\mbox{$\ell_{0}$-\textsc{min}({#1})}} % X
\newcommand{\tagcardcons}[3]{\mbox{$\ell_{0}$-\textsc{cons}({#1},\,{#2},\,{#3})}} % f, k, X
\newcommand{\tagcardreg}[2]{\mbox{$\ell_{0}$-\textsc{reg}({#1},\,{#2})}} % rho, X

\DeclareMathOperator{\spark}{spark}
\DeclareMathOperator{\cospark}{cospark}
\DeclareMathOperator{\girth}{girth}
\DeclareMathOperator{\cogirth}{cogirth}
\DeclareMathOperator{\rank}{rank}
\DeclareMathOperator{\supp}{supp}
\DeclareMathOperator{\sign}{sign}
\DeclareMathOperator{\Diag}{Diag}
\DeclareMathOperator{\diag}{diag}

\newsiamthm{thm}{Theorem}
\newsiamthm{lem}{Lemma}
\newsiamthm{prop}{Proposition}
\newsiamthm{cor}{Corollary}
\newsiamremark{rem}{Remark}
\newsiamremark{defi}{Definition}
\AtBeginDocument{%
  \renewcommand{\figurename}{\footnotesize\textsc{Fig.}}
  \renewcommand{\tablename}{\footnotesize\textsc{Table}}
}
\headers{Cardinality Minimization, Constraints, and Regularization: A Survey}{A. M. Tillmann, D. Bienstock, A. Lodi, and A. Schwartz}
\title{Cardinality Minimization, Constraints, and Regularization: A Survey%
  \thanks{June 17, 2021; August 5, 2022.%Submitted to the editors June 17, 2021; revised August 5, 2022.
  }}
\author{Andreas M. Tillmann\thanks{TU Braunschweig, Institute for
    Mathematical Optimization, Braunschweig, Germany
    (\email{a.tillmann@tu-bs.de}).}
  \and Daniel Bienstock\thanks{Columbia University, Dept.\ of Industrial
    Engineering and Operations Research, New York, NY, USA
    (\email{dano@columbia.edu}).}
  \and Andrea Lodi\thanks{\'{E}cole Polytechnique Montr\'{e}al, CERC Data
    Science for Real-Time Decision-Making, Montr\'{e}al, QC, Canada
    (\email{andrea.lodi@polymtl.ca}).}
  \and Alexandra Schwartz\thanks{TU Dresden, Faculty of Mathematics,
    Dresden, Germany
    (\email{alexandra.schwartz@tu-dresden.de}).}
  }

\begin{document}
  \maketitle

  % max 250 words
%  \input{00_abstract.tex}
\begin{abstract}
  We survey optimization problems that involve the cardinality of variable
  vectors in constraints or the objective function. We provide a unified
  viewpoint on the general problem classes and models, and give concrete
  examples from diverse application fields such as signal and image
  processing, portfolio selection, or machine learning. The paper discusses
  general-purpose modeling techniques and broadly applicable as well as
  problem-specific exact and heuristic solution approaches. While our
  perspective is that of mathematical optimization, a main goal of this
  work is to reach out to and build bridges between the different
  communities in which cardinality optimization problems are frequently
  encountered. In particular, we highlight that modern mixed-integer
  programming, which is often regarded as impractical due to commonly
  unsatisfactory behavior of black-box solvers applied to generic problem
  formulations, can in fact produce provably high-quality or even optimal
  solutions for cardinality optimization problems, even in large-scale
  real-world settings. Achieving such performance typically draws on the
  merits of problem-specific knowledge that may stem from different fields
  of application and, e.g., shed light on structural properties of a model
  or its solutions, or lead to the development of efficient heuristics; we
  also provide some illustrative examples.
\end{abstract}
  
  % REQUIRED
  
  % REQUIRED
  \begin{AMS}
    90-02, %Operations research, mathematical programming -> Research exposition (monographs, survey articles) pertaining to operations research and mathematical programming
    90C05, %Operations research, mathematical programming -> Mathematical programming -> Linear programming
    90C06, %Operations research, mathematical programming -> Mathematical programming -> Large-scale problems in mathematical programming
    90C10, %Operations research, mathematical programming -> Mathematical programming -> Integer programming
    90C11, %Operations research, mathematical programming -> Mathematical programming -> Mixed integer programming
    90C26, %Operations research, mathematical programming -> Mathematical programming -> Nonconvex programming, global optimization
    90C30, %Operations research, mathematical programming -> Mathematical programming -> Nonlinear programming
    90C33, %Operations research, mathematical programming -> Mathematical programming -> Complementarity and equilibrium problems and variational inequalities (finite dimensions) (aspects of mathematical programming)
    90C59, %Operations research, mathematical programming -> Mathematical programming -> Approximation methods and heuristics in mathematical programming
    90C90, %Operations research, mathematical programming -> Mathematical programming -> Applications of mathematical programming
    62J07, %Statistics -> Linear inference, regression -> Ridge regression; shrinkage estimators (Lasso)
    68T99, %Computer Science -> Artificial Intelligence -> "None of the above, but in this section"
    94A12, %Information and communication theory, circuits -> Communication, information -> Signal theory (characterization, reconstruction, filtering, etc.)
    91G10  %Game theory, economics, finance, and other social and behavioral sciences -> Actuarial science and mathematical finance -> Portfolio theory
  \end{AMS}

\section{Introduction}\label{sec:intro}
The cardinality of variable vectors occurs in a plethora of optimization
problems, in either constraints or the objective function. In the
following, we attempt to describe the broad landscape of such problems with
a general emphasis on \emph{continuous} variables. This restriction serves
as a natural distinguishing feature from a myriad of classical operations
research or combinatorial optimization problems, where ``cardinality''
typically appears in the form of minimizing or limiting the number of some
objects associated with (non-auxiliary, i.e., structural) binary decision
variables. Cardinality restrictions on general variables are thus of a
decidedly different flavor, and also require different modeling and
solution techniques than those immediately available in the binary case.

The general classes of problems we are interested in can be formalized as
follows:
\begin{itemize}
  \item \textbf{Cardinality Minimization Problems}
    \begin{equation}\label{cardmin}
      \min~\norm{x}_{0} \st x\in X\subset{\R}^{n}, \tag{\tagcardmin{$X$}}
    \end{equation}
  \item \textbf{Cardinality-Constrained Problems}
    \begin{equation}\label{cardcons}
      \min~f(x) \st \norm{x}_0\leq k,\quad x\in X\subseteq\R^n, \tag{\tagcardcons{$f$}{$k$}{$X$}}
    \end{equation}
  \item \textbf{Regularized Cardinality Problems}
    \begin{equation}\label{cardregul}
      \min~\norm{x}_0 + \rho(x) \st x\in X\subseteq\R^n,  \tag{\tagcardreg{$\rho$}{$X$}}
    \end{equation}
\end{itemize}
where we use $\norm{x}_0\define\abs{\supp(x)}=\abs{\{j:x_j\neq 0\}}$ (the
so-called ``$\ell_0$-norm''), $f:\R^n\to\R$, $k\in\N$, and
$\rho:\R^n\to\R_+$. The set $X$ in any of these problems can be used to
impose further constraints on~$x$. For simplicity, we will usually simply
write out the constraints rather than fully state the corresponding
set~$X$; e.g., we may write \cardcons{$f$}{$k$}{$g(x)\leq 0$} instead of
\cardcons{$f$}{$k$}{$\{x\in\R^n:g(x)\leq 0\}$}.

Most concrete problems we will discuss belong to one of these three
classes, although we will also encounter variations and extensions.
Indeed, very similar problems may arise in very different fields of
application, sometimes resulting in some methodology being reinvented or
researchers being generally unaware of relevant results and developments
from seemingly disparate communities. Moreover, the incomplete transfer of
knowledge between different disciplines may prevent progress in the
resolution of some problems that could strongly benefit of new approaches
for similar problems, developed with completely different applications in
mind. With this document, we hope to provide a useful roadmap connecting
several disciplines and offering an overview of the many different
computational approaches that are available for cardinality optimization
problems. Note that a similar overview was given a couple of years ago
in~\cite{SZL13}, but with a much more limited scope of cardinality problems
and their aspects than we consider here (albeit discussing the related case
of semi-continuous variables in more detail, in particular associated
perspective reformulations, which we mostly skip). Moreover, significant
advances have been achieved in just these past few years, which we include
in this survey.

To emphasize the cross-disciplinary nature of many of the cardinality
optimization problem classes and to provide a clear reference point for
members of different communities to recognize their own problem of interest
in this survey, we will group our overview of various concrete such
problems according to the respective application areas and point out
overlaps and differences. The solution methods we shall discuss cover both
exact and heuristic approaches; our own mathematical programming
perspective tends to favor exact models and algorithms that can provide
provable guarantees on solution quality, a stance that appears to be less
commonly taken in practical applications. This is ``a feature, not a bug''
of the present paper---we hope to bring across that in many cases, mixed-integer programming (MIP) offers an attractive alternative to widely-used heuristic methods. Generally, a typical first step in that direction is experimenting with off-the-shelf solvers to tackle basic MIP formulations. Depending on the application, this may already work very well, especially when solution quality is more important than speed. Importantly, MIP solvers also provide
certifiable error bounds of the computed solution w.r.t.\ the optimum if
terminated prematurely (e.g., when imposing a runtime limit), in contrast to many heuristic methods without general quality guarantees that
are commonly employed in various applications. Moreover, improvement to optimality is often not hopeless and can be achieved either by simply allowing more solving time, or by improving the underlying mathematical model formulation and/or incorporating knowledge of the problem at hand into the MIP solver. Thus, as subsequent steps to substantially improve speed and scalability of MIP approaches, it is worth revisiting the model and guiding or enhancing the MIP solver by customizing existing (and/or adding new problem-specific) algorithmic components---a fact we will document with some examples.

We organize the subsequent discussion as follows: In the remainder of this introductory section, we will clarify some relationships between the main problem classes and fix our notation. Then, in Section~\ref{sec:concreteprobs}, we describe the most common different realizations of the above problems \cardmin{$X$}, \cardcons{$f$}{$k$}{$X$}, and \cardreg{$\rho$}{$X$} as they occur in diverse fields like signal processing, compressed sensing, portfolio opti\-mization, and machine
learning; some further related problems are also discussed. In
Section~\ref{sec:exactopt}, we summarize exact modeling techniques (in
particular, mixed-integer linear and nonlinear programming) and algorithmic approaches from the literature, and provide some exemplary numerical experiments to illustrate how the sometimes unsatisfactory performance of general-purpose models and MIP solvers may be significantly improved by some advanced modeling tricks and,
especially, by integrating problem-specific knowledge and heuristic
methods. This is followed in Section~\ref{sec:inexactopt} by reviewing the
plethora of proposed relaxations, regularization, and heuristic schemes,
including popular $\ell_1$-norm and atomic norm minimization as well as greedy methods. Finally, in Section~\ref{sec:scalability}, we address scalability aspects of exact and approximate/heuristic algorithms, and then conclude the paper in Section~\ref{sec:conclusion}.

Moreover, Table~\ref{tab:pointersForBaseProblems} provides an
alternative overview meant to facilitate navigating this document if one is primarily interested in one specific problem. Since this paper covers too many different problems to provide such an overview for all of them, we do so exemplarily for three of the most-widely used problems, and note that the pointers to topics and locations given for these should also be helpful for many other related problems as well. Specifically, Table~\ref{tab:pointersForBaseProblems} covers \cardmin{$\norm{Ax-b}_2\leq\delta$}, \cardcons{$\norm{Ax-b}_2$}{$k$}{$\R^n$}, and \cardreg{$\tfrac{1}{2\lambda}\norm{Ax-b}_2^2$}{$\R^n$}, which will be formally introduced first in Section~\ref{sec:concreteprobs:signalprocessing} in the context of signal processing, but also appear in virtually all other application areas, 
and can be seen as the ``base problems'' for various related variants and extensions.

\begin{table}[hbtp]
  \begin{center}
  \caption{\footnotesize\itshape Some pointers to locations in this survey where information for the exemplary problems \mbox{\normalfont\cardmin{$\norm{Ax-b}_2\leq\delta$},} {\normalfont\cardcons{$\norm{Ax-b}_2$}{$k$}{$\R^n$}}, and {\normalfont\cardreg{$\tfrac{1}{2\lambda}\norm{Ax-b}_2^2$}{$\R^n$}} can be found. Section~\ref{sec:exactopt:modelingcard} provides several reformulations of cardinality that are applicable to all problems; scalability of algorithms is discussed in Section~\ref{sec:scalability}.\hfill~}
  \label{tab:pointersForBaseProblems}
  \footnotesize % SIAM style guide: tables are set in 8pt... 
  \begin{tabular*}{\textwidth}{@{\extracolsep{\fill}}@{\quad}l@{\quad}l@{\quad}p{0.57\textwidth}@{\quad}}
        \toprule
        problem & location & information\\
        \midrule
        \cardmin{$\norm{Ax-b}_2\leq\delta$} &
        %Sect. \ref{sec:exactopt:modelingcard} & several reformulations of cardinality \\
        Sect. \ref{sec:exactopt:cardmin} & exact solution methods, mostly based on mixed-integer programming   \\
        &Sect. \ref{sec:inexactopt:surrogates} & $\ell_1$-surrogate problem BPDN($\delta,\R^n$) and solution approaches, e.g., \hyperref[par:homotopyMethods]{homotopy methods}, \hyperref[par:ADMM]{ADMM}, and \hyperref[par:smoothingTechniques]{smoothing techniques}\\
        &Sect. \ref{sec:inexactopt:approximation_objective} & heuristics based on nonconvex approximations (but only considering related constraints)\\
        &Sect. \ref{sec:inexactopt:greedy_heuristics} & other heuristics such as \hyperref[par:otherPursuitGreedySchemes]{subspace pursuit}  \\ &&\\        \cardcons{$\norm{Ax-b}_2$}{$k$}{$\R^n$} &
        %Sect. \ref{sec:exactopt:modelingcard} & several reformulations of cardinality  \\
        Sect. \ref{sec:exactopt:cardcons} & exact solution methods, mostly based on mixed-integer programming  \\
        &Sect. \ref{sec:inexactopt:surrogates} & $\ell_1$-surrogate problem LASSO($\tau,\R^n$) and solution approaches, e.g., a \hyperlink{SPGL1}{spectral projected gradient method} and \hyperref[par:smoothingTechniques]{smoothing techniques}\\
        &Sect. \ref{sec:inexactopt:relaxation_constraints} & heuristics based on nonconvex formulations  \\
        &Sect. \ref{sec:inexactopt:greedy_heuristics} & other heuristics such as \hyperref[par:hardThresholding]{iterative hard-thresholding} and \hyperref[par:matchingPursuit]{matching pursuit} \\ &&\\
        \cardreg{$\tfrac{1}{2\lambda}\norm{Ax-b}_2^2$}{$\R^n$} &
        %Sect. \ref{sec:exactopt:modelingcard} & several reformulations of cardinality \\
        Sect. \ref{sec:exactopt:cardreg} & theoretical discussion of the general problem~\cardreg{$\rho$}{$X$}, some exact solution methods for special cases, possibilities for \cardreg{$\tfrac{1}{2\lambda}\norm{Ax-b}_2^2$}{$\R^n$} \\
        &Sect. \ref{sec:inexactopt:surrogates} & $\ell_1$-surrogate problem $\ell_1$-LS($\lambda,\R^n$) and solution approaches, e.g., \hyperref[par:ISTA]{iterative soft-thresholding} and \hyperlink{SMNewt}{semismooth Newton methods}\\
        &Sect. \ref{sec:inexactopt:relaxation_constraints} & heuristics based on nonconvex reformulations \\
        &Sect. \ref{sec:inexactopt:greedy_heuristics} & other heuristics such as \hyperref[par:hardThresholding]{iterative hard-thresholding}  \\
        \bottomrule
    \end{tabular*}    
  \end{center}
\end{table}

\subsection{Relationships Between Main Problem Classes}\label{sec:intro:cardpros_relations}
At least in some communities, it appears to be folklore knowledge that problems belonging to the classes \cardmin{$X$}, \cardcons{$f$}{$k$}{$X$}, or \cardreg{$\rho$}{$X$} can sometimes be equivalent in the sense that they share optimal solutions under certain assumptions on the cardinality, constraint and regularization parameters.
Indeed, the fact that in widely-used surrogate models like $\ell_1$-norm problems, such equivalences always hold for the right parameter choices (cf.~Section~\ref{sec:inexactopt:surrogates}), might mislead one to presume the same is true for the $\ell_0$-based problems. However, this is generally not the case. We formalize (non-)equivalence statements for the three main classes of cardinality problems in the following result, where we let $\tagcardmin{$\delta$}\define \min\{\norm{x}_0\,:\,f(x)\leq\delta,~x\in X\}$ be the typical slight variation of the cardinality minimization problem that most naturally relates to the other problem classes; to simplify notation, we also abbreviate $\ell_0$-\textsc{cons}$(k)\define\tagcardcons{$f$}{$k$}{$X$}$, and $\ell_0$-\textsc{reg}$(\lambda)\define\tagcardreg{$\tfrac{1}{\lambda}f$}{$X$}$.

\begin{prop}\label{prop:cardprobs_relations}\
Let $\lambda>0$, $\delta\geq 0$, $X\subseteq\R^n$, and $f:\R^n\to\R$.
    \begin{enumerate}
    \item\label{prop:cardprobs:i} If $x^*$ is an optimal solution of $\ell_0$-\textsc{reg}$(\lambda)$, then it also optimally solves $\ell_0$-\textsc{cons}$(k)$ for $k=\norm{x^*}_0$ and \cardmin{$\delta$} for $\delta=f(x^*)$. The reverse implications are not true in general.
    \item\label{prop:cardprobs:ii} If all optimal solutions $x^*$ of $\ell_0$-\textsc{cons}$(k)$ have the same cardinality~$\norm{x^*}_0$, then they all also solve \cardmin{$\delta$} for $\delta=f(x^*)$. The equal-cardinality assumption cannot be dropped in general.
    \item\label{prop:cardprobs:iii} If all optimal solutions $x^*$ of \cardmin{$\delta$} have the same function value~$f(x^*)$, then they all also solve $\ell_0$-\textsc{cons}$(k)$ for $k=\norm{x^*}_0$. The equal-value assumption cannot be dropped in general.
    \end{enumerate}
\end{prop}

\begin{proof}
First, let $x^*$ solve $\ell_0$-\textsc{reg}$(\lambda)$. Then, for all $x\in X$ with $\norm{x}_0\leq k=\norm{x^*}_0$, it holds that
$%\[
    f(x^*)= f(x^*)+\lambda(\norm{x^*}_0-k)\leq f(x)+\lambda(\norm{x}_0-k)\leq f(x)
$%\]
, and for all $x\in X$ with $f(x)\leq\delta=f(x^*)$, we have
$
    \norm{x^*}_0 = \norm{x^*}_0+\tfrac{1}{\lambda}(f(x^*)-\delta)\leq \norm{x}_0+\tfrac{1}{\lambda}(f(x)-\delta)\leq\norm{x}_0,
$
 which shows that $x^*$ solves both $\ell_0$-\textsc{cons}$(\norm{x^*}_0)$ and \cardmin{$f(x^*)$} as claimed. To show that the reverse implications do not hold in general, consider the case $X=\R^2$ and $f(x)=\norm{Ax-b}_2^2$ with
\[
A=\begin{pmatrix}0&1\\1&2\end{pmatrix}\quad\text{and}\quad b=\begin{pmatrix}1\\0\end{pmatrix}.
\]
Then, in particular, $\hat{x}_1=(0,\tfrac{1}{5})^\top$ optimally solves both $\ell_0$-\textsc{cons}$(1)$ and \cardmin{$\tfrac{4}{5}$}; $\ell_0$-\textsc{cons}$(0)$ is solved by $\hat{x}_0=(0,0)^\top$ with $f(\hat{x}_0)=1$ and $\ell_0$-\textsc{cons}$(2)$ by $\hat{x}_2=(-2,1)^\top$ with $f(\hat{x}_2)=0$. Thus, the optimal value of $\ell_0$-\textsc{reg}$(\lambda)$ as a function of $\lambda>0$ is
\[
\min\big\{ \tfrac{0}{\lambda}+2, \tfrac{4}{5\lambda}+1, \tfrac{1}{\lambda}+0\big\}=\begin{cases}2, &\lambda\in(0,\tfrac{1}{2}],\\\tfrac{1}{\lambda}, &\lambda\in[\tfrac{1}{2},\infty).\end{cases}
\]
For $\lambda\in(0,\tfrac{1}{2})$, the solution to $\ell_0$-\textsc{reg}$(\lambda)$ is $\hat{x}_2$, for $\lambda=\tfrac{1}{2}$, both $\hat{x}_0$ and $\hat{x}_2$ are optimal, and for $\lambda>\tfrac{1}{2}$, only $\hat{x}_0$ is. This means that $\hat{x}_1$ cannot by recovered by $\ell_0$-\textsc{reg}$(\lambda)$ for \emph{any} $\lambda>0$, which concludes the proof of statement~\ref{prop:cardprobs:i}.

We skip the straightforward proofs of the positive statements in~\ref{prop:cardprobs:ii} and~\ref{prop:cardprobs:iii}. To show that these  implications are not true without the respective assumptions, let $X=\R^2$ and first consider $f(x)=x_1^2$. Then, any $x^*=(0,c)^\top$ with $c\in\R$ optimally solves $\ell_0$-\textsc{cons}$(1)$, but not \cardmin{$0$} unless $c=0$.
Now, let $f(x)=(x_1-2)^2$. Then, $\hat{x}_1=(1,0)^\top$ and $\hat{x}_2=(2,0)^\top$ are both optimal for \cardmin{$1$}, but $\hat{x}_1$ does not solve $\ell_0$-\textsc{cons}$(1)$.
\end{proof}

Note that points~\ref{prop:cardprobs:ii} and~\ref{prop:cardprobs:iii} of Proposition~\ref{prop:cardprobs_relations} imply equivalence of $\ell_0$-\textsc{cons}$(k)$ and \cardmin{$\delta$}, for the appropriate values of $k$ and $\delta$, in case of \emph{solution uniqueness}, which is often an important desideratum (e.g., for signal reconstruction). However, the parameter values that yield such an equivalence are typically not known a priori.

\subsection{Notation}\label{sec:intro:notation}
We let $\R_+$ denote the set of nonnegative real numbers. For a natural
number $n\in\N$, we abbreviate $[n]\define\{1,2,\dots,n\}$. The complement
of a set $S\subset T$ is denoted by $S^c$. The cardinality of a vector $x$
is denoted as $\norm{x}_0\define\abs{\supp(x)}=\abs{\{i:x_i\neq 0\}}$,
where $\supp(x)$ is its support (i.e., index set of nonzero entries). The
standard $\ell_p$-norm (for $1\leq p <\infty$) of a vector $x\in\R^n$ is
defined as $\norm{x}_p\define(\sum_{i=1}^{n}|x_i|^p)^{1/p}$, and
$\norm{x}_\infty\define\max|x_i|$. For a matrix $A$, $\norm{A}_\text{F}$
denotes is Frobenius norm, and $A_i$ its $i$-th column. For a set $S$ and a
vector $x$ or matrix~$A$, $x_S$ and $A_S$ denote the vector restricted to
indices in $S$ or the column-submatrix induced by $S$, respectively. We use
$\ones$ to denote an all-ones vector, and~$I$ to denote the identity
matrix, of appropriate dimensions. A superscript $\top$ denotes
transposition (of a vector or matrix). A diagonal matrix build from a
vector $z$ is denoted by $\Diag(z)$, and conversely, $\diag(Z)$ extracts
the diagonal of a matrix $Z$ as a vector. For vectors $\ell,u\in\R^n$, we
sometimes abbreviate $\ell\leq x\leq u$ (i.e., $\ell_i\leq x_i\leq u_i$ for
all $i\in[n]$) as $x\in[\ell,u]$, extending the standard interval notation
to vectors.

\section{Prominent Cardinality Optimization Problems}\label{sec:concreteprobs}

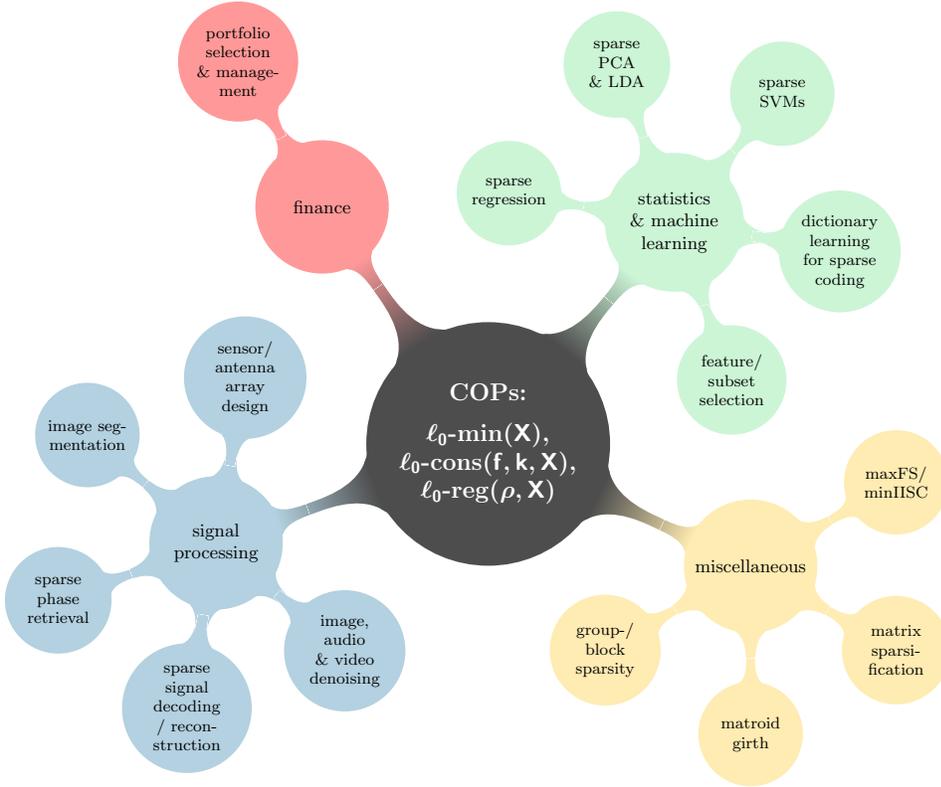
\begin{figure}[t]
  \centering
  \begin{tikzpicture}[scale=0.77,every node/.style={transform shape}]
  \path[mindmap,concept color=black!70, level 1 concept/.append style={sibling angle=75}]
    node[concept,text=white] {\textbf{COPs:}\\[0.5em]\textbf{\boldmath\cardmin{$X$},}\\\textbf{\boldmath\cardcons{$f$}{$k$}{$X$},}\\\textbf{\boldmath\cardreg{$\rho$}{$X$}}}
    [clockwise from=200]
    child[concept color=blue!60!green!30] {
      node[concept] {signal processing}
      [clockwise from=-40]
      child { node[concept] {image, audio \& video denoising} }
      child { node[concept] {sparse signal decoding / reconstruction} }
      child { node[concept] {sparse phase retrieval} }
      child { node[concept] {image segmentation} }
      child { node[concept] {sensor/ antenna array design} }
    }  
    child[concept color=red!40] {
      node[concept] {finance}
      [clockwise from=120]
      child { node[concept] {portfolio selection \& manage\-ment} }
    }
    child[concept color=green!80!blue!20] { node[concept] {statistics \& machine learning} 
      [clockwise from=-190]
      child { node[concept] {sparse regression} }
      child { node[concept] {sparse PCA \& LDA} }
      child { node[concept] {sparse SVMs} }
      child { node[concept] {dictionary learning for sparse coding} }
      child { node[concept] {feature/ subset selection} }
    }
    child[concept color=orange!50!yellow!30] { node[concept] {miscellaneous}
      [clockwise from=30]
      child { node[concept] {maxFS/ minIISC} }
      child { node[concept] {matrix sparsification} }
      child { node[concept] {matroid girth} }
      child { node[concept] {group-/ block sparsity} }
    };
  \end{tikzpicture}
  \caption{\footnotesize\itshape An overview of broader application areas and exemplary problems
    therein that share significant interests in cardinality
    optimization. Overlaps of concrete problem types across fields are
    quite common; for instance, cardinality-constrained Markowitz portfolio
    selection can be rewritten in the form of a constrained sparse
    regression problem (cf.~\cite{BCW18}), which in turn is of the same
    class as certain signal denoising/reconstruction models, see
    Sections~\ref{sec:concreteprobs:signalprocessing}--\ref{sec:concreteprobs:statML}
    for details.}
  \label{fig:applications_mindmap}
\end{figure}

Cardinality optimization problems (COPs, for short) abound in several
different areas of application, such as medical imaging (e.g., X-ray
tomography), face recognition, wireless sensor network design,
stock-picking, crystallography, astronomy, computer vision, classification
and regression, interpretable machine learning, or statistical data
analysis, to name but a few. In this section, we highlight the most
prominent realizations of such problems. To facilitate ``mapping'' concrete
problems to concrete applications, we structure the section according to
the three broad fields in which cardinality optimization problems are
encountered most frequently: signal and image processing, portfolio
optimization and management, and high-dimensional statistics and machine
learning; further related COPs and extensions are gathered in a final
subsection. Along these lines, a first broad overview of applications is
provided in Figure~\ref{fig:applications_mindmap}.

\subsection{Signal and Image Processing}\label{sec:concreteprobs:signalprocessing}
In the broad field of signal processing, it has been found that signal
sparsity can be exploited beneficially in several tasks, e.g., to remove
noise from image or audio data or to reduce the amount of measurements
needed to faithfully reconstruct signals from observations. In particular,
the advent of \emph{compressed sensing} (see~\cite{FR13} for a thorough
introduction) has sparked a tremendous interest in several core cardinality
optimization problems in the past 15~years or so.

At first, the focus was on reconstruction from linear measurements ($b=Ax$), but research quickly also expanded to different nonlinear settings. We will discuss the respective fundamental sparse recovery tasks in Sections~\ref{sec:concreteprobs:signalprocessing:canonical-linear} and~\ref{sec:concreteprobs:signalprocessing:canonical-nonlinear} below; Section~\ref{sec:concreteprobs:signalprocessing:dictionary} covers important generalizations of the main sparsity concept. 

Before we get started, a brief remark on the measurement matrices~$A$ seems in order:
In signal processing applications, $A$ is typically not fully generic but assumes certain forms and properties arising from an underlying physical \emph{measurement} model or setup. Also, much of the theory for efficient solvability (see, e.g., Section~\ref{sec:inexactopt:surrogates:theory}) relies on properties of~$A$ that hold with high probability for certain random matrices. Thus, in signal processing and, in particular, compressed sensing, one often encounters matrices such as Fourier transforms, Gaussian or Bernoulli matrices---sometimes combined with binary masks to blot out random entries, or otherwise modified. In contrast, we note that in other areas of application for the problems introduced in the following (or related tasks), the matrix~$A$ is often comprised of \emph{observational} data (e.g., in finance, regression or machine learning), which is typically unstructured and rarely beholden to specific probability distributions. This distinction may be partially responsible for the many different approaches found across disciplines.

\subsubsection{Sparse Recovery From Linear Measurements}\label{sec:concreteprobs:signalprocessing:canonical-linear}
The fundamental 
\emph{sparse recovery problem}
takes the form\footnote{The decision version of~\cardmin{$Ax=b$} and
  variants with other linear constraints than equality is also called
  \emph{minimum number of relevant variables in linear systems
    (MinRVLS)}~\cite{AK98} or \emph{minimum weight solution to linear
    equations}~\cite{GJ79}.}
\begin{equation*}
  \min~\norm{x}_0 \st Ax=b, \tag{\tagcardmin{$Ax=b$}}
\end{equation*}
where $A\in\R^{m\times n}$ with (w.l.o.g.) $\rank(A) = m < n$, and
$b\in\R^m$. Its variant allowing for noise in the linear measurements is
usually deemed more realistic (although real-world applications for the above noise-free setting do exist) and can be formulated~as
\begin{equation*}
  \min~\norm{x}_0 \st \norm{Ax-b}_2\leq\delta, \tag{\tagcardmin{$\norm{Ax-b}_2\leq\delta$}}
\end{equation*}
with some $\delta\in(0,\norm{b}_2)$ that is often derived from statistical
properties of the noise in applications. The assumption $\delta<\norm{b}_2$
excludes the otherwise trivial all-zero solution. Depending on noise
models and application contexts, the $\ell_2$-norm in the
constraint may be replaced by the $\ell_1$-norm (e.g., when the noise is
impulsive, cf.~\cite{FS09}), by the $\ell_\infty$-norm (in case of uniform
quantization noise or for sparse linear discriminant analysis,
cf.~\cite{BGL16,CL11}, respectively), or possibly by general
$\ell_p$-(quasi-)norms for some~$p>0$.

An alternative to cardinality minimization seeks to optimize data
fidelity within a prescribed sparsity level~$k\in\N$ of the signal vector
to be reconstructed, i.e., typically,
\begin{equation*}
  \min~\norm{Ax-b}_2 \st \norm{x}_0\leq k. \tag{\tagcardcons{$\norm{Ax-b}_2$}{$k$}{$\R^n$}}
\end{equation*}
This problem is often also referred to as \emph{subset selection} or
\emph{feature selection}, see, e.g.,~\cite{M02,BKM16}, and plays an
important role in many regression and machine learning tasks (see also
Section~\ref{sec:concreteprobs:statML}). Here, as for~\cardmin{$\norm{Ax-b}_2\leq\delta$}, the $\ell_2$-norm term is often rewritten
equivalently as $\tfrac{1}{2}\norm{Ax-b}_2^2$ to ensure differentiability (in $x$ with $Ax=b$) and simplify derivative notation; variants employing other norms also exist. The special case with ortho\-gonal~$A$ yields a sparse version of the standard \emph{denoising} problem, where one seeks to ``clean up'' a noisy version $b=x+e$ of the target signal~$x$ (in case $A=I$, cf.~\cite{EA06}), often incorporating an orthogonal basis transformation ($A\neq I$ but orthogonal, as in, e.g., \cite{DJ94a,DJ94b,M09}). Going beyond orthogonal bases, i.e., utilizing sparse representability w.r.t.\ more general~$A$---such as overcomplete dictionaries, see Section~\ref{sec:concreteprobs:signalprocessing:dictionary} below---can further improve denoising capabilities, e.g., in image processing,
see, for instance, \cite{EA06} and references therein.

By its respective definition, \cardmin{$\norm{Ax-b}_2\leq\delta$} requires (approximate) know\-ledge of the noise level~$\delta$, and for \cardcons{$\norm{Ax-b}_2$}{$k$}{$\R^n$}, the user must specify the allowed sparsity level~$k$. Since in practice it may be unclear how to choose either~$\delta$ or~$k$ appropriately, the regularization approach
\begin{equation*}
  \min~\norm{x}_0 + \tfrac{1}{2\lambda}\norm{Ax-b}_2^2 \tag{\tagcardreg{$\tfrac{1}{2\lambda}\norm{Ax-b}_2^2$}{$\R^n$}}  
\end{equation*}
has also been thoroughly investigated. Note that this problem is also particularly suitable to situations where the noise has limited variance (but its level is unknown), and a sparse solution (of unknown cardinality) is sought. Here, the regularization parameter
$\lambda>0$ controls the tradeoff between sparsity of the solution and data
fidelity.  While this model has the potential advantage of being
unconstrained, it is similarly unclear how to ``correctly'' choose
$\lambda$ in most applications. In general, there are many different
approaches to obtain regularization, sparsity or residual error-bound
parameters that work well for an application at hand, including homotopy
schemes and cross-validation techniques. 

A fundamental question from the signal processing perspective is that of
uniqueness of the recovery problem solution. In particular, for the basic
reconstruction problem \cardmin{$Ax=b$}, uniqueness can be characterized by
means of a matrix parameter called the \emph{spark} (see
\cite{GR97,DE03,T19}), which is defined as the smallest number of linearly
dependent columns, i.e.,
\begin{equation*}
    \spark(A)=\min\{\norm{x}_0 \suchthat Ax=0,~x\neq 0\}. \tag{\tagcardmin{$Ax=0,\,x\neq 0$}}
\end{equation*}
Indeed, all $k$-sparse signals $\hat{x}$ are respective unique optimal
solutions of \cardmin{$Ax=A\hat{x}$} if and only if $k<\spark(A)/2$, see
\cite{DE03,GN03} or \cite[Thm.~1.1]{EK12}. The spark is also known as
the \emph{girth} of the matroid defined over the column index set,
cf.~\cite{O92}, and it is also important in other fields, e.g., in the
context of tensor decompositions \cite{K77,KB09,Z17} (by relation to the
so-called ``Kruskal rank'' $\spark(A)-1$) or matrix
completion~\cite{ZMY12}. When working in the binary field $\F_2$, the spark
problem amounts to computing the minimum (Hamming) distance of a binary
linear code, which---along with the strongly related problem of
maximum-likelihood decoding---has been treated extensively in the coding
theory community, see, e.g., the structural and polyhedral results and LP
and MIP techniques discussed in
\cite{PKWTRH10,ZS12,FWK05,KD10,BG86,KTP18,PT20} and references therein.

Another connection to coding theory is found by relating the cardinality-minimi\-za\-tion 
problem  \cardmin{$Ax=b$} to an error-correction perspective in \emph{decoding} applications (see, e.g.,
\cite{CT05}): Suppose a message $y$ is encoded using a linear code $C$ with
full column-rank as $b\define Cy$, but a corrupted version
$\hat{b}\define b+\hat{e}$ is received. If the unknown transmission error
$\hat{e}$ is sufficiently sparse, recovering the true message $y$ can be
formulated as $\min_{x}\norm{\hat{b}-Cx}_0$. Using a left-nullspace
matrix~$B$ for~$C$, multiplying $\hat{b}=Cy+\hat{e}$ from the left by $B$
yields $B\hat{e}=B\hat{b}\eqqcolon d$. Now, the sparse error
vector~$\hat{e}$ can be obtained by solving \cardmin{$Bx=d$}, and once
$\hat{e}$ is known, it remains to solve $Cy=b+\hat{e}$ for $y$ (which is
trivial since $C$ has full column-rank) to recover the original message.

Finally, for all problems defined above, several variants with additional
constraints on the variables have been considered in the literature---in
particular, nonnegativity constraints ($x\geq 0$), more general variable
bounds ($\ell\leq x\leq u$ for $\ell,u\in\R^n\cup\{\pm\infty\}$ with
$\ell\leq 0\leq u$, $\ell<u$), or integrality constraints. The case of
complex-valued variables has also been investigated in compressed sensing
and sparse signal recovery problems; nevertheless, for simplicity, we stick
to the real-valued setting throughout this paper unless explicitly stating otherwise.

\subsubsection{Sparse Recovery From Nonlinear Measurements}\label{sec:concreteprobs:signalprocessing:canonical-nonlinear}
While compres\-sed sensing concentrates on reconstructing sparse signals from
\emph{linear} measurements, analogous tasks have also been investigated for
certain kinds of \emph{nonlinear} observations. In particular, the
classical optics problem of \emph{phase retrieval}~\cite{W63} has been
demonstrated to benefit from sparsity priors as well, see, e.g.,
\cite{MRB07,SBE14}. The (noise-free) sparse phase retrieval problem may be
stated as
\begin{equation*}
  \min~\norm{x}_0 \st \abs{Ax}=b, \tag{\tagcardmin{$\abs{Ax}=b$}}
\end{equation*}
where, generally, $A\in\C^{m\times n}$ (often a Fourier matrix) and $x$ is
also allowed to take on complex values; here, $\abs{Ax}$ denotes the
component-wise absolute value. Naturally, noise-aware variants exist for
this type of problem as well (and are arguably more realistic than the
idealized problem above), as do cardinality-constrained analogues; for
brevity, we do not list them explicitly.  Also, instead of the
``magnitude-only'' measurement model~$\abs{Ax}$, the squared-magnitude
$\abs{Ax}^2$ (again, evaluated component-wise) is often used. Typical
further constraints impose nonnegativity or a priori information on the
signal support, e.g., restricting the solution nonzeros to certain index
ranges. To achieve solution uniqueness up to a global phase factor in phase
retrieval, oversampling (i.e., $m>n$) is necessary in general.

It is worth mentioning that sparse phase retrieval using squared-magnitude
measurements can also be viewed as a special case of what has been termed
\emph{quadratic compressed sensing}~\cite{SESS11}, where the linear
measurements $Ax$ are replaced by quadratic ones $x^\top A_k x$,
$k=1,\dots,K$, with symmetric positive semi-definite matrices~$A_k$. The
most general form of cardinality minimization problem with a (single)
quadratic constraint can be stated~as
\begin{equation*}
  \min~\norm{x}_0 \st x^\top Qx + c^\top x\leq\varepsilon, \tag{\tagcardmin{$x^\top Qx + c^\top x\leq\varepsilon$}}
\end{equation*}
where $Q\in\R^{n\times n}$ is symmetric positive (semi-)definite, and
$\eps>0$. Extensions to multiple quadratic constraints as in quadratic
compressed sensing are conceivable as well. A problem of this type is
considered in the context of \emph{sparse filter design}~\cite{WO13a,WO13b},
namely
\[
  \min~\norm{x}_0 \st (x-b)^\top Q (x-b)\leq\varepsilon
\]
with a positive definite matrix $Q$. Note
that~\cardmin{$\norm{Ax-b}_2\leq\delta$} can also be rewritten in this
form:
\[
  \min~\norm{x}_0 \st x^\top A^\top A x - 2b^\top Ax\leq \delta^2-b^\top b.
\]
Here, however, $Q=A^\top A$ is rank-deficient (for $A\in\R^{m\times n}$
with $\rank(A)=m<n$), resulting in unboundedness of the feasible set in
certain directions.

A cardinality minimization problem of the form
\cardmin{$\norm{Ax-b}_2\leq\delta$,\,$\abs{x}\in\{0,1\}$,\,$x\in\C^n$} was
considered in~\cite{FHMPPT18}, combining nonconvex ``modulus'' constraints
and noise-aware linear measurement constraints. Various related approaches
to exploit the concept of sparsity in the context of
\emph{direction-or-arrival estimation}, \emph{sensor array} or
\emph{antenna design} have also been investigated, see, e.g.,
\cite{SPP14,HLM15,ZLG17,HA21}; however, here, the true cardinality is
typically replaced by an $\ell_1$-norm surrogate
(cf.\ Section~\ref{sec:inexactopt:surrogates}), and group sparsity models
(cf.\ Section~\ref{sec:concreteproblems:miscellaneous:groupsparsity}) may be
used instead of standard vector sparsity.

\subsubsection{Generalized Sparsity Models}\label{sec:concreteprobs:signalprocessing:dictionary}
In the problems considered thus far, the vector~$x$ is assumed to be sparse itself, or to be well approximated by a sparse one. While this basic sparsity model proved adequate and was successfully utilized in numerous examples, in different practical applications, a more general approach is called for, as the signal~$x$ may not be (approximately) sparse directly. Thus, it often makes sense to admit \emph{sparse representations} with respect to a given matrix~$D$ (called \emph{dictionary}), i.e., $x\approx Ds$ with a sparse coefficient vector~$s$. Sometimes, taking $D$ as a certain basis matrix
(e.g., a discrete cosine transform or wavelet basis) can already work quite well, and generally,
overcompleteness in the dictionary---i.e., having more columns than rows---allows for even sparser representations and further applications. For instance, loosely related to the decoding problem outlined earlier, \cite{WMMSHY10} considers face recognition by identifying a new (vectorized) image~$x$ as a sparse linear combination of elements from a large dictionary~$D$ of partially occluded/corrupted images taken under varying illumination, which can be modeled as
\[
  \min~\norm{s}_0 + \norm{e}_0 \st x = Ds + e,
\]
where $e$ is an error vector. (This problem generalizes to the ``robust PCA'' problem of decomposing a matrix into a sparse and a low-rank part, see, e.g.,~\cite{CLMW11}.) Moreover, importantly, a suitable dictionary can be \emph{learned} from data to enhance representability for certain signal or image classes, 
see Section~\ref{sec:concreteprobs:statML}. Thus, in principle, for a
(fixed) dictionary~$D$, one can replace $Ax$ by $ADs$ and $\norm{x}_0$ by
$\norm{s}_0$ in all of the problems from Sections~\ref{sec:concreteprobs:signalprocessing:canonical-linear} and~\ref{sec:concreteprobs:signalprocessing:canonical-nonlinear}.

The above approach is sometimes called the \emph{synthesis} sparsity model,
since the signal $x$ is ``synthesized'' from a few columns of $D$. The
alternative \emph{cosparsity} (or \emph{analysis}) model instead presumes
that $Bx$ is sparse for some matrix $B\in\R^{p\times n}$ with $p> n$, see,
e.g., \cite{EMR07,NDEG13,KR13,SWGK16,D16}. Thus, the respective
analysis-variants of the models discussed earlier can be obtained by simply
replacing $\norm{x}_0$ by $\norm{Bx}_0$; the measurement part (e.g., $Ax=b$
or $\norm{Ax-b}_p\leq\delta$) remains unchanged. Clearly, this constitutes
an immediate generalization of the respective synthesis-variant---note that
the two variants become equivalent when~$B$ is a basis, since then, one can
substitute~$x$ by~$B^{-1}x$ throughout the respective problem and arrive
back at the synthesis model form---and hence offers some more flexibility.

The cosparsity viewpoint has been employed, for instance, in discrete
tomo\-graphy (see, e.g., \cite{DPSS14} for a cosparsity minimization
problem with linear projection equations and box constraints) and
\emph{image segmentation} (see, e.g., \cite{SWFU15} treating a so-called
discretized \emph{Potts model} or \cite{J13,BKLMW09} for one-dimensional ``jump-penalized'' least-squares segmentation, both of which amount to minimization of an $\ell_2$-norm data fidelity term with cosparsity-regularization), where $B$ is taken as a discrete gradient or finite-differences operator. Further applications
include, for example, audio denoising, see, e.g., \cite{GKBG17}.

\subsection{Portfolio Optimization and Management}\label{sec:concreteprobs:finance}
Quadratic programs (QPs) with cardinality constraints rather than a
cardinality objective play a crucial role in financial applications, in
particular, portfolio optimization, see, e.g.,
\cite{B96,CST13,GL13a,MB14,BCW18}.  Broadly speaking, in \emph{portfolio
  selection} (or \emph{portfolio management}), one seeks to find (or
update, resp.) a low-risk/high-return composition of assets from a given
universe, e.g., the constituents of a stock-market index like the
S\&P500. Here, cardinality constraints serve the purpose of reducing the cost and complexity of management of the 
resulting portfolio. These problems are usually formulated in the general
form
\begin{equation}
  \min~x^\top Qx - c^\top x \st Ax\leq b,~\norm{x}_0\leq k,
  \tag{\tagcardcons{$x^\top Qx - c^\top x$}{$k$}{$Ax\leq b$}}
\end{equation}
where the symmetric positive (semi-)definite matrix $Q\in\R^{n\times n}$ is
the (possibly scaled) covariance matrix of the assets and $c\in\R^n$ is the
vector of expected returns. If the focus is on achieving a low risk
(volatility) profile, the return-maximization term $-c^\top x$ is sometimes
replaced by a minimum-return constraint $c^\top x\geq \rho$. Similarly, the
risk term $x^\top Q x$ can be replaced by a maximum-risk constraint
$x^\top Qx \leq r$. The system $Ax\leq b$ subsumes commonly encountered
variables bounds $\ell\leq x\leq u$ (in particular, $x\geq 0$ prohibits
short-selling) as well as further constraints such as $\ones^\top x=1$
(when, as is usual, $x_i\geq 0$ represents allocation percentages) or
minimum-investment constraints\footnote{Minimum-investment constraints have
  the form $x_i\in\{0\}\cup[\ell,u]$, and so are not, technically, linear
  constraints. The associated variables are often referred to as
  semi-continuous (see, e.g., \cite{SZL13}). With standard modeling
  techniques to formalize the cardinality constraints (see, e.g.,
  \cite{B96,BCW18}), however, they can be linearized; e.g., using
  $z_i\in\{0,1\}$ with $z_i=0\Rightarrow x_i=0$, a minimum-investment
  constraint simply becomes $\ell z_i\leq x_i\leq u z_i$; see also
  Section~\ref{sec:exactopt:cardcons:example}.} (e.g., to prevent positions
that incur more transaction fees than they are expected to earn
back). There are also portfolio selection problems with linear objectives,
see, e.g., the summary provided in~\cite{CMBS00}.

Since $Q$ is symmetric positive semidefinite, the above problem is convex
except for the cardinality constraint. Variants of these kinds of models
have been considered that include a further quadratic regularization term
$\lambda\norm{x}_2^2$ in the objective and/or diagonal-matrix extraction
(i.e., separating $Q$ into a positive semidefinite and a diagonal part) as
a kind of preprocessing step; see~\cite{BCW18} for a recent overview.

Moreover, note that~\cardcons{$\norm{Ax-b}_2$}{$k$}{$\R^n$} is a special
case of the above ge\-neral problem, as it can be rewritten~as
\[
  \min~x^\top A^\top Ax-2b^\top A x \st \norm{x}_0\leq k.
\]
However, the matrix $Q\define A^\top A$ is again rank-deficient here for
the matrices~$A$ usually considered in sparse recovery applications. In
fact, by exploiting the fact that a symmetric positive semidefinite
rank-$r$ matrix $Q\in\R^{n\times n}$ can be decomposed as $Q=S^\top S$ with
some $S\in\R^{r\times n}$ (think Cholesky factorization), \cite{BCW18} show
that \cardcons{$x^\top Qx + c^\top x$}{$k$}{$Ax\leq b$} can conversely be
rewritten to resemble
\cardcons{$\tfrac{1}{2}\norm{Ax-b}^2_2$}{$k$}{$\R^n$}, albeit with an
additional linear term in the objective and retaining the other (linear)
constraints.

It is worth mentioning that, in a spirit similar to sparse PCA (see the
next subsection for a definition), the covariance matrix $Q$ in real-world
portfolio selection problems is sometimes replaced by a low-rank estimate,
e.g., from truncating the singular-value decomposition of the $Q$ obtained
with the data, cf.~\cite{BCW18,ZSL14}.

\subsection{High-Dimensional Statistics and Machine Learning}\label{sec:concreteprobs:statML}
Cardinality aspects also play an important role in various applications in machine learning and data science; for clarity, we break down the following discussion into topical subsections.

\subsubsection{Sparse Regression, Feature Selection, and Principle Component Analysis}\label{sec:concreteprobs:statML:regr_fs_pca}
The problems \cardmin{$\norm{Ax-b}_2\leq\delta$} or \cardcons{$\norm{Ax-b}_2$}{$k$}{$\R^n$} are often referred to as \emph{sparse regression}, being cardinality-considerate versions of classical linear regression (ordinary least-squares).
Another problem from statistical estimation that is related to
\cardmin{$\norm{Ax-b}_2\leq\delta$} seeks to find sparse regressors
with a constraint on the maximal absolute correlation between predictors
and the corresponding residual; this can be formulated as the so-called
\emph{discrete Dantzig selector}~\cite{MR17}:
\begin{equation*}
  \min~\norm{x}_0 \st \norm{A^\top(Ax-b)}_\infty\leq\delta. \tag{\tagcardmin{$\norm{A^\top(Ax-b)}_\infty\leq\delta$}}  
\end{equation*}

As mentioned earlier
(cf.\ Section~\ref{sec:concreteprobs:signalprocessing}), the problem
\cardcons{$\norm{Ax-b}_2$}{$k$}{$\R^n$} is also known as \emph{subset selection}
or \emph{feature selection}, see, e.g.,~\cite{M02,BKM16}. Beyond sparse
regression, feature selection is, in fact, a vital part of various machine learning
problems: Wherever a model of some kind is to be trained to perform
inference/prediction tasks, from simple regression to complex neural
networks, the (input) features are typically selected manually and can be
numerous. Thus, integrating a sparsity component to automatically detect
relevant features has become a staple in reducing the computational burden
and sharpen model interpretability; see also
Section~\ref{sec:concreteprobs:statML:classification} below.

Furthermore, QPs with cardinality constraints are not only important in finance (cf.~Section~\ref{sec:concreteprobs:finance}), but are also encountered in feature extraction methods. In particular, the well-known \emph{sparse principal component
  analysis (PCA)} problem (see, e.g., \cite{ZHT06,dABEG08,LT13,DMW18,BB19})
is usually defined as
\begin{equation*}
  \max~x^\top Qx \st \norm{x}_2=1,~\norm{x}_0\leq k. \tag{\tagcardcons{$-x^\top Qx$}{$k$}{$x^\top x=1$}}
\end{equation*}
Clearly, sparse PCA is related to \cardcons{$x^\top Qx + c^\top x$}{$k$}{$Ax\leq b$}, albeit with nonconvex objective---note that earlier, we discussed a minimization problem, but in sparse PCA, we maximize a quadratic term. Also, here, the quadratic equation $\norm{x}_2=1$ $\Leftrightarrow$
$\norm{x}_2^2=1$ introduces further nonconvexity, but may, in fact, be
relaxed to its convex counterpart $\norm{x}_2\leq 1$ in an equivalent
reformulation, see~\cite[Lemma~1]{LT13}. Genera\-lizing
the constraint $\norm{x}_2=1$ to $x^\top Bx=1$ with a symmetric positive
semidefinite matrix $B$, one obtains the \emph{sparse linear discriminant analysis (LDA)} problem, see, e.g., \cite{MWA06}. The sparse PCA problem is also taken up in \cite{LX20}, which presents mixed-integer SDP formulations and an approximate mixed-integer LP formulation, compare their strength to other formulations and analyze their theoretical and practical performance. Similarly, \cite{DMW22} considers the interesting related problem of sparse PCA with global support. Here the goal is, given an $n\times n$ covariance matrix $A$, to compute an $n\times r$ matrix $V$ (with $r$ typically much smaller than $n$) with orthonormal columns, so as to maximize $\mathrm{trace}(V^\top A V)$, but subject to $V$ having at most $k$ nonzero rows. The $r$ columns of $V$ can thus be viewed as a set of $k$-sparse principle components of~$A$ with common global support. 

\subsubsection{Classification}\label{sec:concreteprobs:statML:classification}
Cardinality constraints have also been employed in other machine learning
tasks, and are often introduced to improve interpretability of
learned classification or prediction models. We already mentioned the
feature selection problem \cardcons{$\norm{Ax-b}_2$}{$k$}{$\R^n$}. Another
example is the sparse version of \emph{support vector machines (SVMs)} for
(binary) classification, which can be stated as\footnote{Note that we
  slightly abuse notation by referring to the sparse SVM problem class as
  \cardcons{$L(w,b)$}{$k$}{$(w,b)\in\R^{n+1}$}, since the cardinality
  constraint only involves~$w$ but not~$b$. Never\-theless, clearly, there
  is no requirement that a cardinality constraint involves \emph{all}
  variables of a problem under consideration, although this is typically
  the case in the problems we discuss here.}
\begin{equation*}
  \min~L(w,b) \st \norm{w}_0\leq k, \tag{\tagcardcons{$L(w,b)$}{$k$}{$(w,b)\in\R^{n+1}$}}  
\end{equation*}
where
$L(w,b)\define\sum_{i=1}^{m}\ell(y_i,w^\top
x_i+b)+\tfrac{1}{2\lambda}\norm{w}_2^2$. Here, $\ell$ is one of several
possible convex empirical loss functions (w.r.t.\ input data points
$x_i\in\R^n$ with associated labels $y_i\in\{-1,1\}$) that is minimized by
training the classifier hyperplane $w^\top x+b=0$. Similarly to the
portfolio selection problem treated in \cite{BCW18}, an optional
regularization term $\tfrac{1}{2\lambda}\norm{w}_2^2$ with
$\lambda>0$---called \emph{ridge} or \emph{Tikhonov} penalty---can be used
to ensure strong convexity and thus existence of a unique optimal solution,
see, e.g.,~\cite{BPVP17}.

The idea of ``interpretability by sparsity'' can also be found in recent approaches to train oblique decision trees for (multi-class) classification. While standard decision trees split data inputs at tree nodes according to a single feature (e.g., follow the left branch if $x_i\leq b$, and the right branch otherwise, with tree leaves yielding the predicted class for the input feature vector $x$), more powerful splits use hyperplanes $(a^j)^\top x=b^j$ whose coefficients $(a^j,b^j)$ are obtained via training the model. At least for small tree-depths, one can compute optimal decision trees (in the sense of classification accuracy w.r.t.\ the chosen
task and training/testing data sets) with mixed-integer programming, cf.,
e.g., \cite{BD17}. To retain the clear interpretability of univariate
splits, one can restrict the cardinality of the vectors $a^j$ used at split
nodes $j$ of the classification tree being learned, so each path through the tree represents a series of decisions based on a few features each.

\subsubsection{Dictionary Learning}\label{sec:concreteprobs:statML:DL}
In connection with sparse coding in signal and, in particular, image
processing, \emph{dictionary learning (DL)} problems have received
considerable attention over the past years. Indeed, the observation that
certain signal classes admit sparse approximate representations w.r.t.\
some basis or overcomplete ``dictionary'' matrix (see, e.g.,
\cite{OF96,EA06,MBP14}) was an important motivation for the intense
research on sparse reco\-very techniques. Following this understanding that
signals are not necessarily sparse themselves but may be sparsely
approximated w.r.t.\ a dictionary $D$ (i.e., not $x$ is sparse but
$x\approx Ds$ with $\norm{s}_0$ small), it was soon realized that while
some fixed dictionaries may work reasonably well, better results can be
achieved by adapting the dictionary to the data. Thus, the goal of DL is to
train suitable dictionaries on the datasets of interest for a concrete
application at hand. Example applications include image denoising and
inpainting (see, e.g., \cite{EA06,MES08,MBP14}) or simultaneous dictionary
learning and signal reconstruction from noisy linear or nonlinear
measurements (see, e.g., \cite{MBP14,TEM16}). Possible basic formulations
of the task to algorithmically learn suitable matrices on the basis of
large collections of training signals are
\[
  \min_{\{s^t\},D}~\tfrac{1}{2}\sum_{t=1}^T\norm{x^t-Ds^t}_2^2 + \lambda\sum_{t=1}^T\norm{s^t}_0
\]
or
\[
  \min_{\{s^t\},D}~\sum_{t=1}^T\norm{s^t}_0 \st \norm{x^t-Ds^t}_2 \leq \delta~~\forall t,
\]
usually additionally constraining the columns of $D$ to be unit-norm in
order to avoid scaling ambiguities. Here, all training signals $x^t$,
$t=1,\dots,T$, are sparsely encoded as $Ds^t$ w.r.t.\ the same dictionary
$D$. Unsurprisingly, dictionary learning is also \NP-hard in general (and hard
to approximate)~\cite{T15}, and no general-purpose exact solution methods are
known. Instead, algorithms are typically of a greedy nature or employ
alternating minimization/block coordinate descent, iteratively solving ea\-sier subproblems obtained by fixing all but one group of variables, see,
for instance, \cite{OF96,AEB06,MBPG10}. In particular, many such schemes
involve classical sparse recovery problems like,
e.g.,~\cardmin{$\norm{Ax-b}_2\leq\delta$} or
\cardcons{$\norm{Ax-b}_2$}{$k$}{$X$} as frequent subproblems, so any
progress regarding solvability of those problems can also directly impact
many dictionary learning algorithms. Such DL schemes work reasonably well
in practice, and may even be extended to simultaneously learn a dictionary
for sparse coding and reconstructing the sparse signals, from linear or
nonlinear (noisy) measurements, see, e.g., \cite{TEM16}. Also, note that,
as in compressed sensing and especially
for~\cardmin{$\norm{Ax-b}_2\leq\delta$} and similar problems, the
$\ell_0$-norm is often replaced by its $\ell_1$-surrogate,
cf.~Section~\ref{sec:inexactopt:surrogates}. However, apart
from occasional results demonstra\-ting convergence to stationary points of the typically nonconvex DL models, hardly any success guarantees are known for such methods in general. 

For special cases, researchers have nevertheless considered the question of \emph{dictionary identifiability}, i.e., whether the true underlying dictionary~$D$ can be uniquely reconstructed (up to trivial sign, scale and permutation ambiguities) from measurements $B=DX$ along with sparse signals forming the columns of~$X$. Thus far, results are relatively scarce and mostly yield probabilistic guarantees (typically for certain algorithms) under arguably strong assumptions on the dictionary~$D$ and/or assuming support locations and entry values of~$X$ follow some probability distributions. For instance, \cite{SWW12,SQW17a,SQW17b} investigate the case in which~$D$ is a basis (square, invertible matrix) and measurements are noiseless, \cite{AGM14,AAJNT14,S18} consider noisy measurements and overcomplete but incoherent dictionaries, \cite{BKS15} does so without incoherence requirements, and \cite{AV18} treats the noise-free case with overcomplete~$D$ and a less restrictive ``semi-random'' model for the supports of~$X$. %The algorithmic techniques employed in these works include iterative thresholding, $\ell_1$-minimization, alternating minimization, K-means and Riemannian trust region schemes. 
In~\cite{RMN19}, success guarantees and error bounds are derived for the case of unitary bases~$D$ and~$X$ with certain spectral bound properties that hold with high probability under common probability distribution models for its support/entires. The paper \cite{STW18} relates DL to the geometrical notion of combinatorial rigidity of subspace incidence systems and provides a classification of several DL guarantees from this viewpoint, along with some new identifiability results. Deterministic recovery conditions are even less common; an early example is \cite{AEB06b}, which establishes non-probabilistic identifiability at the cost of potentially exponential sample complexity. More recently, \cite{BT19} avoids probabilistic arguments as well as the inherent intractability of DL and, assuming only a certain norm bound, shows that~$D$ and~$X$ can be approximated up to bounded small violations of the presumed number of dictionary columns and sparsity level of those in~$X$.

\subsubsection{Rank Minimization and Low-Rank Matrix Completion}\label{sec:concreteprobs:statML:rankmin}
A problem related to DL that, in fact, generalizes \cardmin{$Ax=b$}, is the
\emph{affine rank minimization} problem
$\min\{\rank(X) : \mathcal{A}(X)=b\}$, where $\mathcal{A}$ is a linear
map. Clearly, if~$X$ is further constrained to be diagonal, the problem
reduces to finding the sparsest vector in an affine subspace, i.e.,
\cardmin{$Ax=b$}. We refer to \cite{RFP10} for interesting theoretical
analyses of this problem and references to various applications from system
identification and control to collaborative filtering. Another
sparsity-related problem that received attention to due its successful
application to the ``Netflix problem''---in a nutshell, obtaining good
predictions for recommendation systems based on limited (user
rating/preference) observations---is that of \emph{low-rank matrix
  completion}. Here, the most basic problem seeks a matrix
$B\in\R^{m\times n}$ that approximates a given matrix $A\in\R^{m\times n}$
as well as possible under a \emph{rank constraint} $\rank(B)\leq k$. Rank
constraints for matrices are related to cardinality constraints for
vectors; indeed, if $A\in\R^{m\times n}$ has rank~$k$, this means that
only~$k$ of its $\min\{m,n\}$ singular values are nonzero. Thus, using the
singular-value decomposition $A=U\Sigma V^\top$, a rank constraint on $A$
can be expressed as $A=U\Sigma V^\top$, $\norm{\diag(\Sigma)}_0\leq k$.  In
matrix completion, the usual objective is $\min\norm{B-A}_\text{F}$, whence
it is clearly always optimal to have $B$ of rank equal to $k$ (provided
$\rank(A)\geq k$); then, it is common practice to directly split $B$ as
$B=LR$ with $L\in\R^{m\times k}$ and $R\in\R^{k\times n}$ and handle the
rank constraint implicitly by construction. However, the problem has also been viewed as rank minimization under linear constraints, for which the rank can then be modeled semialgebraically, which gives rise to a semidefinite relaxation that is exact under certain conditions, see~\cite{dA03}. \emph{Inductive} or
\emph{interpretable} matrix completion aims at enhancing interpretability
of the reasons for recommendations by substituting $R$ (or $L$, or both) by
$R=ST$ with a known ``feature matrix'' $S$, so that linear combinations of
these features yield $R$, and then enforcing the rank constraint by restricting the cardinality of the
coefficient vectors of these linear combinations to some~$k$, or
restricting the selection of features to~$k$, respectively; see~\cite{BL20}
and references therein.

\subsubsection{Clustering}\label{sec:concreteprobs:statML:clustering}
Finally, it is worth mentioning that the term ``cardinality constraint'' is
also sometimes used with a slightly different meaning. A particular example
is \emph{cardinality-constrained ($k$-means) clustering}, where one seeks
to partition a set of data-points into $k$ clusters, minimizing the
inter-cluster Euclidean distances (to the cluster center). Here, one could
restrict the number of clusters to be considered by an upper cardinality
bound~$k$; however, it is trivially optimal to always use the maximal
possible number of clusters. Then, one can in fact directly incorporate the
knowledge that one will have $k$ clusters into the problem formulation in
other ways (see, e.g., \cite{RSKW19}), similarly to the rank constraint in
matrix completion we saw above. The cardinality of the clusters themselves
may then also be restricted (e.g., to balance the partition to clusters of
equal sizes), which ultimately yields cardinality equalities w.r.t.\ sets
of cluster-assignment variables; as these are typically binary variables,
this type of cardinality constraint is again different from our focus here
(cf.\ beginning of Section~\ref{sec:intro}).

 \subsection{Miscellaneous Related Problems and Extensions}\label{sec:concreteproblems:miscellaneous}
The various classes of cardinality optimization problems discussed up to
now can be generalized and extended in different directions. In this
section, we briefly point out some of these connections.

\subsubsection{``Classical'' Combinatorial Optimization Problems}\label{sec:concreteproblems:miscellaneous:combopt}
As mentioned earlier, $\spark(A)$ corresponds to the girth of the vector
matroid $\M(A)$ defined over the column subsets of $A$; cf.~\cite{O92} for
details on matroid theory and terminology. Thus, $\spark(A)$ is a special
case of the more general problem
\begin{equation*}%\label{girth}
  \girth(\M)\define\min\{\norm{x}_0 \suchthat x=\chi_C\text{ for a circuit $C$ of matroid $\M$}\},
\end{equation*}
where $(\chi_C)_j=1$ if $j\in C$ and zero otherwise (i.e., $\chi_C$ is the
characteristic vector of $C$). Moreover, recall that, when considered over
the binary field~$\F_2$, the spark problem coincides with the problem of
determining the minimum distance of a binary linear code,
cf. Section~\ref{sec:concreteprobs:signalprocessing:canonical-linear} and the references
given there; this amounts to the binary-matroid girth problem. The girth of
a matroid~$\M$ equals the cogirth of the associated dual
matroid~$\M^*$. Moreover, cocircuits of~$\M$ (i.e., circuits of~$\M^*$)
correspond exactly to the complements of hyperplanes of~$\M$,~so
\begin{equation*}
  \cogirth(\M)\define\min\{\norm{x}_0 \suchthat x=\chi_{H^c}\text{ for a hyperplane $H$ of matroid $\M$}\},
\end{equation*}
where $H^c$ is the complement of $H$ w.r.t.\ the matroid's ground set. Note
that in the case of vector matroids, the cogirth is known as \emph{cospark}
(cf. \cite{CT05}) and can be written~as
\begin{equation*}
  \cospark(A)\define\min\{\norm{Ax}_0 \suchthat x\neq 0\};
\end{equation*}
similarly to the spark, it appears in recovery and uniqueness conditions
for analysis signal models and decoding, see, e.g., \cite{CT05,NDEG13}. The
spark and cospark can thus also be interpreted as dual problems, since
$\spark(A)=\cospark(B)$ for any~$B$ whose columns span the nullspace of
$A$.

In fact, $\cospark(B)$ constitutes a special case of the more general
\emph{minimum number of unsatisfied linear relations} (MinULR) problem,
where for $B\in\R^{p\times q}$ and $b\in\R^p$, one seeks to minimize the
number of violated relations in an infeasible system $Bz\sim b$, with
$\sim\;\in\{=,\geq,>,\neq\}^p$ representing all sorts of linear relations;
see \cite{AK95,AK98}. MinULR is also known by the name \emph{minimum
  irreducible infeasible subsystem cover (Min\-IISC)}, and is a
well-investigated combinatorial problem; the same holds for its
complementary problem \emph{maximum feasible subsystem} (MaxFS), which
seeks to find a cardinality-maximal feasible subsystem of $Bz\sim b$,
cf.~\cite{AK95,APT03,P08}. Problems like MinULR play an important role in
infeasibility analysis of linear systems, e.g., when analyzing demand
satisfiability in gas transportation networks \cite{JP18,J15}. Note that
for the inhomogeneous equation $Bz=b$, MinULR can be rephrased via
\[
  \min~\norm{Bz-b}_0\quad\Leftrightarrow\quad \min\{\norm{x}_0 \suchthat x-Bz=b\},
\]
and thus can be seen as a \emph{weighted} version of~\cardmin{$Ax=b$}, with
weights zero for the $z$-variables in the objective. Conversely,
\cardmin{$Ax=b$} can also be rephrased as a special case of MaxFS (or
MinULR, of course), see, e.g., \cite{JP08}. Using a diagonal (and thus,
effectively, binary) weight matrix within the $\ell_0$-term obviously
yields special cases of the analysis formulations that generalize
$\norm{x}_0$ to $\norm{Bx}_0$ for some matrix~$B$. To the best of our knowledge, it has not yet been investigated if and how results on (or involving) the spark from the signal processing context might aid the solution of discrete optimization problems by means of the connections laid out above or by exploiting ``hidden'' spark-like subproblems such as, e.g., in Proposition~\ref{prop:1stgreedyMSsubprob} below.

Finally, countless problems from combinatorial optimization and operations
research applications seek to minimize (or restrict) ``the number of
something'', which is naturally formulated as cardinality minimization
w.r.t.\ non-auxiliary (structural) binary decision vectors under broad
general or highly problem-specific constraints. Classical examples are the
standard packing/partitioning/covering problems
\begin{equation*}
  \min~\ones^\top y \st Ay\sim\ones,~y\in\{0,1\}^n,
\end{equation*}
with $\sim\;\in\{\leq,=,\geq\}$, respectively, and $A\in\{0,1\}^{m\times n}$;
see, e.g., \cite{B98} and textbooks like \cite{KV12}. We do not delve into
these kinds of problems here, since our focus is on handling the
cardinality of \emph{continuous} variable vectors.

\subsubsection{Matrix Sparsification and Sparse Nullspace Bases}\label{sec:concreteproblems:miscellaneous:sparsification}
Another related combinatorial optimization problem essentially extends the
idea of sparse representations from vectors to matrices: For a given matrix
$A\in\R^{m\times n}$ (w.l.o.g.\ with $\rank(A)=m< n$), the \emph{Matrix
  Sparsification} problem is given by
\begin{equation}\label{MS}
  \min~\norm{VA}_0\st \rank(V)=m, ~V\in\R^{m\times m},\tag{MS}
\end{equation}
where $\norm{M}_0=\abs{\{(i,j)\suchthat M_{ij}\neq 0\}}$ counts the
nonzeros of a matrix $M$, extending the common ``$\ell_0$-norm'' from
vectors to matrices. The problem is polynomially equivalent to that of
finding a sparsest basis for the nullspace of a given matrix, by arguments
similar to the aforementioned ``duality'' relation between spark and
cospark. This \emph{Sparsest Nullspace Basis} problem can be formally
stated as
\begin{equation}\label{SNB}
  \min~\norm{B}_0\st AB=0, ~\rank(B)=n-m, ~B\in\R^{n\times (n-m)}\tag{SNB}
\end{equation}
(recall that for an $m\times n$ matrix $A$ with full row-rank $m$, the
nullspace has dimension $n-m$). MS and SNB have been studied quite
extensively from the combinatorial optimization perspective, see, e.g.,
\cite{McC83,HMcC84,CP84,GH87,EM98}; \cite{GN16} provides a nice overview of
their equivalence relation and associated complexity results, and also
establishes some connections to compressed sensing.

Exact solution of the problems MS and SNB, as well as certain approximate
versions, are known to be \NP-hard\ tasks, see
\cite{McC83,CP84,GN16,T13,T15}. Connections to matroid theory reveal an
optimal greedy method for matrix sparsification that sparsifies a given $A$
by solving a sequence of $m$ subproblems; the scheme can be described
compactly as follows (cf.~\cite{EM98,GN16}): \medskip
\begin{enumerate}
   \item Initialize $V=[\,]$ (empty matrix).
   \item For $k=1,\dots,m$, find a $v^k\in\R^m$ that is linearly
     independent of the rows of~$V$ and minimizes $\norm{v^\top A}_0$, and update $V\define (V^\top,v^k)^\top$.
\end{enumerate}
\medskip
The final $V$ minimizes $\norm{VA}_0$ and has full rank $m$, i.e.,
is indeed a solution of MS.

It turns out that the first of the above subproblems amounts exactly to a
\emph{spark} computation, i.e., a problem of the
form~\cardmin{$Ax=0,\,x\neq 0$}:
\begin{prop}\label{prop:1stgreedyMSsubprob}
  The first subproblem in the above greedy MS algorithm can be solved as a spark problem.
\end{prop}
\begin{proof}
  The first subproblem can be written as
  $\min\{\norm{A^\top v}_0\,:\,v\neq 0\}$, which we recognize as
  $\cospark(A^\top)$. In light of the earlier discussion, this is
  polynomially equivalent to $\spark(D)$, where $D\in\R^{(n-m)\times n}$ is
  such that $A^\top$ is a basis for its nullspace
  (cf. \cite[Lemma~3.1]{T19}). In particular, a solution $\bar{v}$ to
  $\min\{\norm{A^\top v}_0\suchthat v\neq 0\}$ can be retrieved from a
  solution $\bar{x}$ to $\spark(D)$ as the unique solution to
  $A^\top\bar{v}=\bar{x}$, i.e., $\bar{v}=\big(AA^\top\big)^{-1}A\bar{x}$
  (recall that $A^\top\in\R^{n\times m}$ with full column-rank $m<n$).
\end{proof}

In~\cite{GN16}, the authors show that the $m$ subproblems of the greedy MS
algorithm can, in principle, each be solved by means of sequences of $n$
problems of the form $\min~\norm{Bz - b}_0$, i.e., MinULR w.r.t.\ $Bz=b$.
An ongoing work by the present first author pursues a different strategy,
aiming at leveraging the relation to the spark problem without resorting to
breaking down each subproblem of the greedy scheme into many further
subproblems (which are still \NP-hard, too). So far, the literature
apparently only describes a handful of (combinatorial) heuristics for MS or
SNB, see \cite{McC83,CP84,BHKLPW85,CMcC92} and some further references
gathered in~\cite{T13}.

Finally, it is interesting to note that matrix sparsification can also be
interpreted as a special dictionary learning task: The columns of the given
matrix correspond to the ``training signals'' and~$V^{-1}$ to the sought
dictionary that enables sparse representations. Two crucial differences to
the usual applications of DL are that MS requires~$V$ to be a basis (rather
than the common overcomplete dictionary) and the stricter accuracy
requirements w.r.t.\ the obtained sparse representations (i.e., $\delta=0$,
whereas in signal/image processing, one is typically satisfied with, or
even desires, $Ax\approx b$ only). To the best of our knowledge, the
relationship between MS and DL has not yet been explored in either
direction.

\subsubsection{Group-/Block-Sparsity}\label{sec:concreteproblems:miscellaneous:groupsparsity}
Another extension of the sparsity concept in signal processing and learning
leads to \emph{group- (or block-)sparsity} models: Here, the prior is not
sparsity of the full variable vector, but sparsity w.r.t.\ groups of
variables, i.e., whole blocks of variables are simultaneously treated as
``off'' (zero) or ``on'' (all group members are nonzero; may also mean that
at least one member is nonzero). This perspective can be useful in many
signal processing applications like simultaneous sparse approximation or
multi-task compressed sensing/learning (e.g., \cite{TGS06,SPH09,EKB10}),
dictionary learning for image restoration (e.g., \cite{ZZG14}),
neurological imaging or bioinformatics (e.g., \cite{PVMH08}), and may offer
additional interpretability due to identification of the respective active
groups. For instance, in a feature selection context, one may have several
(disjoint or overlapping) groups of related features along with knowledge
that features within a group are either all irrelevant or all have combined
explanatory value together. A typical formulation would then read, e.g.,
\[
  \min~\norm{Ax-b}_2 \st \supp(x)\subseteq
  \bigcup_{G\in\mathcal{S}}G,~\mathcal{S}\subseteq
  \mathcal{G},~\abs{\mathcal{S}}\leq k,
\]
where $\mathcal{G}$ is a known group structure (collection of index
subsets). The group-cardi\-nality constraint is represented by
$\abs{\mathcal{S}}\leq k$ here, ensuring that the computed solution~$x$ has
support restricted to the union of a selection $\mathcal{S}$ of at most~$k$
groups. In \cite{DAP19}, the extension of cardinality constraints to group sparsity is introduced via the concept of \emph{affine sparsity constraints} (ASC), and structural properties of systems of ASCs are studied. For more details and practical application references, intractability results and relaxation properties, we refer to \cite{BBCKS16,HZ10,BH19} and references therein. 

\section{Exact Models and Solution Methods}\label{sec:exactopt}
The cardinality problems descri\-bed in the previous section are all
\NP-hard in general, and are often also very hard to solve
approximately. On the one hand, samples of such intractability results cover, in particular, \cardmin{$Ax=b$} \cite{GJ79,AK95,AK98,T15},
\cardmin{$\norm{Ax-b}_2\leq\delta$} \cite{N95}, cardinality-constrained QPs
\cite{B96}, sparse PCA \cite{TP14},
\cardreg{$\tfrac{1}{2\lambda}\norm{Ax-b}_2^2$}{$x\in\{\R^n,\R^n_+\}$}
\cite{NSID19}, and genera\-lized variants (with other norms or
sparsity-inducing penalty functions) of some such problems \cite{CYW19}, as
well as related problems such as matroid (co-)girth and (co-)spark
\cite{K95,V97,TP14,T19}, MinULR/MaxFS \cite{AK95,AK98}, and matrix
sparsification \cite{McC83,T13,GN16}. On the other hand, there are a few examples of polynomially solvable special cases in the literature that involve certain sparsity patterns or combinatorial properties of the matrix~$A$, see, e.g.,~\cite{DPDW20} for \tagcardcons{$\norm{Ax-b}_2$}{$k$}{$\R^n$} and~\cite{GI10,T19} for compressed sensing sparse recovery.

Thus, polynomial-time exact solution algorithms generally cannot exist
unless \P$=$\NP, which justifies the extensive efforts to devise
practically efficient approximate (heuristic) methods; see
Section~\ref{sec:inexactopt} below. Unfortunately, despite there being
numerous success guarantees under certain conditions on the matrix~$A$ (and
optimal solution sparsity and uniqueness) for most algorithms proposed in
the literature, the strongest such conditions are typically themselves
\NP-hard to evaluate exactly or approximately; see, for instance,
corresponding results on $\spark(A)$, the nullspace property, and the
restricted isometry property (RIP) \cite{TP14,W18}.

Nevertheless, in light of the impressive improvements in modern solvers
over the last decades, it is still worth investigating exact solution
approaches for the different cardinality optimization problems. Here, we
focus on reformulations as mixed-integer linear and nonlinear programs
(MIPs and MINLPs, for short), accompanying structural results, and
specialized solution techniques and solver components for the considered
problems. As mentioned earlier, satisfactory results may already be achievable with off-the-shelf software applied to generic models; depending on the concrete problem/application, scalability and performance can then often be further improved by exploiting problem-specific knowledge in the solving process.

We begin with describing different approaches to model the cardinality of a
variable vector, see Section~\ref{sec:exactopt:modelingcard}. Subsequently,
we will provide overviews of both~general-purpose and problem-specific
modeling and exact solution techniques, following\linebreak our broad
classification into cardinality minimization or constrained problems
(Sections~\ref{sec:exactopt:cardmin} and~\ref{sec:exactopt:cardcons},
resp.) and cardinality-regularized optimization tasks
(Section~\ref{sec:exactopt:cardreg}).

\subsection{Modeling Cardinality}\label{sec:exactopt:modelingcard}
Typically, cardinality terms are modeled using binary indicator variables that effectively encode whether an original problem variable is zero or nonzero. This can be done in a linear fashion when the problem variables are (explicitly or implicitly) bounded, see Section~\ref{sec:exactopt:modelingcard:bounds}, or via nonlinear constraints of the complementarity type, see Section~\ref{sec:exactopt:modelingcard:complem}. It is also possible to employ a bilinear replacement technique (again using binary auxiliary variables), or to model cardinality using continuous auxiliary variables and nonlinear constraints, see Section~\ref{sec:exactopt:modelingcard:other}.

\subsubsection{Exploiting (Auxiliary) Variable Bounds}\label{sec:exactopt:modelingcard:bounds}
The classical approach to model the cardinality of a continuous variable
vector $x \in \R^n$ in a MI(NL)P is by introducing big-M constraints and
auxiliary binary variables $y \in \{0,1\}^n$ that encode whether a
continuous variable is zero or nonzero. More precisely, we can rewrite
$\norm{x}_0$ as $\ones^\top y=\sum_{i\in[n]}y_i$ provided that
\begin{equation*}
  -\M y\leq x\leq \M y,\quad y\in\{0,1\}^n,
\end{equation*}
with $\M>0$ being a sufficiently large constant. Here, if $y_i=0$, the
big-M constraint forces $x_i=0$, while in case of $y_i=1$, no restriction
is imposed upon $x_i$; conversely, if $x_i\neq 0$, then $y_i$ cannot be set
to zero and therefore must be equal to $1$, so $\ones^\top y$ indeed counts
the nonzero entries of $x$. Note that $y_i=1$, $x_i=0$ is still possible,
so generally, we only have $\ones^\top y\geq\norm{x}_0$. Nevertheless,
equality obviously holds at least in optimal points of cardinality
\emph{minimization} problems, and bounding $\ones^\top y$ from above still
correctly represents a cardinality \emph{constraint} w.r.t.~$x$. Therefore,
we may refer to $\ones^\top y$ as the cardinality of $x$ for simplicity.

From a theoretical standpoint, for problems with unbounded variables, it
might not be possible to define sufficiently large bounds within MIP or
even MINLP representations, see~\cite{HL16,JL84}. In practice, appropriate
bounds (or constants $\M$) may also not be available a priori. While
theoretical bounds based on encoding lengths of the data may exist (see,
e.g., \cite{GLS93}), they are impractically huge. Similarly, using
arbitrary large values will, generally, introduce numerical instability
(in floating-point arithmetic). Indeed, supposing a solver works with a
numerical tolerance of, say, $10^{-6}$ (the typical default tolerance of
linear programming (LP) solvers), a value of, e.g., $\M=10^{7}$ can render
the model invalid numerically: For instance, one might then have
$y_i\approx 5\times 10^{-7}$, which the solver counts as zero due to its
tolerance settings, but then the big-M constraints read
$-1/2\lesssim x_i\lesssim 1/2$ and no longer correctly enforce $x_i=0$.

Generally, it is well known that a big-M approach may lead to weak
relaxations, which can significantly slow down the solving process of
(branch-and-bound) algorithms, see, e.g., \cite{BBFLMNGS16} and also the
example later in Section~\ref{sec:card_min_example}. Nonetheless, it is a
simple and flexible approach that still may work reasonably well, and is
therefore often tried in a first effort. The general big-M modeling
paradigm can be refined by using individual lower and upper bound constants
for each variable. In particular, if bounds $\ell\leq x\leq u$ are part of
the original problem, with $0\in[\ell,u]$, we can replace the above big-M
box constraint by
\[
  Ly\leq x\leq Uy,
\]
with $L\define\Diag(\ell), U\define\Diag(u)\in\R^{n\times n}$. Individual
bounds $\ell_i$ and $u_i$ for each $x_i$ could also be derived from the
data by considering the minimal and maximal values each variable may attain
while retaining overall feasibility. Given that it is not unusual
that a MI(NL)P formulation of a COP requires considerable computational
effort to be solved to provable optimality, it may indeed be worth spending some
time to tighten a valid big-M model by computing individual bounds via
\[
  \ell_i \define\min\{ x_i\suchthat x\in X\cap F\},\quad u_i\define\max\{x_i\suchthat x\in X\cap F\},
\]
where the set $F$ symbolizes further constraints possibly required to keep
these problems bounded---e.g.,
$F$ could be a level-set of the objective function w.r.t.\ a known
(sub-)optimal value, cf.~\cite{BBFLMNGS16}. In fact, especially in the
context of solving MI(NL)Ps with a branch-and-bound algorithm, one may even
consider adaptively tightening the bounds by incorporating information on,
e.g., the optimal support size or objective function value obtained along
the way. Should these bound-computation problems turn out to be
impractically hard to solve to optimality themselves, relaxations could be
employed to still provide improved valid bounds, see,
e.g.,~\cite{BBFLMNGS16}.

Some examples for problem-specific derivations of variable bounds can be
found in Sections~\ref{sec:exactopt:cardmin} and~\ref{sec:exactopt:cardcons} below. Note also that in some problems, variables
may be scaled arbitrarily, in which case $\M$ can feasibly be set to any
positive value, and that it may even be possible to not explicitly include
the variables~$x$ in a problem---see the discussion of models and methods
for \cardmin{$Ax=0,x\neq 0$} and \cardmin{$Ax=b$} in
Section~\ref{sec:exactopt:cardmin}. Section~\ref{sec:card_min_example} also provides an illustrative example which, in particular, shows the benefits of good
choices of the constant $\M$.

\subsubsection{Complementarity-Type Formulations}\label{sec:exactopt:modelingcard:complem}
A conceptually different possibility to model the cardinality
and/or support couples auxiliary binary variables to the continuous
variables by means of (nonlinear) \emph{complementarity(-type)
  constraints}:
\begin{equation*}
 x_i(1-y_i)=0\quad\forall\,i\in[n],~y\in\{0,1\}^n\qquad\Rightarrow\quad
 \ones^\top y\geq \norm{x}_0.
\end{equation*}
Here, $y_i=0$ again implies $x_i=0$, so $\ones^\top y \geq \norm{x}_0$ for
all feasible points $x,y$. In optimal solutions of cardinality minimization
problems, integrality of $y$ and $\ones^\top y = \norm{x}_0$ holds
automatically, so $y\in\{0,1\}^n$ can \emph{always} be relaxed to
$0\leq y\leq\ones$, see, e.g.,
\cite{FMPSW18}. Figure~\ref{fig:AuxiliaryVariablesFeasibleSets} illustrates
the effects of the auxiliary variables $y$ here compared to the big-M
formulation. Note also that complementarity-type constraints as above do \emph{not} (implicitly) assume boundedness of the $x$-variables---a potential advantage over the big-M approach, albeit at the cost of linearity.

\begin{figure}
  \centering
  \subfloat[$\ell_i y_i \leq x_i \leq u_iy_i, y_i \in \{0,1\}$]{
    \begin{tikzpicture}
      \fill[blue!20] (0,0) -- (1.5,1) -- (-1,1) -- cycle;\fill[blue] (0,0) circle(0.075);
      \draw[blue, line width = 2pt] (-1,1) -- (1.5,1);

      \draw[->] (-2,0) -- (-1,0) node[below]{$\ell_i$} -- (0,0) node[below]{$0$} -- (1.5,0) node[below]{$u_i$} -- (2.5,0) node[right]{$x_i$};
      \draw[->] (0,0) -- (0,1) node[left]{$1$} -- (0,1.5) node[left]{$y_i$};
    \end{tikzpicture}
  }
  \qquad\qquad
  \subfloat[$\ell_i \leq x_i \leq u_i, x_i(1-y_i) = 0, y_i \in \{0,1\}$]{
    \begin{tikzpicture}
      \draw[blue!40, line width = 3pt] (0,0) -- (0,1);
      \fill[blue] (0,0) circle(0.075);
      \draw[blue, line width = 2pt] (-1,1) -- (1.5,1);

      \draw[->] (-2,0) -- (-1,0) node[below]{$\ell_i$} -- (0,0) node[below]{$0$} -- (1.5,0) node[below]{$u_i$} -- (2.5,0) node[right]{$x_i$};
      \draw[->] (0,0) -- (0,1) node[left]{$1$} -- (0,1.5) node[left]{$y_i$};
    \end{tikzpicture}
  }
  \caption{\footnotesize\itshape Effect of auxiliary variables $y$ in the
    big-M formulation (left) and complementarity-type formulation
    (right). Shaded areas illustrate the effect of relaxing the binary
    variable to $y_i \in [0,1]$.}\label{fig:AuxiliaryVariablesFeasibleSets}
\end{figure}
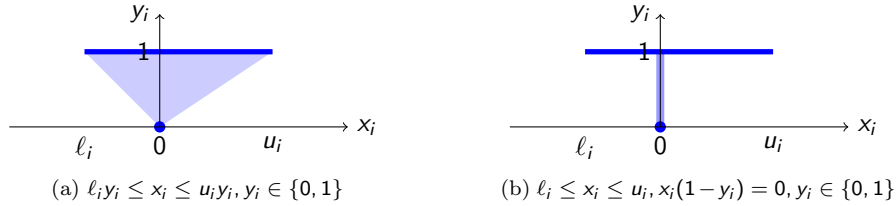

Constraints like $x_i(1-y_i)=0$ are related to the class of
\emph{equilibrium constraints}, cf.~\cite{LPR96}, and can also be
interpreted as \emph{specially-ordered set constraints of type~1} (SOS-1
constraints)~\cite{BW05}, since only one out of a group of
variables---here, a pair $x_i,(1-y_i)$---may be nonzero. Modern MIP solvers
can exploit this structural know\-ledge in certain ways (e.g., for
bound-tightening), so it may be worth informing a solver of this explicitly
in addition to another employed formulation, as done, e.g.,
in~\cite{BKM16}. Specialized branching schemes for SOS-1 or complementarity
constraints were discussed, e.g., in~\cite{BT70,DFJN01}.

Note also that these complementarity-type constraints are bilinear and can
therefore be relaxed using McCormick envelopes~\cite{McC76}, a
relaxation-by-linearization technique that actually is an \emph{exact}
reformulation for bounded $\ell\leq x\leq u$ and $y\in\{0,1\}^n$:
Introducing auxiliary variables $z_i\define x_i y_i$ to replace each
bilinear term and additional linear constraints $z_i\geq\ell_i y_i$,
$z_i\geq x_i+u_i y_i-u_i$, $z_i\leq u_i y_i$, and
$z_i\leq x_i +\ell_i y_i -\ell_i$ ensures equivalence of the original and the
extended problem in this case. However, in the special case in which the bilinear terms are associated with complementarity constraints deriving from cardinality, the McCormick envelopes do not add anything to the big-M reformulation above.

The paper~\cite{FMPSW18} describes various ways to reformulate the
complementarity-type constraints. Because complementarity constraints are
usually defined for nonnegative variables, the above variant is called
\emph{half-comple\-men\-tari\-ty constraints} there. A variant with
classical \emph{full} complementarity constraints can easily be obtained by
splitting the variable $x$ into its nonnegative and nonpositive parts,
respectively.  Moreover, \cite{FMPSW18} discusses four equivalent nonlinear
reformulations. The motivating problem of that paper is of the form
\cardmin{$Ax\geq b,Cx=d$}, though most of the theoretical results on
optimality conditions of the nonlinear reformulations were developed for
the more general problem class
\mbox{\cardreg{$\tfrac{1}{\gamma}f(x)$}{$g(x)=0,h(x)\leq 0$}} with
$\gamma>0$ and continuously differentiable functions $f:\R^n\to\R$,
$g:\R^n\to\R^p$, and $h:\R^n\to\R^q$.

The very recent work~\cite{YMP19} introduced a branch-and-cut algorithm to
solve general \emph{linear programs with complementarity constraints}
(LPCCs) to \emph{global} optimality. In contrast, the previous work
\cite{FMPSW18} was largely concerned with computing stationary solutions.
The LPCC viewpoint offers a quite flexible modeling paradigm with a host of
diverse applications (see, e.g., those surveyed in~\cite{HMPY12}),
including, in particular, \cardmin{$X$} and \cardcons{$c^\top x$}{$k$}{$X$}
for polyhedral feasible sets~$X$, cf.~\cite{FMPSW18,BKS16}. For an overview
of related earlier works on exact methods for (certain subclasses of) LPCCs
or strongly related problems, see~\cite{YMP19} and the many references
therein. We would like to mention explicitly the interesting
minimax/Benders decomposition approach of~\cite{HMPBK08} that was extended
to convex QPs with complementarity constraints in~\cite{BMP13}, and the
quite extensive research into polyhedral aspects---i.e., cutting
planes---in~\cite{DFJN01,DFJN02,DFN03,S10,MPY12,K14,DFKZ14,KDF15,KTP19,FP17,FP18,FP20};
the last three references also consider overlapping cardinality constraints
(formulated as complementarity or SOS-1 constraints) and other MIP solver
components like branching rules for corresponding LPCCs. It is also worth
mentioning that convex quadratic constraints, as they appear in
\cardmin{$\norm{Ax-b}_2\leq\delta$} and similar problems, can be recast as
\emph{second order cone (SOC) constraints}. These have also been studied
extensively, often with a particular focus on deriving cutting planes for
(mixed-integer) SOC programs, see, e.g.,
\cite{GLind08,SBL08,D09,BS13,V13,MT15,CPT18,AJ19,FG06, FFG16, FFG17} and
references therein.

\subsubsection{Further Ways to Model Cardinality}\label{sec:exactopt:modelingcard:other}
Another alternative to model cardinality is considered in \cite{BCWP19}
(see also \cite{BCW18}): Here, the auxiliary binary variables $y$ are
linked to $x$ in the same way as before, i.e., they essentially encapsulate
the logical constraint that if $y_i=0$, then $x_i=0$ shall hold as
well. (Indeed, ``$y_i=0~\Rightarrow~x_i=0$'' is a special case of an
\emph{indicator constraint}, and the reformulations discussed here can be
applied to more general such constraints, see, e.g.,
\cite{BBFLMNGS16,BLTW15} for detailed discussions.) The key observation
then is that one can replace $x_i$ by $y_i x_i$ throughout the problem
formulation; indeed, any $x_i$ then only contributes\footnote{Note that,
  however, $x_i\neq 0$ would then in principle be possible even if
  $y_i=0$. While this does not influence feasibility or the optimal
  solution value in the sense of the original formulation, it needs to be
  considered when extracting the optimal solution. There, the corresponding
  $x_i$ can w.l.o.g.\ be set to zero. \cite{BCWP19,BCW18} propose to add a
  ridge regularization term $\gamma\norm{x}_2^2$ to the objective for
  algorithmic reasons, which automatically enforces that $y_i=0$ indeed
  implies $x_i=0$ in an optimal solution.} to a constraint or the objective
if $y_i=1$. The resulting mixed-integer nonlinear problems considered in
\cite{BCWP19,BCW18} are solved by an \emph{outer-approximation} scheme that
is shown to often work more efficiently than using black-box MINLP
solvers. The general technique is well-known and quite broadly applicable,
cf. \cite{DG86,FL94}. The main idea is a decomposition of the problem that
allows for repeatedly solving an outer problem involving only the binary
variables $y$, and an inner problem that can be solved efficiently (for a
fixed $y$) and provides subgradient cuts (i.e., linear inequalities based
on the subgradient of the inner problem, which can be seen as a convex
function in $y$) that refine the outer problem. Note that the same
arguments as for (half-)complementarity constraints would allow to
linearize various constraints (e.g., linear inequalities) in the context of
the replacement-reformulation technique mentioned earlier (replacing $x_i$
by $y_i x_i$ directly as suggested in \cite{BCWP19,BCW18}), which seems to
not have been tried out yet.

Interestingly, it is also possible to exactly model the cardinality of a vector $x\in\R^n$ using only continuous auxiliary variables, along with certain (nonlinear) constraints. For instance, \cite{YG16} show that
\[
    \norm{x}_0=\min\big\{\norm{u}_1 \,:\, \norm{x}_1=x^\top u,~-\ones\leq u\leq\ones\big\}.
\]
Similar but more complicated reformulations for cardinality constraints ($\norm{x}_0\leq k$) can be found in, e.g., \cite{HG14} (see also Sections~\ref{sec:exactopt:cardcons} and~\ref{sec:inexactopt:relaxation_constraints}), though it seems unclear whether those might be helpful in a cardinality minimization context.

\subsection{Cardinality Minimization}\label{sec:exactopt:cardmin}
The generic cardinality minimization problem \cardmin{$X$} can be
reformulated using auxiliary binary variables~$y$ with any of the
techniques of the previous subsection.

The big-M approach has been applied in~\cite{BNCM16,MBN19} to
\cardmin{$\norm{Ax-b}_p\leq\delta$}, for $p\in\{1,2,\infty\}$, as well as
the corresponding cardinality-constrained and -regularized problems in a
unified fashion; see also references therein for partial earlier
treatments, e.g., of \cardmin{$Ax=b$} in~\cite{KEI13}. While the resulting
mixed-integer (linear or nonlinear) problems were solved with an
off-the-shelf MIP solver in~\cite{BNCM16}, \cite{MBN19} demonstrated (for
$p=2$) that considerable runtime improvements can be achieved if the usual
LP-relaxations that form the standard backbone of modern MIP solvers are
replaced by problem-specific other relaxations, involving the $\ell_1$-norm
as a proxy for sparsity, that admit very fast first-order solution
algorithms (see Section~\ref{sec:inexactopt:surrogates} for an overview of
many such methods). Both these works apparently employ a simple heuristic
to select the big-M constant: starting with
$\M=1.1\norm{A^\top y}_\infty/\norm{y}_2^2$ (a least-squares estimate of
the maximum amplitude of $1$-sparse solutions), accept the computed optimal
solution $x^*$ if $\norm{x^*}_\infty <\M$ and restart otherwise with $\M$
increased to $1.1\M$.

As indicated in the previous subsection, we may consider computing
individual bounds on each variable (and locally tightening them within a
branch-and-bound solving process). As an example, let us consider $Ax=b$,
with the usual assumption that $\rank(A)=m<n$ (and $b\neq 0$). For
\cardmin{$Ax=b$}, we then know that the optimal value is at most $m$ (since
there exists an invertible $m\times m$ submatrix of $A$); thus, for each
$i\in[n]$, we could consider
\[
  \ell_i\define\inf\{x_i\,:\,Ax=b,~\norm{x}_0\leq m\},\qquad u_i\define\sup\{x_i\,:\,Ax=b,~\norm{x}_0\leq m\}.
\]
However, these problems may be as hard to solve to optimality as the
original problem, so one may want to consider relaxations, and one might
also encounter unboundedness (even though the original problem is bounded)
which may be non-trivial to circumvent. In particular, suppose we use a
complementarity reformulation of the cardinality constraint here:
\[
  \ell_i=\inf\{x_i\,:\,Ax=b,~x_j(1-z_j)=0~\forall\,j\in[n],~\ones^\top z\leq m,~z\in[0,1]^{n-1}\}
\]
($u_i$ analogously). Then, one could employ known relaxations of
complementarity constraints (see Section~\ref{sec:exactopt:modelingcard}
and, in particular, Section~\ref{sec:inexactopt:relaxation_constraints}) to
obtain valid values for~$\ell_i$ and~$u_i$---or detect subproblem
unboundedness---by solving the respective relaxations. Alternatively,
boundedness provided, we may combine the big-M selection heuristic from
\cite{BNCM16} outlined earlier with the bound-computation problems: For any
$i\in[n]$, let $L_i$ and $U_i$ be diagonal matrices with the bounds already
computed for variables $x_1,\dots,x_{i-1}$ on their respective diagonals,
and let $\M>0$. Then, to compute a lower bound for $x_i$ (analogously for
an upper bound), we can solve
\begin{align*}
  \min\quad &x_i\\
  \text{s.t.}\quad &Ax=b,~\Diag\big((\diag(L_i),-\M\ones)\big) z \leq x\leq
                     \Diag\big((\diag(U_i),\M\ones)\big) z,\\
  &\ones^\top z\leq m,~z\in[0,1]^n
\end{align*}
repeatedly with increased $\M$ as long as the solution satisfies any of the
big-M constraints with equality. Note that the above problem is an LP and
therefore efficiently solvable in practice; for bound validity, we do not
need to retain the integrality of $z$. Note also that we could easily
integrate possible lower bounds $\underline{s}$ on the optimal cardinality
into either of the above problems by means of the inequality
$\ones^\top z\geq\underline{s}$. Depending on the problem, such bounds may
be available a priori; e.g., for \cardmin{$Ax=b$}, we trivially know that
any solution must have at least two nonzero entries unless $b$ is a scaled
version of a column of $A$ (which can easily be checked).

Another example can be found in~\cite{WO13b}, which considers a big-M
mixed-integer QP (MIQP) reformulation of
\cardmin{$(x-b)^\top Q(x-b)\leq\varepsilon$}, with $Q$ positive
definite. There, individual bounds $\ell_i$ and $u_i$ for each $x_i$ are
derived from the data and even turn out to have closed-form expressions:
\begin{align*}
  \ell_i = &\min\{ x_i\suchthat (x-b)^\top Q(x-b)\leq\varepsilon \} = b_i-\sqrt{\varepsilon\big(Q^{-1}\big)_{ii}},\\
  u_i = &\max\{ x_i\suchthat (x-b)^\top Q(x-b)\leq\varepsilon \} = b_i+\sqrt{\varepsilon\big(Q^{-1}\big)_{ii}}.
\end{align*}
Note that the feasible set of \cardmin{$(x-b)^\top Q(x-b)\leq\varepsilon$}
extends infinitely in certain directions if $Q$ is rank-deficient (i.e.,
only semi-definite). In particular, this is the case for the
correspondingly reformulated problem \cardmin{$\norm{Ax-b}_2\leq\delta$} in
the usual setting with $A\in\R^{m\times n}$, $\rank(A)=m<n$. Then, as for
\cardmin{$Ax=b$}, boundedness of the bound-computation problems has to be
ensured explicitly, for which a cardinality constraint again seems the
natural choice, and relaxation offers ways to circumvent intractability
issues.

As alluded to earlier, some problems allow reformulations or specialized
models that can avoid the need for a big-M or complementarity/bilinear
cardinality formulation. For \cardmin{$Ax=b$} and
\cardmin{$\norm{Ax-b}_\infty\leq\delta$}, \cite{JP08} proposed a
branch-and-cut algorithm that exploits a reformulation of these problems as
MaxFS instances. For instance, $Ax=b$ (with $\rank(A)=m<n$) can be
transformed into reduced row-echelon form via Gaussian elimination,
yielding an equivalent system $u+Rv=r$; a minimum-support solution for this
can be found by finding a maximum feasible subsystem of the infeasible
system $u+Rv=r,u=0,v=0$. A characterization of minimally infeasible
subsystems (the complements of maximal feasible subsystems) by means of the
so-called alternative polyhedron (cf.~\cite{GR90,PR96}) then yields a
binary IP model with exponentially many constraints that are separated and
added to the model dynamically within a branch-and-bound solver framework;
see \cite{JP08,P08,APT03} for the details. At the time of publication, this
branch-and-cut method could only solve rather small instances to
optimality. The scheme also incorporates several heuristics for the MaxFS
(or MinIISC) problem, adapted to the resulting special instances, with one
noteworthy conclusion being that the common $\ell_1$-norm minimization
approach may not be the best choice.

For the problem \cardmin{$Ax=0,x\neq 0$}, i.e., computing $\spark(A)$, note
that any feasible vector lies in the nullspace of a matrix and, therefore,
can be scaled arbitrarily without compromising its feasibility or affecting
its $\ell_0$-norm. Thus, every value $\M>0$ works in a big-M cardinality
modeling approach. Spark computation is discussed in detail in~\cite{T19};
in particular, a formulation with $\M=1$ was employed and---utilizing
additional auxiliary binary variables to model the nontriviality constraint
$x\neq 0$---the resulting MIP was given as
\begin{equation}\label{eq:sparkMIP}
  \min\left\{\ones^\top y\,:\,Ax=0,~-y+2z\leq x\leq y,~\ones^\top z=1;~y,z\in\{0,1\}^n,~x\in\R^n\right\};
\end{equation}
see also analogous MIP models and/or exact algorithms for the cospark,
i.e., vector matroid cogirth problem, in \cite{CCD07,KPD11,AH14}. Here,
only one of the $z$-variables can become $1$, and $z_i=1$ implies
$y_i=x_i=1$, thus ensuring $x\neq 0$ and also eliminating sign symmetry (if
$Ax=0$, then also $A(-x)=0$). Moreover, by exploiting relations to matroid
theory, \cite{T19} proposed the following pure binary IP model for the
spark computation problem \cardmin{$Ax=0,x\neq 0$}:
\begin{equation}\label{eq:sparkIP}
  \min\left\{\ones^\top y\,:\,\ones^\top y_{B^c}\geq 1\quad\forall\,B\subset[n]:~\abs{B}=\rank(A_B)=m;~y\in\{0,1\}^n\right\},
\end{equation}
where $B^c\define[n]\setminus B$. %and $y_S$ or $A_S$ denote the restriction
%of $y$ or $A$ to entries or columns indexed by~$S$, respectively. %def.'d in 1.1
This formulation avoids an explicit representation of~$x$ altogether, at the
cost of having an exponential number of constraints.  Nevertheless, these
constraints can be separated in polynomial time by a simple greedy method,
and~\cite{T19} devises a problem-specific branch-and-cut method combining
the above model~\eqref{eq:sparkMIP} with dynamic generation of the
inequalities from~\eqref{eq:sparkIP} (and some other valid inequalities),
and incorporating dedicated heuristics, propagation and pruning rules as
well as a branching scheme. Using numerical experiments detailed
in~\cite{T19} as an example, Figure~\ref{fig:spark_solved_over_time}
illustrates a key point we wish to emphasize for COPs in general---namely,
that (on average) dedicated solvers can solve more instances more quickly
than by simply plugging a compact model into a general-purpose MIP solver,
and prove better quality guarantees in cases that took unreasonably long to
solve to optimality.

\begin{figure}[t]
  \centering
  \includegraphics[width = 0.5\textwidth]{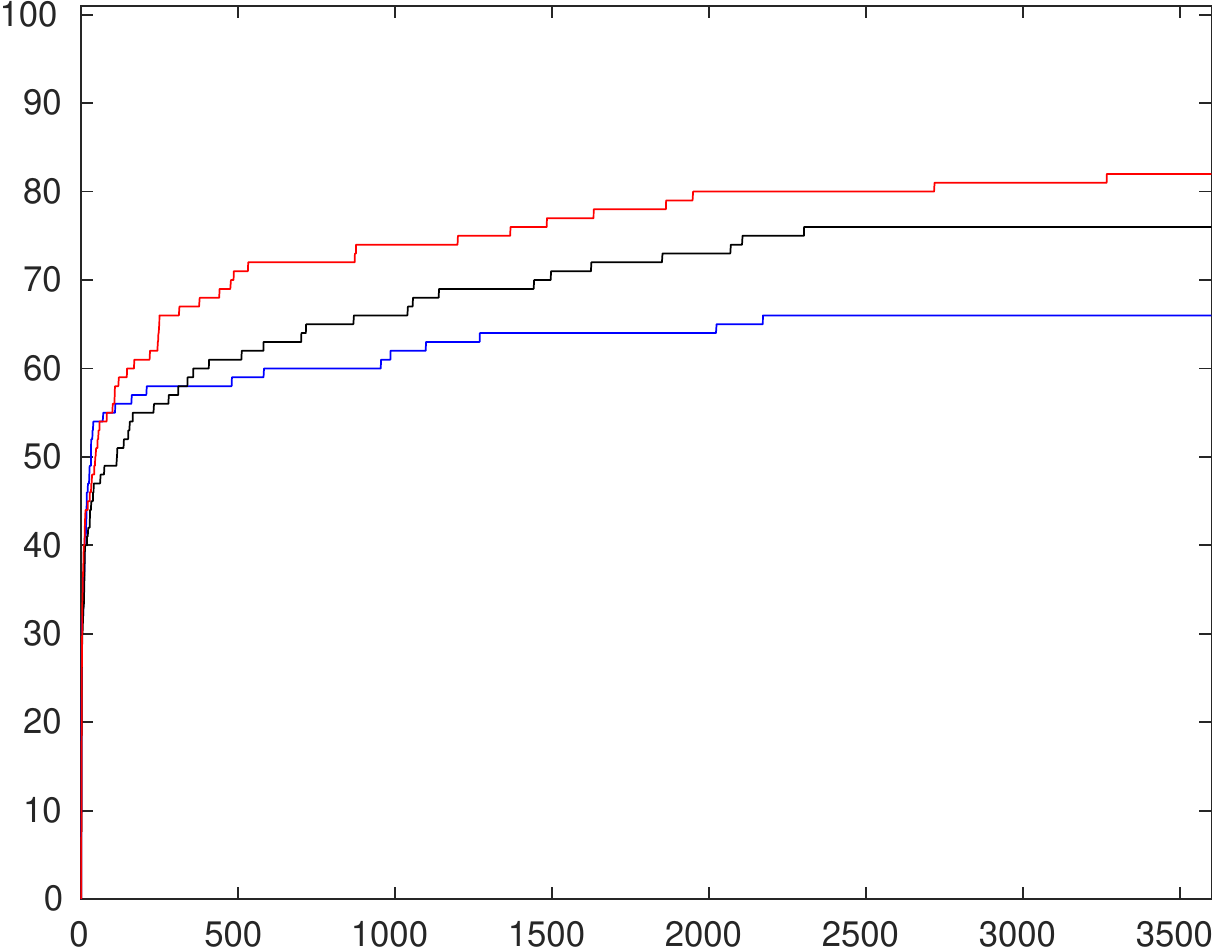}
  \caption{\footnotesize\itshape Results from~\cite{T19}: running time (in seconds; horizontal
    axis) versus number of instances solved to optimality (out of 100;
    vertical axis) for the commercial MIP solver CPLEX~\cite{cplex} applied to the
    compact spark model (blue), the pure binary-IP spark solver (SparkIP;
    black) and the combined MIP spark solver (SparkMIP; red); the latter
    two were implemented in SCIP~\cite{scip} and employed CPLEX as LP-relaxation
    solver. CPLEX achieved the smallest final optimality gap 66 times
    (i.e., only on those instances it could solve to optimality within the
    time limit, which were also solved by the others), SparkIP 79 times,
    and SparkMIP 95 times.}
  \label{fig:spark_solved_over_time}
\end{figure}

A binary IP formulation analogous to~\eqref{eq:sparkIP} can also be given
for \cardmin{$Ax=b$}, based on the following result (we omit its
straightforward proof):

\begin{lem}\label{lem:l0minCovIneqs}
  A set $\emptyset\neq S\subseteq[n]$ is a (inclusion-wise minimal)
  feasible support for~$x$ w.r.t.\ $Ax=b$ if and only if
  $S\cap B^c\neq\emptyset$ for all (maximal) infeasible supports $B$.
\end{lem}
This IP then reads
\begin{equation}\label{eq:sparseRecoveryIP}
  \min\left\{\ones^\top y \,:\, \ones^\top y_{B^c}\geq 1\quad\forall\text{(max.) infeas.\ supports }B;~y\in\{0,1\}^n\right\}.
\end{equation}
In fact, Lemma~\ref{lem:l0minCovIneqs} and~\eqref{eq:sparseRecoveryIP}
extend directly to \cardmin{$\norm{Ax-b}_\infty\leq\delta$}; however, in
contrast to the spark case~\eqref{eq:sparkIP}, the separation problem for
the inequalities in~\eqref{eq:sparseRecoveryIP} w.r.t.\ maximal infeasible
supports can be shown to be generally \NP-hard~\cite{TP12}. The approach is
strongly related to the model from~\cite{JP08} discussed earlier, which
essentially splits the support into that of the positive and negative parts
of~$x$, respectively. This split seems to have some structural advantages
w.r.t.\ greedy separation heuristics, even though the underlying IP has
(roughly) twice as many variables due to the transformation to
MaxFS/MinIISC.

Note that one can also make use of the spark IP formulation, and (with
slight modifications) the solver from~\cite{T19}, to tackle
\cardmin{$Ax=b$}: While \cardmin{$Ax=0$, $x\neq 0$} searches for the overall
smallest circuit of the vector matroid induced by the columns of $A$,
\cardmin{$Ax=b$} can be viewed as seeking the smallest circuit of the
vector matroid over $(A,-b)$ that mandatorily contains the right-hand-side
column (see also~\cite{CCD07}). Thus, \cardmin{$Ax=b$} is equivalent to
\begin{align}
\label{eq:sparseRecoveryViaSparkIP}  \min\quad &\ones^\top y\\
\nonumber  \st &\ones^\top y_{B^c}\geq 1\quad\forall\,B\subset[n]:~\abs{B}=\rank((A_B,-b))=m,~y\in\{0,1\}^{n},
\end{align}
i.e., the covering-type inequalities hold for all complements of bases of
the matroid over the columns of $(A,-b)$ that contain $n+1$. Indeed, it can
easily be seen that bases containing the $(n+1)$-th column of $(A,-b)$
correspond to infeasible supports (w.r.t.\ $Ax=b$) from
Lemma~\ref{lem:l0minCovIneqs} once that column is removed. While these
infeasible supports are not necessarily maximal (there could be columns of
$A$ that are linearly dependent on the ones contained in the basis and thus
could be included in a linear combination without changing infeasibility),
the separation problem can still be solved by a greedy method, unlike for
the covering inequalities corresponding to \emph{maximal} infeasible
supports.  We will explore big-M selection and the approach to solve
\cardmin{$Ax=b$} via \eqref{eq:sparseRecoveryViaSparkIP} a bit further as
an illustrative example in Section~\ref{sec:card_min_example} below.

Besides tightening big-M bounds as discussed earlier, it has been shown in
se\-veral cardinality minimization applications that standard
(general-purpose) branch-and-bound MIP solvers can be improved
significantly by exploiting problem-specific heuristics and relaxations (or
methods tailored to the specific structure of the relaxations encountered
in the process of solving the MIPs), see, e.g.,
\cite{MBN19,WO13a,WO13b,T19}. Of course, this also holds for certain
mixed-integer nonlinear programming (MINLP) formulations. An example is the
efficient heuristic providing high-quality starting solutions for the exact
branch-and-cut scheme (including specialized cuts and branching rules) for
\cardmin{$\norm{Ax-b}_2\leq\delta$,\,$\abs{x}\in\{0,1\}$,\,$x\in\C^n$}
recently proposed in~\cite{FHMPPT18}.

A general impression, which may be gleaned from all the MIP efforts in the
aforementioned references is that \emph{finding} good or even optimal
solutions is often possible with dedicated heuristics (including running
MIP solvers for a limited amount of time), but that \emph{proving}
optimality is hard not only in theory but also in practice.

Finally, we point out that cardinality \emph{minimization} problems could
also be solved by means of a sequence of cardinality-\emph{constrained}
problems with $\norm{x}_0\leq k$ for $k=1,2,\dots$, until the first
feasible subproblem is found (indicating minimality of the respective
cardinality level); of course, one could also apply, e.g., binary search in
this context. Depending on the concrete constraints, since this sequential
approach leaves the possibility to choose an arbitrary objective function,
one could simplify the feasible set by moving parts of it into the
objective---for instance, \cardmin{$\norm{Ax-b}\leq\delta$} could be
tackled by solving \cardcons{$\norm{Ax-b}$}{$k$}{$\R^n$} for $k=1,2,\dots$
until the first subproblem with optimal objective value of at most~$\delta$
was found.

\subsubsection{An Illustrative Example}\label{sec:card_min_example}
Since~\eqref{eq:sparseRecoveryViaSparkIP} appears to be novel, we adapted
the spark-specific code from~\cite{T19} to this variant in order to, in the
following, provide an illustrative example on how heavily exploiting
problem-specific knowledge can significantly improve the performance of
exact solvers versus black-box models. The following example consists of
the matrix~$A$ from instance 67 from~\cite{T19} (a $192\times 384$ binary
parity-check matrix of a rate-$1/2$ WRAN LDPC code) and a right-hand side
$b\define Ax$, where $x\in\R^{384}$ has $30$ nonzero components with
positions drawn uniformly at random and entries drawn i.i.d.\ from the
standard normal distribution. The (modified) spark code is implemented
using the open-source MIP solver SCIP \cite{scip}, which we also employ as
a black-box solver for the standard big-M formulation of \cardmin{$Ax=b$},
i.e.,
\begin{equation}\label{eq:l0minMIP}
  %\min \ones^\top y\quad\st\quad Ax=b,~-\M y\leq x\leq \M y,~y\in\{0,1\}^n.
  \min \big\{ \ones^\top y\,:\, Ax=b,~-\M y\leq x\leq \M y,~y\in\{0,1\}^n\big\}.
\end{equation}
We will consider different choices of $\M$, to illustrate how the quality
of the bounds greatly influences solver efficiency. Specifically, we solve
the problem for conservative choices $\M=1000$ and $\M=100$, an
``optimistic'' choice $\M=10$, as well as $\M=3.9$ (corresponding
approximately to the $99.99\%$ quantile of the standard normal
distribution) and $\M=2.6$, the largest absolute value of the entries in
the generated $x$ (rounded up to one significant digit). Thus, in terms of
uniform bounds for all variables, the latter two exploit knowledge of at
least the distribution of the ground-truth signal vector $x$ and are unrealistically tight, and highly instance-specific. For the problem-specific IP solver, which
does not require big-M values, we illustrate how adding different
components that are tailored to the problem under consideration yields an
increasingly faster algorithm (we refer to~\cite{T19} for detailed
descriptions and omit stating the straightforward modifications to
implicitly handle $Ax=b$ rather than $Ax=0$): Version (i) is characterized
by the basic model~\eqref{eq:sparseRecoveryViaSparkIP} with separation of
covering inequalities as well as generalized-cycle inequalities (but no
local cuts), version (ii) adds a propagation routine (to infer, e.g.,
further variable fixings after branching, and execute problem-specific
pruning rules), (iii) additionally incorporates the well-known
$\ell_1$-minimization problem $\min\{\norm{x}_1:Ax=b\}$ as a primal
heuristic, and finally, in version (iv), we turn off several costly primal
heuristics that are part of SCIP but were not helpful (in particular,
diving and rounding heuristics).  The results are summarized in
Table~\ref{tab:exampleL0minBigMvsSparkIP}; all experiments were run in
single-thread mode under Linux on a laptop (Intel Core i7-8565U CPUs %@
1.8\,GHz and 8\,GB memory); we used the LP-solver SoPlex 5.0.0 that comes
with SCIP 7.0.0 to solve all relaxations.

\begin{table}[t]
  \begin{center}
  \caption{\footnotesize\itshape Runtimes and number of node for model/solver variants on the example instance of \cardmin{$Ax=b$}.}
  \label{tab:exampleL0minBigMvsSparkIP}
  \footnotesize % SIAM style guide: tables are set in 8pt... 
  \begin{tabular*}{\textwidth}{@{\extracolsep{\fill}}@{~}lr@{\qquad}rr@{\quad}}\toprule          
    solver/model              & $\M$ & runtime [s] & nodes \\
    \midrule
    black-box SCIP big-M-MIP  & 1000 & 544.1 & 9503\\
    black-box SCIP big-M-MIP  &  100 &  86.6 &  859\\
    black-box SCIP big-M-MIP  &   10 &  43.1 &  423\\
    black-box SCIP big-M-MIP  &  3.9 &  25.5 &    5\\
    black-box SCIP big-M-MIP  &  2.6 &   3.5 &    1\\[0.5em]  
    modified Spark-IP (i)     &   -- &  62.2 &  621\\
    modified Spark-IP (ii)    &   -- &  40.7 &  377\\
    modified Spark-IP (iii)   &   -- &  20.7 &  165\\
    modified Spark-IP (iv)    &   -- &  15.6 &  149\\
    \bottomrule
    \end{tabular*}
  \end{center}
\end{table}

The experiment clearly shows that the black-box approach profits greatly
from good big-M values. Nevertheless, it is important to keep in mind that
in practical applications, one may not have sufficiently useful information
(such as exploiting knowledge of the ``signal'' distribution to come up
with a reasonable guess like $\M=3.9$ in the above example). If, on the
other hand, explicit bounds are known for $x$, then the approach may work
quite well and is certainly worth a try. However, even if good guesses
(like the $3.9$ above) are available, a dedicated problem-specific solver
may still achieve significant performance improvements. Here, in its ``most
sophisticated'' variants (iii) and (iv), the modified spark computation
code outperforms the black-box solver by at least 20\% in terms of running
time. Even the more rudimentary version (i) is significantly faster than
the big-M approach if no knowledge regarding a good choice of $\M$ is
available and one needs to choose a relatively large $\M$ ($100$ or $1000$)
to be on the safe side. Finally, note that one could, in principle, merge
the two approaches (i.e., use the MIP model~\eqref{eq:l0minMIP} with
dynamically added cuts derived from~\eqref{eq:sparseRecoveryViaSparkIP});
for \cardmin{$Ax=0,x\neq 0$}, such an approach indeed turned out to be
beneficial (see~\cite{T19}), but recall that there were no big-M selection
issues due to scalability of nullspace vectors.

\subsection{Cardinality-Constrained Optimization}\label{sec:exactopt:cardcons}
Naturally, the techniques discussed in
Section~\ref{sec:exactopt:modelingcard} can also be used to reformulate
cardinality-\emph{constrained} problems. For instance, in the presence of
variable bounds $\ell\leq x\leq u$ (possibly of a big-M nature), the
general problem \cardcons{$f$}{$k$}{$X$} can be written as
\begin{align*}
  \min\left\{f(x)\,:\, L y\leq x\leq U y,~\ones^\top y\leq k,~x\in X,~y\in\{0,1\}^n\right\},
\end{align*}
where $U\define\Diag(u)$ and $L\define\Diag(\ell)$. Similarly, the
reformulations using comple\-men\-tarity-type constraints can be employed
in an analogous fashion, although the resulting theoretical properties may
differ in some fine points. For the sake of brevity, we omit the
straightforward details. 

It is also possible to algebraically formulate cardinality constraints 
on vectors, as well as rank constraints on matrices, using continuous auxiliary variables and a set of linear constraints plus one bilinear inequality, see~\cite{HG14}. Somewhat surprisingly, it seems that these reformulations are not very well known and have, to our knowledge, hardly been employed in practical algorithms thus far. The key result for vector sparsity is \cite[Thm.~1]{HG14}: $x\in\R^n$ satisfies $\norm{x}_0\leq k$ if and only if there exist $t\in\R$ and $y,q,w\in\R^n$ such that 
\[
  \norm{q}_1+(k+1)\norm{w}_\infty\leq t\leq x^\top y,\quad x=q+w,\quad \norm{y}_1\leq k,\quad\norm{y}_\infty\leq 1;
\]
note that the $\ell_1$- and $\ell_\infty$-norm terms can be linearized as usual. For the ana\-logous result on rank constraints, see \cite[Thms.~2 and~3]{HG14}. 
The reformulations from~\cite{HG14} are closely related to the sum of the~$k$ largest absolute values of entries in the vector case, or singular values in the matrix case, respectively (see also the ``trimmed LASSO'' discussed at the end of Section~\ref{sec:inexactopt:relaxation_constraints}). A related characterization of a cardinality constraint can be derived from~\cite{YG16}, where it is shown that $\norm{x}_0=\min\{\norm{u}_1 : \norm{x}_1=x^\top u,~-\ones\leq u\leq\ones\}$ for any $x\in\R^n$, so consequently,
\begin{equation}\label{cardconsYG16}
    \norm{x}_0\leq k\quad\Leftrightarrow\quad\norm{u}_1\leq k,~\norm{x}_1=x^\top u,~-\ones\leq u\leq\ones.
\end{equation}

It is noteworthy that bound-computation problems may be easier for
cardinality-constrained problems than for cardinality-minimization: To
ensure validity of computed bounds, it suffices to ensure that a known
upper bound on the minimum objective value is not exceeded
(cf.~\cite{BBFLMNGS16})---for problems of the class \cardmin{$X$}, this
unfortunately leads to (generally intractable) cardinality
constraints. Here, the cardinality constraint can actually be omitted (or,
more precisely, relaxed to $\norm{x}_0\leq n$), so bound-computation
problems may look like
\[
  \inf_x/\sup_x\left\{ x_i \,:\, f(x)\leq \bar{f},~x\in X\right\}.
\]
For instance, \cite{BKM16} suggests the data-driven bounds
$\inf/\sup\{x_i : \norm{Ax-b}_2\leq\bar{f}\}$ for
\cardcons{$\norm{Ax-b}_2$}{$k$}{$\R^n$}, which are simple convex problems.
The required bound~$\bar{f}$ on the optimal objective value can be obtained
by any heuristic, or possibly analytically.

Thus, in particular, the exact branch-and-cut solvers of~\cite{YMP19} (and
some earlier works referenced therein) for LPCCs can be used if $f$ is an
affine-linear function.  Moreover, the polyhedral results (valid
inequalities for polytopes with cardinality constraints) from the
references given at the end of Section~\ref{sec:exactopt:modelingcard}, as
well as the aforementioned branching schemes (e.g., \cite{DFJN01}) can also
be applied in the present general context. Note that~\cite{B96} describes a
branching rule that allows to avoid auxiliary binary variables.

An exact mixed-binary minimax (or outer-approximation) algorithm was
developed in~\cite{BVP17,BPVP17} for the sparse SVM problem
\cardcons{$L(w,b)$}{$k$}{$(w,b)\in\R^{n+1}$}, subsuming a ridge
regularization term in $L$ (cf. Section~\ref{sec:concreteprobs:statML}),
with encouraging performance in the context of logistic regression and
hinge loss sparse SVM. It was later extended to more general MIQPs,
including, in particular, the portfolio selection problem,
see~\cite{BCW18,BCWP19}. This method is the latest in a series of exact
algorithm proposals for variants of MIQPs with cardinality constraints,
often focusing on portfolio optimization applications, that includes, in
particular, \cite{B96,SLK08,BS09,BL09,GL13a,GL13b,BMP13,CST13,CZZS13}. A
recent survey of models and exact methods for portfolio selection tasks,
including cases with cardinality constraints, is provided by~\cite{MDA19};
another fairly broad overview of MIQP with cardinality constraints can be
found in~\cite{ZSLS14}. A MIQP algorithm for the special case of feature
selection (or sparse regression), \cardcons{$\norm{Ax-b}_2$}{$k$}{$\R^n$},
was proposed in~\cite{BKM16}, including the aforementioned ways to compute
tighter big-M bounds; some statistical properties of such sparse regression
problems and relations to their regularized versions are discussed in,
e.g., \cite{ZZ12,SPZZ13}. Other tweaks of the straightforward big-M MIQP
approach are discussed in~\cite{KCF18} (see also \cite{AG19}). Introducing
a ridge regularization term to the regression objective, \cite{BVP17}
recast the problem as a binary convex optimization problem and propose an
outer-approximation solution algorithm that scales to large dimensions, at
least for sufficiently small~$k$. A different (big-M free) MIQP formulations is considered in~\cite{XD20}, which also includes an analysis of different relaxation bounds and a numerical comparison with the method from~\cite{BVP17} and some existing and novel heuristics in a large-scale setting. A similarly scalable problem-specific branch-and-bound method for a MIQP model of the corresponding regularized problem---i.e., minimizing a weighted objective with an $\ell_2$ data fidelity, an $\ell_0$ cardinality, and a ridge penalty term---is discussed in~\cite{HMR21b}. An extension of this method to group sparsity is described in~\cite{HMR21}, where both exact and approximate solutions are considered. It is also possible to recast cardinality-constrained least-squares problems with ridge penalty as
mixed-integer semidefinite programs (MISDPs), see~\cite{PWEG15,GPU18}, but
those can only be solved exactly for small-scale instances, despite
providing stronger relaxations. For the sparse PCA problem
\cardcons{$x^\top Qx$}{$k$}{$x^\top x=1$}, two exact MIQP solvers were very
recently developed in \cite{DMW18} and \cite{BB19}.

It is also worth mentioning that simple cardinality-constrained problems
with separable objective function $\phi(x)=\sum_{i=1}^{n}\phi_i(x_i)$ and
$X=X_1\times\cdots\times X_n$ with $0\in X_i$ for all $i$, admit a
closed-form solution, see~\cite{LZ13}.

Portfolio optimization seems to be the showcase example for cardinality
constraints. Therefore, in the following, we provide some more details on
the formulation of such problems as MIQPs, along with a few numerical
experiments to shed some light on their practical solution with black-box
MIP solvers.

\subsubsection{Illustrative Example: Cardinality-Constrained Portfolio Optimization Problems}\label{sec:exactopt:cardcons:example}
The classical Markowitz mean-variance optimization problem (cf.~\cite{M52})
can be described as follows:
\begin{equation}\label{harry}
  \min\big\{\lambda x^\top Q x - \bar{\mu}^\top x\,:\,A x \geq b\big\}.
\end{equation}
%\begin{subequations} \label{harry}
%\begin{align}
%  \min\quad&\lambda x^\top Q x \ - \ \bar \mu^\top x \label{portobj}\\
%  \st & A x \geq b.
%\end{align}
%\end{subequations}
Here, $x \in \R^n$ is a vector of \textit{asset positions},
$Q\in\R^{n \times n}$ is the \textit{sample covariance matrix of asset returns},
$\bar \mu \in \R^n$ the \textit{vector of average asset returns},
$\lambda \geq 0$ is a \textit{risk-aversion multiplier}, and $A x \geq b$
are generic linear \textit{portfolio construction requirements}. The
objective of~\eqref{harry} represents a tradeoff between risk and portfolio
performance.  In the simplest form, the linear requirements for feasible
portfolios are
\begin{subequations} \label{typical}
  \begin{align}
    & \sum_{j = 1}^n x_j \, = \, 1, \label{convexity}\\
    & x \ge 0. \label{longonly}
    \end{align}
\end{subequations}    
In this case, $x_j \geq 0$ represents the percentage of a portfolio
invested in an asset $j$.

Modern versions of problem \eqref{harry} incorporate features that require
binary variables. There is a large literature that addresses such features,
see, e.g. \cite{B96,BS09,CST13}. A critical feature is a
\textit{cardinality constraint} on the number of positions to be taken,
e.g., an upper bound on the number of nonzero~$|x_j|$. Here, we detail
typical portfolio optimization/management constraints along with their
respective practical motivation and (numerical) aspects to consider when
building and solving such models. Moreover, we discuss some experiments
using a recent version of the commercial MI(Q)P solver Gurobi~\cite{gurobi}
on formulations that incorporate several modern features, using real-world
data.
\begin{itemize}
\item {\bf Long-short portfolios.} In the modern practice, an asset $j$ can
  be ``long'', ``short'' or ``neutral'', represented, respectively, by
  $x_j > 0$, $x_j< 0$ or $x_j= 0$. We can write, for any asset $j$,
  $x_j = x^+_j - x^-_j$,
  %\begin{align*}
  %  & x_j \ = \ x^+_j - x^-_j, %\label{finalj}
  %\end{align*}
  with the (important) proviso that $x^+_j x^-_j = 0$. This complementarity
  constraint provides an example of the use of binary variables, as
  discussed earlier: The problem will always be endowed with upper bounds
  on $x^+_j$ and $x^-_j$; denote them by $u^+_j$ and $u^-_j$,
  respectively. Then, to ensure $x^+_j x^-_j = 0$, we write
  \begin{align}
    & x^+_j \le u^+_j y_j, \quad x^-_j \le u^-_j(1 - y_j), \quad y_j \in \{0,1\}. \label{selectsidej}
  \end{align}
  A portfolio manager may also seek to limit the total exposure in the long
  and short side. This takes the form of respective constraints
  \begin{align*}
    & L^+ \, \le \, \sum_{j = 1}^n x^+_j \, \le \, U^+ \quad \text{and} \quad
      L^- \, \le \, \sum_{j = 1}^n x^-_j \, \le \, U^-, %\label{exposures}
  \end{align*}
  for appropriate nonnegative quantities $L^\pm$ and $U^\pm$. %$L^+, U^+, L^-$, and $U^-$.  
  A constraint of the form \eqref{convexity} does not make sense in a
  long-short setting; instead one can impose
  $\sum_{j = 1}^n( x^+_j + x^-_j) = 1$ (together with
  \eqref{selectsidej}). Additionally, one may impose upper and lower bounds
  on the ratio between the total long and short exposures.

\item {\bf Portfolio update rules.} In a typical portfolio management
  setting, a portfolio is being updated rather than constructed ``from
  scratch''.  Each asset $j$ has an \textit{initial position} $x^0_j$ which
  could itself be long, short or neutral. Thus, we can write
  \begin{align*}
    & x_j \ = \ x^0_j + \delta^+_j - \delta^-_j, %\label{movej}
  \end{align*}
  where $\delta^+_j \ge 0$ and $\delta^-_j \ge 0$ are the changes in the
  long and short direction, respectively. These quantities may themselves
  be (individually) upper- and lower-bounded, and the same may apply to the
  sums $\sum_{j = 1}^n \delta^+_j$ and $\sum_{j = 1}^n \delta^-_j$.

\item {\bf Cardinality constraints.} As stated above, a typical requirement
  is to place an upper bound on the number of nonzero positions $x_j$. We
  can effect this through the use of binary variables, by repurposing the
  binary variable $y_j$ introduced above, introducing a new binary variable
  $z_j$, and imposing
  \begin{align*}
    & x^+_j \leq u^+_j y_j, \quad x^-_j \leq u^-_j z_j, \quad z_j \leq 1 - y_j, \quad z_j \in \{0, 1\}, \\
    & \sum_{j = 1}^n (y_j + z_j) \ \leq \ k,
  \end{align*}
  where $k > 0$ is the upper bound on the number of nonzero positions.
  However, this is not the only case a cardinality constraint may be
  needed. Such rules may also apply, for example, to specific subsets of
  assets (e.g., within a certain industrial sector).

\item {\bf Threshold rules.}  When $n$ is large and $\lambda$ is large, the
  standard mean-variance problem may produce portfolios that include assets
  in minute quantities. Hence, a manager may seek to enforce a rule that
  states that an asset is either neutral (i.e. $x_j = 0$) or takes a
  position that is ``large enough'', resulting in so-called semi-continuous
  variables. We can reuse the binary variables just described, for this
  purpose: For any asset $j$, we constrain
  \begin{align*}
    & x^+_j \geq \theta^+_j y_j \quad \mbox{and} \quad x^-_j \geq \theta^-_j z_j,
  \end{align*}
  where $\theta^+_j$ and $\theta^-_j$ are the respective threshold values.
  In addition, we may apply similar rules to the $\delta^\pm$ quantities
  introduced above (so as to deter unnecessary movements).

\item {\bf Reduced-rank approximations of the sample covariance matrix $Q$.}
  Typical covariance matrices arising in portfolio management have high
  rank (usually, full rank) but with many tiny eigenvalues. In fact, the
  spectrum of such matrices displays the usual ``real-world'' behavior of
  rapidly declining eigenvalues. For instance, if $n = 1000$ (say), only
  the top $200$ eigenvalues may be significant, and of those, the top $50$
  will dominate. Usually the top eigenvalue is significantly larger than
  the second largest, and so on.
   
  Let us consider the spectral decomposition of $Q = V \Omega V^\top$,
  where $V$ is the $n \times n$ matrix whose columns are the eigenvectors
  of $Q$ and $\Omega = \diag (\omega_1, \omega_2, \ldots, \omega_n)$ is the
  diagonal matrix holding the respective eigenvalues
  $\omega_1 \geq \omega_2 \geq \ldots \geq \omega_n \ (\geq 0)$.  We can
  then approximate
  \begin{align}
    & Q \approx \sum_{i = 1}^H \omega_i v_i v_i^\top, \label{Qapprox}
  \end{align}
  where $v_i$ is the $i$-th eigenvector of $Q$ and $H \leq n$ is
  appropriately chosen. The primary reason (as seen by practitioners) for replacing the sample covariance
  matrix with the approximation in \eqref{Qapprox} is that by doing so one removes ``noise'', i.e., that one obtains a better
  representation of the ``true'' underlying covariance matrix.

  Denoting by $V^H$ the $H \times n$ matrix whose
  $i$-th row (for $1 \leq i \leq H$) is $v_i^\top$, the objective of
  problem \eqref{harry} can be written as
  \begin{align}
    \min\quad&\lambda \sum_{i = 1}^H \omega_i f^2_i - \bar{\mu}^\top x, \label{portobj2}
  \end{align}
  where the $f_i$ are new variables, subject to the constraint $V^H x=f$. 
  %\begin{align*}
  %  & V^H x  =  f.
  %\end{align*}
  As stated above, the choice of $H$ hinges on how quickly the eigenvalues
  $\omega_i$ decrease. A prematurely small choice for $H$ may result in a
  poor approximation to $Q$, and a large choice yields a formulation with
  very small parameters. One can overcome these issues by relying on the
  \textit{residuals}, that is to say the quantities
  \begin{align}
    & \rho_j \ \define \left(Q - \sum_{i = 1}^H
      \omega_i v_i v_i^\top\right)_{jj} = \ Q_{jj} - \sum_{i = 1}^H \omega_i v_{ij}^2, \quad  j \in [n].
  \end{align}  
  These quantities are nonnegative since
  $Q - \sum_{i = 1}^H \omega_i v_i v_i^\top = \sum_{i = H+1}^n \omega_i v_i
  v_i^\top $ is positive semidefinite. Moreover, in general, the $\rho_j$
  should be small if the approximation \eqref{Qapprox} is good (namely,
  since $\sum_{i = H+1}^n \omega_i v_i v_i^\top$ is diagonal-dominant). The
  residuals can be used to update the objective \eqref{portobj2} as
  follows:
  \begin{align}
    \min\quad&\lambda \sum_{i = 1}^H \omega_i f^2_i \, + \lambda \sum_{j = 1}^n \rho_j x_j^2 \ \ - \ \bar \mu^\top x. \label{portobj3}
  \end{align}  
  Some of the $\rho_j$ may be extremely small---in such a case, it is
  numerically convenient to replace them with zeros.
\end{itemize}

\subsubsection{Portfolio Optimization Example: Experiments}\label{sec:exactopt:cardcons:experiments}
We next outline the results on a challenging instance with the following
attributes:
\begin{itemize}
\item $n = 741$ with data from the Russell 1000 Index (made publicly
  available by the authors of \cite{BCW18})
\item Full covariance matrix
\item Cardinality limit $k=50$ with all threshold values set at
  $\theta^+_j = \theta^-_j = 0.01$
\item Long-short model with maximum and minimum long exposures set at
  $U^+=0.5$ and $L^+=0.4$, respectively, and maximum short exposure set at
  $U^-=0.2$ (and no minimum short exposure, i.e., $L^-=0$)
\end{itemize}
The above formulation was run using Gurobi 9.1.1 on a machine with 20
physical cores (Intel Xeon E5-2687W v3, 3.10\,GHz) and 256\,GB of
RAM. Table~\ref{table:gurobitest13} summarizes the observed performance
using default settings.

\begin{table}[t]
  \centering      
  \caption{\footnotesize\itshape Behavior of solver on cardinality-constrained portfolio optimization problem.}\label{table:gurobitest13}
  \footnotesize % SIAM style guide sets tables in 8pt...
  \begin{tabular*}{\textwidth}{@{\extracolsep{\fill}}@{\quad}rr@{\qquad}rrr@{\quad}}
    \toprule                  
    nodes & incumbent & best bound & gap & runtime [s] \\
    \midrule
    0        &  25.40531  &  0.53332 & 97.9$\%$  &    14 \\
    0        &   0.86096  &  0.53332 & 38.1$\%$  &    18 \\
    15       &   0.78644  &  0.54223 & 31.1$\%$  &    25 \\
    5611     &   0.72123  &  0.55025 & 23.7$\%$  &    92 \\
    11838    &   0.71995  &  0.55025 & 23.6$\%$  &   130 \\
    106774   &   0.71904  &  0.57292 & 20.3$\%$  &   500  \\
    242365   &   0.71898  &  0.57873 & 19.5$\%$  &  1000 \\
    524934   &   0.71898  &  0.58455 & 18.7$\%$  &  2002 \\
    817290   &   0.71898  &  0.58788  & 18.2$\%$ &  3002 \\
    994932   &   0.71898  &  0.58946  & 18.0$\%$ &  3605 \\
    \bottomrule
  \end{tabular*}
\end{table}

Let us consider now the outcome when we run the same portfolio optimization
problem, but now using the approximation to the covariance matrix obtained
by taking the top $H = 250$ modes. In this case, the top eigenvalue equals
$8.10 \times 10^{-2}$ %$8e-2$
while the $250^\text{th}$ is approximately $1.57\times
10^{-4}$. Table~\ref{table:gurobitest13_pca_250} summarizes the results.

Each table provides relevant statistics concerning the corresponding run;
the rows were selected to highlight significant steps within the run (e.g.,
discovery of a new incumbent or improvement of the best lower bound) so as to
provide the reader with a qualitative understanding of the progress made by
the solver.

\begin{table}[t]
  \caption{\footnotesize\itshape Behavior of solver on approximation to instances in Table \ref{table:gurobitest13} obtained by using $H = 250$ modes.}
  \label{table:gurobitest13_pca_250}
  \centering
  \footnotesize % SIAM style guide sets tables in 8pt...
  \begin{tabular*}{\textwidth}{@{\extracolsep{\fill}}@{\quad}rr@{\qquad}rrr@{\quad}}
    \toprule              
    nodes & incumbent & best bound & gap & runtime [s] \\
    \midrule
    0    & 4.98030  & -0.00080  &  100.0$\%$   &   6 \\
    217  & 2.08509  &  0.00009  &  100.0$\%$   &  39 \\
    562  & 0.69882  &  0.00009  &  100.0$\%$   &  52 \\
    924  & 0.31362  &  0.00009  &  100.0$\%$   &  80 \\
    1048 & 0.31362  &  0.31293  &    0.2$\%$   & 105 \\           
    1067 & 0.31362  &  0.31362  &    0.0$\%$   & 116 \\           
    \bottomrule
  \end{tabular*}
\end{table}

In the first case, we ended the solving run after approximately one hour of
elapsed time with a remaining optimality gap of about 18\,\%. In the second
case (cf.~Table~\ref{table:gurobitest13_pca_250}), the solver is able to
close the gap so as to attain optimality within tolerance after about two
minutes. We stress that such reduced-rank problems are not always
significantly easier than their full-rank counterparts, but, overall, they
prove more practicable on average. The objective value of a solution as per
the rank-reduced problem is a lower bound for its value in the true problem
(since we are ignoring positive terms in the spectral expansion of the
covariance matrix), but beyond this simple statement, an accurate
estimation of how close this lower bound actually is can be nontrivial.

More importantly, a portfolio manager would prefer a rank-reduced
formulation because the modes being ignored are quite small and hence may
\emph{seem} negligible. However, it is important to note that this reasoning does
not amount to a mathematically correct statement. Indeed, note that the
best lower bound after one hour in Table~\ref{table:gurobitest13} is
notably larger than the optimal value of the approximated problem in
Table~\ref{table:gurobitest13_pca_250}.

An additional and important aspect of this discussion that we are not
addressing is the practical impact on portfolio management that a
reduced-rank representation will have. A portfolio manager will not simply
be interested in solution speed---rather, the performance of the resulting
portfolio is of great interest. This point is significant in the sense that
the covariance matrix $Q$, and, of course, the spectral decomposition
$Q = V \Omega V^\top$, are data-driven. Both objects are bound to be very
``noisy'', for lack of a better term. It can be observed that (in
particular) the leading modes of $Q$ (i.e., the columns of the matrix
$V^H$) can be quite noisy. A closely related issue concerns the number of
modes $H$ to rely on. The proper way to handle such noise is by applying
some form of robust optimization, see \cite{GI03}, but an in-depth analysis
of these topics is outside the scope of the present work.

\subsection{Cardinality Regularization Problems}\label{sec:exactopt:cardreg}
There appears to be hardly any literature focusing specifically on the
\emph{exact} solution of regularized cardinality mini\-mization problems,
\cardreg{$\rho$}{$X$}. As mentioned earlier, \cite{FMPSW18} contains
theoretical optimality conditions (but no exact algorithm) for
\mbox{\cardreg{$\tfrac{1}{\gamma}f(x)$}{$g(x)=0,h(x)\leq 0$}} with
$\gamma>0$ and continuously differentiable functions $f:\R^n\to\R$,
$g:\R^n\to\R^p$, and $h:\R^n\to\R^q$. Similarly, \cite{LZ13} considers such
problems allowing for additional constraints that form a closed convex set
$X$. They also show that for separable~$f$ and constraints representable as
$X_1\times\cdots\times X_n$ with $0\in X_i$ for all~$i$, the
$\ell_0$-regularized problem admits a closed-form solution. A statistical
discussion of (solution properties of) least-squares regression with
cardinality regularization,
\cardreg{$\tfrac{1}{2\lambda}\norm{Ax-b}_2^2$}{$\R^n$}, as well as other
concave regularizers, can be found in \cite{ZZ12} (albeit without
algorithmic results) and suggests that from a statistical perspective,
cardinality-constrained least-squares regression is preferable to its
cardinality-regularized variants.

In the cosparsity model, the problem of one-dimensional ``jump-penalized'' least-squares segmentation,
\[
\min\tfrac{1}{2}\norm{b-x}_2^2 + \lambda\norm{Bx}_0,
\]
where~$B$ is the difference operation (so that $\norm{Bx}_0=\sum_{i=2}^{n}\chi_{\{x_i\neq x_{i-1}\}}$) and $\lambda>0$, can be solved in polynomial time, see \cite{J13,BKLMW09} and references therein. Similarly, if~$B$ encodes more general adjacency relations between entries of~$x$,
\[
\min\big\{\tfrac{1}{2}\norm{b-x}_2^2 + \lambda\norm{Bx}_2^2 +\mu\norm{x}_0\,:\, x\geq 0\big\}
\]
(with $\lambda$, $\mu>0$) admits a polynomial-time solution, while for the variant with a cardinality constraint $\norm{x}_0\leq k$ instead of the second regularization term, no such methods are known and MIQP techniques can be applied, see~\cite{AGH21} and the previous works detailed therein.

Generally, the techniques from Section~\ref{sec:exactopt:cardmin} are
applicable to regularized cardinality problems as well: If, for instance,
auxiliary binary variables $y\in\{0,1\}^n$ are used to reformulate
$\norm{x}_0$ as $\ones^\top y$ (coupling $x$ and $y$ via, e.g., big-M
constraints), one can simply integrate the regularization term into the new
objective $\ones^\top y+\rho(x)$.  Depending on the concrete choice of the
function $\rho$, mixed-integer linear or nonlinear programming can then be
applied analogously.

Similarly, some techniques from Section~\ref{sec:exactopt:cardcons}
\emph{might} also be applicable in the regularization context after
reformulating \cardreg{$\rho$}{$X$} as
\begin{equation*}
    \min_{t,x}\left\{t\,:\,\norm{x}_0+\rho(x)\leq t,~x\in X,~t \geq 0\right\}.
\end{equation*}
However, $\norm{x}_0 \leq t-\rho(x)$ is obviously not a classical
cardinality constraint, as the right-hand side also depends on $x$ and
$t \geq 0$ is a variable. We are not aware of any work investigating this
type of mixed constraint.

Finally, \cite{DCL15} considers \cardreg{$\rho$}{$X$} and derives an
exponential-size convexification through disjunctive programming. Based on
this convexification, the authors propose a class of penalty functions
called ``perspective penalties'' that are the counterpart of the
perspective relaxation well-known in the mixed-integer nonlinear context
\cite{GLind08}. Computational experiments comparing various lower bounds
(including those from \cite{PWEG15}) are discussed. A similar convexification
attempt is considered in~\cite{WGK20}, but this time applied to the setting in
which the cardinality is explicitly modeled via
$x_i(1 - y_i) =0~\forall i \in [n]$ with $y$ binary. The convexification is
then obtained by exploiting the interplay between non-separable convex
objectives and combinatorial constraints on the indicator variables; some
computational results on real-world datasets are reported as well.

\section{Relaxations and Heuristics}\label{sec:inexactopt}
Most of the exact solution methods mentioned in the last section make use
of problem-specific heuristics and/or efficient ways to solve the
encountered relaxations. Indeed, incorporating such components into
dedicated MIP and MINLP algorithms (along with other aspects like
propagation and branching rules or cutting planes) can drastically improve
performance compared to black-box approaches with general-purpose solvers,
see, e.g., \cite{BPVP17,BB19,T19}, or the example for \cardmin{$Ax=b$} in
Section~\ref{sec:card_min_example}.

Additionally, heuristic methods are of interest in their own rights, as
they are often (at least empirically) able to provide good-quality
solutions in fractions of the sometimes considerable runtime an exact mixed-integer
programming approach may take, and therefore also open the possibility---or
sometimes the only reasonable way---to tackle very high-dimensional,
large-scale instances (see also Section~\ref{sec:scalability}).

Thus, in this section, we attempt to survey the countless heuristics,
relaxation and approximation methods proposed for the various cardinality
problems discussed in Section~\ref{sec:concreteprobs}. We begin with the well-known $\ell_1$-norm surrogate for the cardinality in Section~\ref{sec:inexactopt:surrogates}, devoting subsections to the reasons for its success and the many different (classes of) algorithms that have been proposed for various $\ell_1$-problems. Due to the sheer number of results, variations, and
improvements, in Section~\ref{sec:inexactopt:surrogates:theory} we limit ourselves to some key results that exhibit the general flavor of so-called ``recovery guarantees'' and introduce some
of the most important concepts. Then, in Section~\ref{sec:inexactopt:surrogates:algorithms}, we give an extensive (though likely still not exhaustive) overview of algorithmic approaches to
solve $\ell_1$-minimization problems; earlier, but less comprehensive,
overviews can also be found in, e.g., \cite{FR13,BJMO12,LPT15}. 
Moving beyond $\ell_1$, we subsequently discuss the more general concept of atomic norms (Section~\ref{sec:inexactopt:surrogates:atomic}) and further approximations of cardinality objectives (Section~\ref{sec:inexactopt:approximation_objective}) and constraints (Section~\ref{sec:inexactopt:relaxation_constraints}). Finally, Section~\ref{sec:inexactopt:greedy_heuristics} survey greedy-like and miscellaneous other heuristics.
We remark that readers who are already very familiar with $\ell_1$-norm theory and algorithms might want to skip Section~\ref{sec:inexactopt:surrogates} and may find the results/tools of the later sections, some of which are fairly new and/or perhaps less known, more useful.

Note that the polyhedral results mentioned in the previous section, i.e.,
valid inequalities for various kinds of cardinality problems, can be viewed
as a means to strengthen the respective LP (or other) relaxations, and
could quite possibly be combined with many heuristic- and/or
relaxation-based approaches. For brevity, we do not repeat the pointers to
the literature in this context. Such integration possibilities appear to
have been largely overlooked thus far, and might offer an interesting
avenue for future refinements of existing inexact models and algorithms.

\subsection{$\ell_1$-Norm Surrogates: Basis Pursuit, LASSO, etc.}\label{sec:inexactopt:surrogates}

The most popular relaxation technique replaces the so-called $\ell_0$-norm
by the ``closest'' convex real norm---the $\ell_1$-norm.  Indeed, it is
easily seen that
\[
  \lim_{p\searrow 0} \norm{x}_p^p = \lim_{p\searrow 0} \sum_{i=1}^{n} \abs{x}^p = \norm{x}_0.
\]
This sentiment along with empirical observations led to the wide-spread use
of the $\ell_1$-norm as a tractable surrogate to promote sparsity, and has
since been underpinned with various theoretical results on when such
approaches work correctly, see, e.g.,~\cite{FR13} for an overview of
breakthrough results in the field of compressed sensing.

In sparse regression, compressed sensing, and statistical estimation, the
following incarnations of such $\ell_1$-based problems are encountered most
often:
\begin{align}
  \min\quad &\norm{x}_1 \st Ax=b,~x\in X; \tag{\mbox{BP($X$)}}\\
  \min\quad &\norm{x}_1 \st \norm{Ax-b}_2\leq\delta,~x\in X; \tag{\mbox{BPDN($\delta,X$)}}\\
  \min\quad &\tfrac{1}{2}\norm{Ax-b}_2^2 \st \norm{x}_1\leq\tau,~x\in X; \tag{\mbox{LASSO($\tau,X$)}}\\
  \min\quad &\norm{x}_1 + \tfrac{1}{2\lambda}\norm{Ax-b}_{2}^{2} \st x\in X. \tag{\mbox{$\ell_1$-LS($\lambda,X$)}}
\end{align}
The \emph{basis pursuit} problem BP($X$) was first discussed and proven to
provide sparse solutions for underdetermined linear equations
in~\cite{CDS98}. Usually, $X=\R^n$ here, but the nonnegative ($X=\R^n_+$),
bounded ($\ell\leq x\leq u$), complex ($X=\C^n$), or integral
($X\subseteq\Z^n$) settings have also been investigated, see, e.g.,
\cite{DT05,FR13,LPST16,KKLP17}. The \emph{basis pursuit denoising} problem
BPDN($\delta,X$) extends the noise-free model BP($X$) by allowing
deviations from exact equality and thus providing robustness against
measurement noise as well as the possibility to achieve even sparser
solutions. As for the original cardinality minimization problem, other
norms than the $\ell_2$-norm have been considered for the constraints,
e.g., the $\ell_\infty$-norm in~\cite{BLT18} or the $\ell_1$-norm
in~\cite{JR15}. The \emph{least absolute shrinkage and selection operator}
LASSO($\tau,X$) was motivated in a regression context as a way to improve
prediction accuracy and interpretability by promoting shrinkage (and thus,
ultimately, sparsity) of the predictor variables~\cite{T96}; it can be seen
as the $\ell_1$-approximation to the cardinality-constrained least-squares
problem. Finally, the \emph{$\ell_1$-regularized least-squares} problem
$\ell_1$-LS($\lambda,X$) is often employed as well, especially if no
immediate bounds $\delta$ or $\tau$ for the related BPDN or LASSO problems
are known, and because it is an unconstrained problem (provided $X=\R^n$)
and thus potentially can be solved even more efficiently. In fact, in contrast to the associated $\ell_0$-based problems (cf.~Prop.~\ref{prop:cardprobs_relations}), it is known that BPDN($\delta,\R^n$), LASSO($\tau,\R^n$) and
$\ell_1$-LS($\lambda,\R^n$) are always \emph{equivalent} for certain values of the
parameters $\delta$, $\tau$ and $\lambda$ (see, e.g., \cite{VDBF08}),
although the precise values for which this holds are data-dependent and
generally unknown a priori. The recent work~\cite{BPY21} analyzes the stability of these programs w.r.t.\ parameter choices in a denoising setting with $A=I$, and indicates that the regularized version behaves most robustly. In all these problems, typically $X=\R^n$,
though like for BP($X$), other constraints are occasionally considered as
well.

Two more $\ell_1$-minimization problem variants that have turned out to be
of special interest in some applications are the so-called \emph{Dantzig
  Selector}~\cite{CT07}
\begin{equation}
  \min\quad \norm{x}_1 \st \norm{A^\top(Ax-b)}_\infty\leq\varepsilon,~x\in X, \tag{\mbox{DS($\varepsilon,X$)}}
\end{equation}
whose cardinality-minimization counterpart was proposed in~\cite{MR17}, and
the Tikhonov/ ridge-regularized $\ell_1$-LS problem
\begin{equation}
  \min\quad \norm{x}_1 + \tfrac{1}{2\lambda_1}\norm{Ax-b}_{2}^{2} +
  \tfrac{\lambda_2}{\lambda_1}\norm{x}_2^2 \st x\in X, \tag{\mbox{EN($\lambda_1,\lambda_2,X$)}}
\end{equation}
known as the \emph{elastic net}~\cite{ZH05}. The additional ridge penalty
here ensures strong convexity of the objective function and, consequently,
uniqueness of the minimizer. Like $\ell_1$-LS($\lambda,X$),
EN($\lambda_1,\lambda_2,X$) has been used in several applications such as
portfolio optimization~\cite{BPVP17} or support vector machine
learning~\cite{WZZ06}.

In the following, we first provide a very brief overview of
the theoretical success guarantees that led to the popularity of
$\ell_1$-formulations, and then discuss algorithms.

\subsubsection{Dipping a Toe Into Why $\ell_1$-Reformulations Became Popular}\label{sec:inexactopt:surrogates:theory}

Nowadays, using the $\ell_1$-norm as a tractable surrogate for the
cardinality is commonplace. This rise to popularity was in large part
fueled by the advent of compressed sensing, the signal processing paradigm
that reduces measurement acquisition efforts at the cost of more complex
signal reconstruction. Low cardinality of signal vectors (i.e., sparsity)
has proven to be key for solving the nontrivial recovery problems, and
\cite{CDS98} laid essential groundwork in demonstrating that the
$\ell_1$-surrogate offers a viable and efficient alternative to the ``true
sparsity'' represented by the $\ell_0$-norm and associated \NP-hard
reconstruction tasks. Most of the earliest sparse recovery research focused
on the problems BP($\R^n$) and BPDN($\delta,\R^n$), so for the sake of
exposition, we highlight them and their $\ell_0$ counterparts
\cardmin{$Ax=b$} and \cardmin{$\norm{Ax-b}_2\leq\delta$} here, too.  The
interesting setup in compressed sensing has $A\in\R^{m\times n}$ with
$\rank(A)=m<n$, so the system $Ax=b$ is underdetermined and has infinitely
many solutions. Sparsity is key to overcome this ill-posedness by allowing,
in principle, the exact reconstruction of sufficiently sparse signals as
the respective \emph{unique} sparsest solutions to $Ax=b$, i.e., unique
optimal solutions of \cardmin{$Ax=b$}. As mentioned earlier
(cf.\ Section~\ref{sec:concreteprobs:signalprocessing}), this uniqueness
requires that the signal cardinality is smaller than $\spark(A)/2$,
regardless of any algorithm being used to solve the actual reconstruction
problem, see, e.g., \cite{DE03}.

Although proven to be \NP-hard only much later in~\cite{TP14}, computing
the spark was deemed intractable early on, and since cardinality
minimization problems like \cardmin{$\norm{Ax-b}_2\leq\delta$} were already
known to be \NP-hard as well (cf.\ \cite{GJ79,N95}), the focus quickly
shifted to alternative conditions that ensure sparse solution uniqueness
and/or reconstruction error analysis of surrogate methods, in particular
$\ell_1$-minimization.

The best-known such \emph{recovery conditions} can all be formulated using
a few key matrix parameters, namely:
\begin{itemize}
\item The \emph{mutual coherence} of a matrix $A$,
  \[
    \mu(A)\define \max_{\begin{subarray}{c}i,j\in[n],\\i\neq j\end{subarray}} \frac{\abs{A_i^\top A_j}}{\norm{A_i}_2\,\norm{A_j}_2}.
  \]
\item The order-$k$ \emph{nullspace constant} ($k$-NSC) of a matrix $A$,
  \[
    \alpha_k\define
    \max_{x,S}\left\{\,\norm{x_S}_1\,:\,Ax=0,~\norm{x}_1=1,~S\subseteq[n],~\abs{S}\leq k\,\right\}.
  \]
\item The order-$k$ \emph{restricted isometry constant} ($k$-RIC) of a matrix $A$,
  \[
    \delta_k\define\min_x\left\{\,\delta\,:\,(1-\delta)\norm{x}_2^2\leq\norm{Ax}_2^2\leq(1+\delta)\norm{x}_2^2~\forall
      x\text{ with }1\leq\norm{x}_0\leq k\,\right\}.
  \]
\end{itemize}
For a $k$-sparse solution $\hat{x}$ of $Ax=b$, these parameters can all be
used to certify uniqueness of $\hat{x}$ as the sparsest solution, for
instance via $2k<1+1/\mu(A)^2$ \cite{LZP17}, $\alpha_k<1/2$ \cite{DH01}, or
$\delta_{2k}<1$ \cite{CT05,C08}---indeed, these conditions each imply
$k<\spark(A)/2$. Yet more interestingly, these parameters also yield
recovery conditions for $\ell_1$-minimization and other (algorithmic)
approaches, i.e., they can be used to ensure correctness of the solutions
to surrogate problems (possibly up to certain error bounds) w.r.t.\ the
original sparsity target. In particular, \emph{uniform sparse recovery
  conditions (SRCs)} such as incoherence (small enough $\mu(A)$), the
nullspace property (NSP) or the restricted isometry property (RIP) ensure
$\ell_0$-$\ell_1$-equivalence for all $k$-sparse vectors, i.e., that the
solution to BP($\R^n$) is unique and coincides with the unique solution of
\cardmin{$Ax=b$}. Besides uniform SRCs, there are also various
\emph{individual SRCs} that establish uniqueness of, e.g.,
BP($\R^n$)-solutions for specific $\hat{x}$. The most powerful one in this
regard is sometimes called \emph{strong source condition}~\cite{GHS11}:
$\hat{x}$ is the unique optimal solution to BP($\R^n$) if and only if
$\rank(A_{\supp(\hat{x})})=\norm{\hat{x}}_0$ and there exists a vector~$w$
with $A^\top w\in\partial\norm{\hat{x}}_1$ such that $\abs{(A^\top w)_j}<1$
for all $j\notin\supp(\hat{x})$. (This condition can be derived via
first-order optimality and complementary slackness.) Similarly to uniform
SRCs, the strong source condition has its analogues for other
$\ell_1$-problems like BPDN($\delta,\R^n$), $\ell_1$-LS($\lambda,\R^n$),
LASSO($\tau,\R^n$), and their ``analysis/cosparse counterparts'' (with
$\norm{Bx}_1$ instead of $\norm{x}_1$ for some matrix $B$), see,
e.g.,~\cite{ZYY16}.

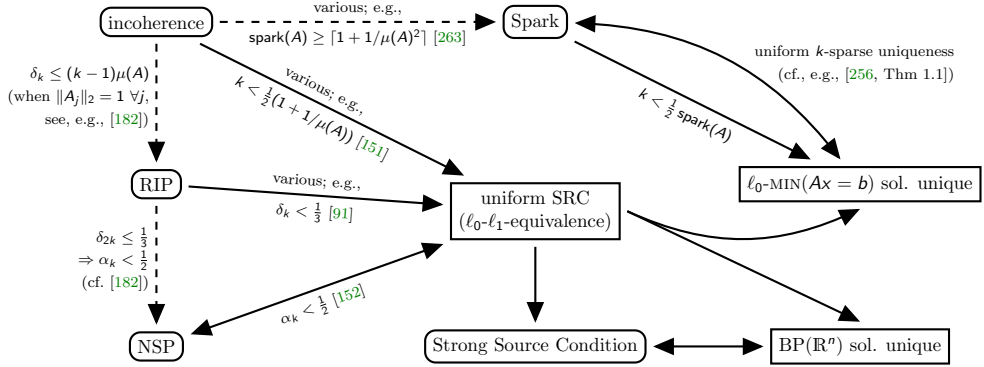
\begin{figure}[t]
  \centering
  \vspace*{2em}
    \begin{tikzpicture}[scale=0.7196, transform shape, >=triangle 45]
      \tikzstyle{biarc}=[<->,draw=black!100,thick]
      \tikzstyle{textnode}=[draw=black!100,thick,inner sep=4pt, outer sep=4pt, minimum width = 15pt]
      \tikzstyle{arc}=[->,draw=black!100,thick]
      \node[textnode, rounded corners] (nsp) at (0,3) {NSP};
      \node[textnode, rounded corners] (rip) at (0,6) {RIP};
      \node[textnode, rounded corners] (inc) at (0,9) {incoherence};
      \node[textnode, rounded corners] (s) at (7,9) {Spark};
      \node[textnode, rounded corners] (ssc) at (7,3) {Strong Source Condition};
      \node[textnode] (p0) at (13,6) {\cardmin{$Ax=b$} sol. unique};
      \node[textnode] (p1) at (13,3) {BP($\R^n$) sol. unique};
      \node[textnode, align=center] (uni) at (7, 5.5) {uniform SRC\\($\ell_0$-$\ell_1$-equivalence)};
      \draw[biarc](nsp) --node[below, sloped, pos=0.5] {\footnotesize{$\alpha_k<\tfrac{1}{2}$ \cite{DH01}}}  (uni); 
      \draw[arc](rip) --node[above, sloped, pos=0.5]{\footnotesize{various;
        e.g.,}} node[below, sloped, pos=0.5]{\footnotesize{$\delta_k<\tfrac{1}{3}$ \cite{CZ13}}} (uni);
      \draw[arc,dashed,thick](rip) to node[left, align=right, pos=0.45]{\footnotesize{$\delta_{2k}\leq\tfrac{1}{3}$}\\ \footnotesize{$\Rightarrow\alpha_k<\tfrac{1}{2}$}\\\footnotesize{(cf.~\cite{F10})}} 
      (nsp);
      \draw[arc](inc) --node[above, sloped, pos=0.45]{\footnotesize{various; e.g.,}} node[below, sloped, pos=0.45]{\footnotesize{$k < \tfrac{1}{2}(1+1/\mu(A))$ \cite{DE03}}} (uni);
      \draw[arc,dashed,thick](inc) --node[left, align=right,pos=0.45]{\footnotesize{$\delta_k\leq (k-1)\mu(A)$}\\\footnotesize{(when $\norm{A_j}_2=1~\forall j$,}\\\footnotesize{see, e.g.,~\cite{F10})}} (rip);
       \draw[arc,dashed,thick](inc) --node[above]{\footnotesize{various;
           e.g., }} node[below]{\footnotesize{$\spark(A)\geq \lceil 1 + 1/\mu(A)^2\rceil$ \cite{LZP17}}} (s);
       \draw[arc](s) --node[below, sloped]{\footnotesize{$k < \tfrac{1}{2}\spark(A)$}} (p0);
       \draw[biarc](s) to [bend left=25] node[above, align=right, pos=0.7]
       {\footnotesize{$\qquad\qquad\qquad\qquad\qquad$uniform $k$-sparse
           uniqueness}\\ \footnotesize{(cf., e.g., \cite[Thm~1.1]{EK12})}} (p0);
       \draw[biarc](ssc) -- (p1);
       \draw[arc](uni) -- (ssc);
       \draw[arc,thick]([yshift=0.0cm]uni.east) to [bend right=25] node[below, sloped]{} (p0.south);
       \draw[arc,thick]([yshift=-0.0cm]uni.east) -- node[above, sloped]{} (p1.north);
    \end{tikzpicture}
    \caption{\footnotesize\itshape Illustration of relations between matrix
      properties and associated SRCs, and their implications w.r.t.\
      recovery of $k$-sparse solutions of \cardmin{$Ax=b$} and
      BP($\R^n$). (Adapted from~\cite{T13}.)}
  \label{fig:recoveryconditionsBP}
\end{figure}

The strongest SRCs are known for basis pursuit, i.e., BP($\R^n$); we
illustrate some such conditions and their relationships in
Figure~\ref{fig:recoveryconditionsBP}. This figure was adapted
from~\cite{T13}, where many more details about recovery conditions are
described, with a focus on BP($\R^n$) and computational
complexity\footnote{Computing the $k$-NSC or $k$-RIC of a matrix is
  \NP-hard, see~\cite{TP14}, whereas the mutual coherence can obviously be
  computed efficiently.}.

Regarding other $\ell_1$-problems, in particular BPDN($\delta,\R^n$) or
$\ell_1$-LS($\lambda,\R^n$), we refer to \cite{FR13} and the concise
summary, derivations, and references therein; the following is an example
of the noise-aware recovery conditions one can find in this context. Let
$\sigma_k(x)_p=\inf\{\norm{x-z}_p : \norm{z}_0\leq k\}$ be the
$\ell_p$-norm error of the best $k$-term approximation of a vector
$x\in\C^n$; this comes into play in situations where $x$ is not exactly
sparse (see also~\cite{CDDV09}). Let a matrix $A\in\C^{m\times n}$ satisfy
\[
  \norm{y_S}_2\leq \frac{\rho}{k^{1/2}}\norm{y_{S^c}}_1+\tau\norm{Ay}_2\qquad\forall y\in\C^n~\forall S\subseteq[n]:\abs{S}\leq k,
\]
where $\rho\in(0,1)$ and $\tau>0$ are constants; this is called the
$\ell_2$-robust nullspace property of order~$k$. Then (see
\cite[Thm.~4.22]{FR13}), for any $\hat{x}\in\C^n$, a solution $x^*$ of
BPDN($\delta$,$\C^n$) with $b=A\hat{x}+e$ and $\norm{e}_2\leq\delta$
recovers the vector~$\hat{x}$ with an $\ell_p$-error ($1\leq p\leq 2$) of
at most
\[
  \norm{x^*-\hat{x}}_p \leq
  \frac{\alpha}{k^{1-1/p}}\sigma_k(\hat{x})_1+\beta k^{1/2-1/p}\delta,
\]
for some constants $\alpha,\beta>0$ that depend solely on $\rho$ and
$\tau$.

Another typical kind of question investigated in compressed sensing
pertains to the number of measurements, i.e., the number of rows of $A$,
that are needed to ensure recovery of $k$-sparse vectors by
$\ell_1$-approaches (and others). The arguments typically use the same
matrix parameters as before, with an apparent focus on restricted isometry
properties of random matrices. Briefly, it can be shown for Gaussian (and
other) random matrices that $\ell_0$-$\ell_1$-equivalence for $k$-sparse
vectors holds with high probabi\-li\-ty if $m\geq C(\delta_k) k\log(n/k)$,
where $C(\delta_k)$ is a constant depending only on the order-$k$ RIC; see
\cite[Chapter~9]{FR13}. Very many similar results establish that for sufficiently many random measurements of some kind, one of the (deterministic) sparse recovery conditions holds with high probability. A slightly different approach is taken in \cite{dAEELKa13}, where it is shown how to relax the standard NSP into one that holds with high probability under certain distributional assumptions on the nullspace (rather than the sensing matrix itself), and relate the task of verifying this condition to classical combinatorial optimization problems.

These types of results laid the foundation for the success of
$\ell_1$-approximations to sparsity, or cardinality terms, and gave rise to
a huge amount of research, both on theoretical improvements and efficient
algorithms for various problem variants. In order to not dilute the focus
of the present paper too much, we do not delve further into the theory
outlined above, and refer to \cite{FR13,EK12,E10,D06} as good starting
points for anyone wishing to dig deeper. We will, nonetheless, complement
the present primer with an overview of the various algorithms for
$\ell_1$-norm optimization problems in the following subsection.

\subsubsection{Algorithmic Approaches to $\ell_1$-Problems}\label{sec:inexactopt:surrogates:algorithms}

A plethora of different solution methods have been applied and specialized
to efficiently handle one or more of the above problems or slight
variations. It is noteworthy that most methods are first-order methods that
often do not need the matrix $A$ to be given explicitly and can instead
work with fast operators implementing matrix-vector products with $A$
and/or $A^\top$. This enables application of such algorithms in large-scale
regimes and special settings where $A$ corresponds to, e.g., a fast Fourier
transform. Moreover, the algorithms can often handle complex data and
variables as well; for simplicity, we again focus only on the real-valued
setting. 

For clarity, we group the different approaches according to the broader
categories they fall into:

\paragraph{\bf Reformulation as LPs or SOCPs}

It is well known that the absolute value function, and thus, by extension,
the $\ell_1$-norm, can be linearized. Hence, any problem involving only
linear and $\ell_1$-norm terms in the objective and the constraints can be
written as a \emph{linear program (LP)}. Similarly, convex $\ell_2$-norm
terms (as in, e.g., BDPN($\delta,\R^n$)) can be reformulated using
second-order cone techniques, yielding \emph{second-order cone programs
  (SOCPs)}. For both these classes, there are well-known standard solution
methods like simplex method variants (for LPs), active-set or interior
point algorithms (for both), see, e.g., \cite{V01,BV04}, with highly sophisticated genera-purpose implementations (e.g., \cite{cplex,gurobi,W96,scip}).

Namely, $\ell_1$-\textsc{magic} (see \cite{CR05}) employs the generic
primal-dual interior point solver from~\cite{BV04} to solve the LP
reformulation
\begin{equation*}
  \min~\ones^\top u \st Ax=b,~-u\leq x\leq u
\end{equation*}
of BP($\R^n$). This approach has also been applied to some related problems
that can be written as LPs, such as DS($\varepsilon,\R^n$), and,
analogously (using another general-purpose log-barrier algorithm
from~\cite{BV04}), to problems such as BPDN($\delta,\R^n$) or
total-variation minimization that can be recast as SOCPs.
Similarly, SolveBP, described in~\cite{CDS98,CDS01}, solves (a
perturbed version of) the alternative LP reformulation of BP($\R^n$) with
variable splits, i.e.,
\begin{equation}\label{eq:l1varsplit}
  \min~\ones^\top x^+ +\ones^\top x^- \st Ax^+ -Ax^-=b,~x^\pm\geq 0,
\end{equation}
by employing the primal-dual log-barrier solver PDCO (based on~\cite{GMPS91}), which can also handle other problems via suitable LP- or SOCP-reformulations, e.g.,
BPDN($\delta,\R^n$).

For $\ell_1$-minimization problems with linear constraints $Ax=b$ or
\mbox{$\norm{Ax-b}_p\leq\delta$} with $p\in\{1,\infty\}$ (as well as some
related problems such as the ``least absolute deviation (LAD)-Lasso''
$\min\{\norm{Ax-b}_1\suchthat\norm{x}_1\leq\tau\}$, cf.~\cite{WLJ07}), it
has been proposed in~\cite{PZVL17} to employ the parametric simplex
method~\cite{D63,V01}. (The earlier version~\cite{PZVL15} of~\cite{PZVL17}
contains more details and applications.) Note that any implementation will
face the typical challenges of a simplex solver---efficient basis updates,
cycling avoidance, etc.---and is therefore nontrivial. For similarities
with and differences to the related homotopy methods discussed below, see
the discussion in~\cite{BLT18}.

The recent contribution~\cite{MWZ19} demonstrates that $\ell_1$-problems,
recast as LPs---in particular, BP($\R^n$) and DS($\tau,\R^n$)---can be
solved very efficiently by using column generation (cf.\ \cite{DDS05}) and
dynamic constraint generation (i.e., cutting planes). The suggested method
initializes the variable and constraint index sets to be included in the
first master problem based on the (efficiently obtainable) solution of the
homotopy method for $\ell_1$-LS($\lambda,\R^n$). Afterwards, violated but
not yet included constraints are identified and added to the model and new
variables are added by solving a classical LP-based pricing problem.
Iterating over the resulting sequence of %(increasing, but significantly
%smaller than the full original problem) 
smaller subproblems is demonstrated to
yield the optimum for the original problem at hand much faster than
directly solving it as an LP or with an alternating direction scheme (see
below). The same idea, i.e., column (and constraint) generation based on
LP reformulations, was also proposed recently for $\ell_1$-regularized
training of SVMs, see~\cite{DMW19}.

Finally, \verb#l1_ls#, described in~\cite{KKLBG07}, is an interior-point
(primal log-barrier) solver for problems of the form
$\ell_1$-LS($\lambda,\R^n$) or $\ell_1$-LS($\lambda,\R^n_+$); the algorithm
employs a truncated Newton subroutine to obtain approximate search
directions.

\paragraph{\bf Homotopy Methods}\label{par:homotopyMethods}

Homotopy methods for $\ell_1$-minimization problems have been described in,
e.g.,~\cite{OPT00,MCW05,AR09,DT08}. The basic idea is to directly and
efficiently identify breakpoints of the piecewise-linear solution path of
$\ell_1$-LS($\lambda,\R^n$), following changes in $\lambda$ from
$\lambda\geq\norm{A^\top b}_\infty$ (for which the optimum is
$x^*_\lambda=0$) in a sequence decreasing to $0$, reaching an optimal
solution of BP($\R^n$). Stopping as soon as $\lambda$ drops below $\delta$
yields an optimal solution for BPDN($\delta,\R^n$). The $\ell_1$-homotopy
framework can also be applied, with small modifications, to solve
DS($\varepsilon,\R^n$), LASSO($\tau$,$\R^n$) and several other related problems,
cf.\ \cite{A08,AR09,A13,T96,EHJT04,OPT00}.

For BPDN with $\ell_\infty$-constraints, i.e., for
$\min\{\norm{x}_1\suchthat\norm{Ax-b}_\infty\leq\delta\}$ (and thus, since
$\delta=0$ is possible, also for BP($\R^n$)), a related homotopy method
called $\ell_1$-\textsc{Houdini} was developed in~\cite{BLT18}.  In fact,
$\ell_1$-\textsc{Houdini} can be extended to treat the more general problem
class
$\min\{\norm{x}_1\suchthat \ell\leq Ax-b\leq u,~Dx=d\}$~\cite{BLT18,B18},
which includes the Dantzig selector problem DS($\varepsilon,\R^n$) as a
special case. The algorithm works in a primal-dual fashion\footnote{Note
  that, in principle, the $\ell_2$-norm-based homotopy methods described
  earlier are also of a primal-dual nature; however, there, solutions to
  the respective dual subproblems admit a closed-form solution that can be
  integrated into the primal formulas~directly.}, solving auxiliary LPs
efficiently with a dedicated active-set algorithm.

Yet another homotopy method, the DASSO algorithm, is introduced
in~\cite{JRL09} for DS($\varepsilon,\R^n$). Similarly to
$\ell_1$-\textsc{Houdini}, it solves auxiliary LPs in every iteration. The
paper also provides conditions under which the homotopy solution paths for
DS($\varepsilon,\R^n$) and $\ell_1$-LS($\lambda,\R^n$) or
BPDN($\delta,\R^n$) coincide.
    
For sufficiently sparse solutions, all these homotopy algorithms are highly
efficient, beating even commercial LP solvers (cf.~\cite{LPT15,BLT18}), and
can also be used for cross-validation purposes when a suitable
measurement-error bound or regularization parameter is yet unknown, since
they provide solutions for the whole homotopy path (i.e., all values of
$\lambda$, possibly translated to $\delta$ for BPDN-constraints, that
induce a change in the optimal solution support). Efficiency in the form of
the so-called $k$-step solution property---i.e., recovering $k$-sparse
solutions in~$k$ iterations---is discussed, e.g., in~\cite{DT08}. However,
similarly to the simplex method for LPs, these homotopy methods can
generally take an exponential number of iterations in the worst
case~\cite{MY12,B18}.

Finally, the famous LARS algorithm (\emph{least angle regression},
see~\cite{EHJT04}) is a heuristic variant of the $\ell_1$-LS homotopy
method that also computes the true optimum for sufficiently sparse
solutions (via the $k$-step solution property mentioned above). However,
LARS is generally not an exact solver, because it allows only for increases
in the current support set, whereas full homotopy schemes also allow for
the (possibly necessary) removal of indices that entered the support at
some previous iteration. 

\paragraph{\bf Iterative Shrinkage/Thresholding and Other Gradient Descent-Like
  Algorithms}\label{par:ISTA}
  
A large number of proposed methods belong to the broad class of
\emph{iterative shrinkage/thresholding algorithms} (ISTA). Such methods
have mostly been derived for $\ell_1$-LS($\lambda,\R^n$) or closely related
problems, from different viewpoints and under different names, such as ISTA
and its accelerated cousin FISTA~\cite{BT09}, thresholded Landweber
iterations~\cite{DDDM04}, iterative soft-thresholding~\cite{BL08},
fixed-point iterations~\cite{HYZ08,WYGZ10}, or (proximal) forward-backward
(or monotone operator) splitting~\cite{CW05,CombP11,GSB14,R76}; see also
\cite{FN03,FN05}. Variants and extensions are numerous and sometimes known
by yet other names (e.g., Douglas-Rachford splitting or the Arrow-Hurwicz
method, both of which are special cases of the Chambolle-Pock
algorithm~\cite{CP11}).

In essence, such methods perform a gradient-descent-like step followed by
the application of a proximity operator. For instance, for
$\ell_1$-LS($\lambda,\R^n$), the basic ISTA iteration updates
\[
  x^{k+1} = \mathcal{S}_{\lambda\gamma^k}\left( x^k - \gamma^k A^\top\left(Ax^k -b\right)\right)
\]
with stepsizes $\gamma^k$, where $\mathcal{S}_\alpha$ is the
\emph{soft-thresholding operator}, defined component-wise as
\[
  \mathcal{S}_{\alpha}(x)_i \define \sign(x_i)\max\left\{\abs{x_i}-\alpha,\,0\right\}.
\]
This very general scheme that can be applied to many more problems than
just $\ell_1$-LS($\lambda,\R^n$). The stepsizes are typically chosen as
constants or related to Lipschitz constants of the least-squares term.
Acceleration of iterative shrinkage/thresholding schemes can be achieved by
homotopy-like continuation schemes (e.g., as in~\cite{WYGZ10}),
sophisticated stepsize selection routines (e.g., as in
\cite{BT09,FNW07,WNF09}), or by mitigating the negative influence of~$A$ being ill-conditioned (see~\cite{BDF07}, and also~\cite{GS10}). 
%the two-step modification TwIST~\cite{BDF07}
%aimed at mitigating the negative influence of ill-conditioning of the
%matrix~$A$ on the convergence speed of the iterative soft-thresholding
%scheme. Similarly, the CGIST algorithm from~\cite{GS10} tackles
%$\ell_1$-LS($\lambda,\R^n$) via a forward-backward splitting algorithm and
%proposes acceleration and robustification against ill-conditioned matrices
%by employing a specialized conjugate gradient solver for the (quadratic)
%subproblems. 
%Introducing a gradient-like term with respect to the employed
%thresholding operator (derived from belief propagation in graphical models)
%into the iterate update formulas, \cite{DMM09} proposes a variation of
%iterative thresholding algorithms that mimics a first-order
%\emph{approximate message passing} (AMP) scheme and reportedly improves the
%undersampling-sparsity tradeoff compared to more classical ISTA approaches.
Another variation mimics a first-order \emph{approximate message passing} (AMP) scheme~\cite{DMM09}.

In~\cite{FHT10}, an algorithmic framework called \texttt{glmnet} is
proposed for generalized linear models with convex regularization terms, in
particular including $\ell_1$-LS($\lambda,\R^n$), DS($\varepsilon,\R^n$),
and EN($\lambda_1,\lambda_2,X$). The method combines homotopy-like
parameter continuation with cyclic coordinate descent, making the update
steps extremely efficient and the algorithm one of the fastest for $\ell_1$-regularized least-squares problems (cf.~\cite{FHT10,MWZ19}).
%. computational evidence in~\cite{FHT10,MWZ19}
%suggests that the \texttt{glmnet} algorithm may be the fastest one to solve
%$\ell_1$-regularized least-squares problems. 
Nevertheless, note that it does not yield the full homotopy solution path, but instead imposes a sequence of regularization parameters that are chosen a priori or
adaptively, but not guided by homotopy path breakpoints.

The STELA algorithm~\cite{YP17} solves $\ell_1$-LS($\lambda, \R^n$) by means
of successive \mbox{(pseudo-)} convex approximations, based on a parallel
best-response Jacobi algorithm, and can be
%. The acronym STELA is derived from interpreting the algorithm 
interpreted as an iterative soft-thresholding algorithm with
exact line search. %In particular, STELA is shown to be a competitive method
%for $\ell_1$-problems also in~\cite{KT16}, and 
It has been extended to the
sparse phase retrieval problem and more general nonconvex regularizers, see~\cite{YPEO19,YPCO18}, and is further related to the
majorization-minimization approach and block coordinate descent.

The SpaRSA algorithm~\cite{WNF09} can solve $\ell_1$-LS($\lambda,\R^n$)
and, in fact, much more general problems that minimize the sum of a smooth
function and a nonsmooth, possibly nonconvex regularizer. It is related to
IST algorithms like the above, GPSR (see directly below) and trust-region
methods, but handles subproblems and stepsize selection differently. The
algorithm consists of iteratively solving subproblems that can be viewed as
a quadratic separable approximation of the $\ell_2$-norm term at the
current iterate, using a diagonal Hessian approximation for the
second-order part. %The iterative scheme involves the soft-thresholding
%operator and a Barzilai-Borwein-like stepsize rule combined with an
%acceptance criterion to avoid large oscillations of the (non-monotone)
%objective values. 
The overall scheme can, moreover, also be applied to
\cardreg{$\tfrac{1}{2\lambda}\norm{Ax-b}_2^2$}{$\R^n$}, resulting in the
use of hard- instead of soft-thresholding for the subproblem solutions.

When focusing on constrained problems like BP($\R^n$) or
BPDN($\delta,\R^n$), gradient-descent-like iterations can also be combined
with projections onto the constraint set: GPSR (\emph{gradient projection
  for sparse reconstruction})~\cite{FNW07} is such an algorithm, derived to
solve $\ell_1$-LS($\lambda,\R^n$). It applies a gradient projection scheme
with either Armijo-linesearch/backtracking or Barzilai-Borwein stepsize
selection to a reformulation of $\ell_1$-LS($\lambda,\R^n$) as a QP with
nonnegativity constraints, obtained by a standard variable split as
in~\eqref{eq:l1varsplit}. A variant using continuation is also
discussed. It is worth mentioning that a different projected gradient
scheme is proposed in~\cite{DFL08}, derived as an accelerated extension of
the iterative shrinkage/thresholding principle.

Another example is ISAL1, an infeasible-point subgradient algorithm for
$\ell_1$-minimization problems that uses adaptive approximate projections
onto the constraint set, see~\cite{LPT14}. It can handle a variety of
constraints and, in particular, is able to solve BP($\R^n$), BPDN($\delta,\R^n$),
or unconstrained problems like $\ell_1$-LS($\lambda,\R^n$). More details
are provided in \cite{T13}, including a variable target-value version of
the algorithm.

\hypertarget{SPGL1}{}
Finally, the SPGL1~\cite{VDBF08} algorithm can solve problems BP($\R^n$), BPDN($\delta,\R^n$) and LASSO($\tau,\R^n$) by employing a sequence of LASSO subproblems with suitably chosen $\tau$-parameters that are approximately solved with an efficient specialization of the \emph{spectral projected gradient} method from~\cite{BMR00}. Later, SPGL1 was generalized to objective functions of the gauge-function type and more general constraints, e.g., additionally including nonnegativity, see~\cite{VDBF11}.

\paragraph{\bf Alternating Direction Method of Multipliers (ADMM)}\label{par:ADMM}

This class of algorithms alternates improvement steps with respect to
different variable groups; in particular, auxiliary variables may be
introduced as in the augmented Lagrangian approach to relax constraints
into the objective function. The idea of treating variable groups
separately is typically to obtain comparatively easy subproblems, allowing
for fast iterations that enable applicability also in large-scale regimes
(similarly to block coordinate descent). A generic example for such a
decomposition in a more general context will be provided in
Section~\ref{sec:inexactopt:relaxation_constraints}.

YALL1, described in~\cite{YZ11}, is a framework of specialized alternating
direction methods for several $\ell_1$-problems, including BP($\R^n$),
BP($\R^n_+$), BPDN($\delta,\R^n$), and BPDN($\delta,\R^n_+$) as well as
weighted-$\ell_1$-norm minimization or $\ell_1$-constrained variants.

SALSA~\cite{ABDF10}---short for \emph{(constrained) split augmented
  Lagrangian shrinkage algorithm}---is another ADMM scheme applied to a
classic augmented Lagrangian reformulation of $\ell_1$-LS($\lambda,\R^n$)
that is obtained by introducing auxiliary variables $v=x$ and
Lagrange-relaxing this constraint. The scheme can be extended to BPDN($\delta,\R^n$)~\cite{ABDF11} and more general objectives than the $\ell_1$-norm; 
%SALSA also allows for more general
%regularization functions than the $\ell_1$-norm. Algorithm
%C-SALSA~\cite{ABDF11} extends the approach to BPDN($\delta,\R^n$) (again,
%also for more general objectives than $\ell_1$) by first moving the
%constraint into the objective via the corresponding indicator function, and
%then applying the SALSA reformulation. 
%The algorithm 
it hinges on efficient proximity operators for the regularization term, provided by standard soft-thresholding in the $\ell_1$-case, and requires computation of
inverses for $(A^\top A+\alpha I)$ or $(A A^\top+\alpha I)$, $\alpha>0$.

\paragraph{\bf Smoothing Techniques}\label{par:smoothingTechniques}

The NESTA algorithm~\cite{BBC11} is developed for problem
BPDN($\delta,\R^n$) and works by applying Nesterov's smoothing techniques
(cf.~\cite{N05}) to the $\ell_1$-norm objective function. The general method can also be applied to related problems, e.g., with a weighted-$\ell_1$ objective or the $\ell_\infty$-norm constrained problem (in its Lagrangian/regularized form), and may be combined with a homotopy-like parameter continuation scheme for decreasing $\delta$-values. The algorithm was mainly designed for the case
$A^\top A=I$; it can handle the non-orthogonal setting as well, but then
may require costly subroutines such as computing a full singular
value decomposition of~$A$.

The paper~\cite{GLW13} proposes two related algorithms: NESTA-LASSO is a
specialization of NESTA (i.e., essentially, Nesterov's algorithm) to LASSO($\tau,\R^n$) with a slight modification to establish
additional convergence properties, and \textsc{ParNes} combines SPGL1 (described earlier) with NESTA-LASSO, solving the LASSO subproblems
of the spectral projected gradient (SPG) scheme approximately with the
novel algorithm. Thus, \textsc{ParNes} can solve both BPDN($\delta,\R^n$)
and $\ell_1$-LS($\lambda,\R^n$) in particular.

The TFOCS~\cite{BCG11} framework for solving a variety of $\ell_1$-related
(as well as more general) problems is based on reformulating constraints in
the form of $A(x)+b\in\mathcal{K}$ with a linear operator $A$ and a closed
convex cone $\mathcal{K}\in\R^n$, smoothing the typically nonsmooth
objective function (e.g., the $\ell_1$-norm), and then applying efficient
first-order methods on the dual smoothed problem, along with a way to eventually recover associated approximate primal solutions. The TFOCS
framework can thus be adapted to concrete problems at hand (e.g.,
BP($\R^n$) or LASSO($\tau,\R^n$)) in a template-like fashion, combining
first-order methods like FISTA or standard projected-gradient schemes with
proximity or projection operators and other building blocks.

\paragraph{\bf Bregman Iterative Algorithms}

The paper \cite{YOGD08} (see also \cite{YO13}) proposes \emph{Bregman
  iterative regularization} to solve BP($\R^n$), extending previous
work~\cite{OBGXY05}. The method builds on iteratively solving subproblems
involving the so-called Bregman distance, which is essentially the slack of
a subgradient inequality, and can be traced back to~\cite{B67}. These
subproblems turn out to reduce to problems of the form
$\ell_1$-LS($\lambda,\R^n$) with a different right-hand-side vector $b$ in
each iteration. Thus, any available solver for $\ell_1$-regularized
least-squares problems can be employed to solve the subproblems of the
Bregman iterative scheme. In~\cite{YOGD08}, the authors propose to use the
fixed-point continuation (FPC) algorithm of~\cite{HYZ08}, but note that
today, more efficient methods are known (even compared to the active-set
improvement of FPC, \verb#FPC_AS#, introduced in \cite{WYGZ10}), e.g.,
\texttt{glmnet}~\cite{FHT10,MWZ19}. It is noted in~\cite{YO13} that the
Bregman iterative procedure is equivalent to the augmented Lagrangian
method.

The \emph{Linearized Bregman iteration}~\cite{YOGD08,COS09,OMDY11} also
tackles BP($\R^n$), by solving a Tikhonov-regularized version of the problem, i.e.,
$\min\{\norm{x}_1+\tfrac{1}{2\lambda}\norm{x}_2^2\suchthat Ax=b\}$, which can be shown to yield the same solution as BP($\R^n$) for sufficiently large $\lambda>0$ \cite{FT07,LY13}. The necessary value of $\lambda$ is data-dependent
and generally unknown, but may be estimated for practical purposes as a
small multiple of the maximal absolute-value entry of the (unknown) optimal
solution~\cite{LY13}. The crucial difference to the standard Bregman
iteration is that the quadratic data-fidelity term
$\tfrac{1}{2}\norm{Ax-b}_2^2$ in the $\ell_1$-regularized least-squares
problems is replaced by its (gradient-based) linear approximation
$x^\top A^\top(Ax-b)$, hence \emph{linearized} Bregman. 
%The strongly convex
%additional ridge/Tikhonov regularization term allows to prove global linear
%convergence with explicit rates, see~\cite{LY13}. 
The %latter 
paper~\cite{LY13} %also
discusses extensions of the method to BPDN($\delta,\R^n$) and low-rank
matrix recovery problems, and \cite{Y10} shows that it can be viewed as gradient descent applied to a certain dual reformulation, and as such can be sped up significantly by incorporating, e.g., %Barzilai-Borwein stepsizes, 
stepsize linesearch or Nesterov's acceleration technique. Moreover, \cite{SSKSP19} provided a partial Newton method acceleration scheme for the linearized Bregman
method, extending an earlier improvement suggestion of~\cite{QLW14}
involving generalized inverse matrices.

In \cite{GO09}, a \emph{Split Bregman} formulation is proposed, which
amounts to applying the Bregman approach to an augmented Lagrangian model involving auxiliary variables, solving subproblems by alternating minimization.
%introduces auxiliary variables coupled to the original variables by linear
%equality constraints that are then moved to the objective as a
%least-squares deviation term. The Bregman approach is then applied to this
%new formulation, and the associated subproblems are solved by alternating
%minimization (instead of, e.g., FPC schemes or other solvers for
%$\ell_1$-LS-like problems). The Split Bregman algorithm was demonstrated to
%be faster than linearized Bregman iterations on several application
%problems, including BP($\R^n$). 

\paragraph{\bf Other Noteworthy Algorithmic Approaches}%\label{par:otherL1Approaches}

There are some further interesting methods that, while certainly related to
some degree, do not quite fit into the previous categories. Therefore, we
list them here.
\begin{itemize}  
\item \hypertarget{SMNewt}{}A \emph{semismooth Newton method} is proposed in~\cite{GL08} for
  $\ell_1$-LS($\lambda,\R^n$). While %shown to be 
  locally superlinearly convergent, the method %, which can be realized as an active-set algorithm,
  depends strongly on the selected starting points, 
  which is overcome by the globalization strategy described in~\cite{MU14} that, essentially, %leads to a
  %combination of ISTA and the semismooth Newton scheme, replacing steps of
  replaces iterations by ISTA steps if a certain filter-based acceptance
  criterion fails. Regularization parameter choice in the context of
  semismooth Newton methods applied to $\ell_1$-regularized least-squares
  problems is discussed in~\cite{CJK10}. Recently, \cite{BCNO16} proposed a
  unifying semismooth Newton framework that can be used to generate
  different incarnations of such second-order methods, including active-set
  and second-order ISTA schemes. Further related methods are zero-memory quasi-Newton forward-backward splitting algorithm from~\cite{BF12} (see also~\cite{BFO18}) that can also handle more general problems, or, e.g., the orthant-wise learning algorithm
  from~\cite{AG07} and the SmoothL1 and ProjectionL1 methods from
  \cite{SFR07} that also utilize (restricted) second-order information.

  %It is also worth mentioning that \cite{BF12} (see also~\cite{BFO18})
  %introduces a zero-memory quasi-Newton forward-backward splitting
  %algorithm (using diagonal$+$\mbox{rank-$1$} Hessian approximations) that
  %turned out to be a semismooth-Newton-like method in some of the cases the
  %quite general framework can handle, in particular, $\ell_1$-related
  %ones. Several related methods for $\ell_1$-regularized problems that also
  %make use of (restricted) second-order information are mentioned in
  %\cite{BF12,BFO18}, e.g., the orthant-wise learning algorithm
  %from~\cite{AG07} or the SmoothL1 and ProjectionL1 methods from
  %\cite{SFR07}.

\item \emph{Iterative Support Detection (ISD)} and the analogous \emph{threshold-ISD} \cite{WY10} target improving the solution sparsity if BP($\R^n$) or $\ell_1$-LS($\lambda,\R^n$), respectively, fail to work as intended in reconstructing a sparse signal (e.g., if the number of measurements $m$ is too small). The methods are specific reweighting schemes that iteratively solve smaller BP or $\ell_1$-LS instances, each setting to zero the objective contribution of the support of the previous instance's solution. These subproblems can, in
  principle, be tackled by any specialized solver (the authors
  of~\cite{WY10} use YALL1, i.e., an ADMM approach). Empirically, ISD is
  demonstrated to perform slightly better than the strongly related
  \emph{iteratively reweighted $\ell_1$-minimization (IRL1)} method
  from~\cite{CWB08} (which follows---and actually introduced---the same
  general idea, but uses different weights derived from the respective
  previous subproblem solution) as well as the similar \emph{iteratively
    reweighted least-squares (IRLS)} algorithm, also known as \emph{FOCUSS}
  (FOCal Underdetermined System Solver), which employs a reweighted
  $\ell_2$-norm objective, or possibly nonconvex $\ell_p$-quasinorms with
  $0<p<1$, cf.~\cite{CW08,DDVFG10,GR97,O85}.
    
\item The term \emph{Active-Set Pursuit} refers to a collection of
  algorithms based on a dual active-set QP approach to basis pursuit and
  related problems, described in~\cite{FS12}. It can also solve
  BPDN($\delta,\R^n$) and be utilized in a reweighted basis pursuit
  algorithm (solving a sequence of BP-like problems with objective
  $\norm{W^k x}_1$ with different diagonal weighting matrices $W^k$) aiming
  at further reducing the sparsity of computed solutions, similarly to ISD.

\item The \emph{polytope faces pursuit} algorithm from~\cite{P06} is a
  greedy method for the solution of the Basis Pursuit
  problem~BP($\R^n$). In essence, it proceeds by identifying active faces
  of the polytope that constitutes the feasible set of the dual of the
  standard variable-split LP reformulation \eqref{eq:l1varsplit} of
  BP($\R^n$) and adding or removing solution components one at a time. Numerical experiments show a favorable comparison against matching pursuit
  (cf.~Section~\ref{sec:inexactopt:greedy_heuristics}) and an
  interior-point LP solver for BP($\R^n$) for a few specific
  signal types.

\item The paper~\cite{CCG12} proposes to use standard algorithms for
  general convex feasibility problems to compute sparse solutions, by
  employing either cyclic or simultaneous (weighted) subgradient
  projections w.r.t.\ equality constraints $Ax=b$ (projecting onto rows
  separately) and constraints $\norm{x}_1\leq\tau$. Thus, these methods can
  be understood to asymptotically solve LASSO($\tau,\R^n$) if the combined set
  $\{x: Ax=b,~\norm{x}_1\leq\tau\}$ is nonempty, although~\cite{CCG12}
  motivates them differently and it is typically not possible to determine
  this type of constraint consistency a priori (it essentially amounts to
  optimally choosing the parameter $\tau$). Thus, the proposed algorithms
  CSP-CS and SSP-CS are not really exact solvers for a certain
  $\ell_1$-problem, but should rather be viewed as $\ell_1$-based
  heuristics. Note also that, in principle, one could project onto the
  whole feasible set $Ax=b$ directly in explicit closed form, although the
  projections onto single rows are significantly cheaper. Alternating
  projection methods for convex sets are a special case of the splitting
  methods discussed earlier, and as such also come with various convergence
  guarantees. In particular, the heuristics from~\cite{CCG12} could easily
  be extended to other problems, e.g., involving constraints like
  $\norm{Ax-b}_2\leq\delta$, by employing approximate projections such as
  those utilized in ISAL1~\cite{LPT14,T13}.
\end{itemize}

\bigskip
To conclude this section, we point out the extensive numerical comparison
for several Basis Pursuit solvers (i.e., implementations provided by the
respective authors) reported in~\cite{LPT15}; see also~\cite{KT16}. This
comparison demonstrates that the interior-point codes
$\ell_1$-\textsc{magic} and SolveBP are not competitive with other methods,
including, in particular, the respective dual simplex algorithms of
SoPlex~\cite{W96,scip} and CPLEX~\cite{cplex} applied to the variable-split
LP formulation \eqref{eq:l1varsplit}. The overall ``winner'' of this solver
comparison for BP($\R^n$) is the $\ell_1$-homotopy method based on
$\ell_1$-regularized least squares. However, recall that the more recent
work \cite{MWZ19} demonstrated that LP techniques can be made faster than
the homotopy method by integrating column and constraint generation.

Moreover, significant speed-ups and accuracy improvements can be achieved
for almost all methods by incorporating a so-called \emph{heuristic
  optimality check (HOC)}, described in~\cite{LPT15} for BP($\R^n$),
extended to BPDN($\delta,\R^n$) and $\ell_1$-LS($\lambda,X$) in~\cite{T13},
and generalized to BPDN-like problems with arbitrary norms in the
constraints in~\cite{BLT15}. It is also worth noting that the work
\cite{MD10} provides extensive parameter tuning experiments for various
iterative (hard and soft) thresholding methods and some other algorithms,
aiming at relieving users from the burden of having to select appropriate
regularization, noise- or sparsity-level parameters when using one of the
noise-aware $\ell_1$-optimization models and dedicated solvers; see also~\cite{BPY21} for recent results on parameter choice sensitivity of these models.

Numerous further papers treat more variations of the above methods and
ideas, often providing slight improvements to the originally proposed
schemes, generalizing them to a broader context, or treating much more
general optimization problems that contain one or more of the above
$\ell_1$-problems as special cases (e.g., minimization of composite convex
objective functions). For instance, quadratic/nonlinear basis pursuit is
discussed in \cite{OYVS13,OYDS13}, and so-called compressed phase retrieval
in~\cite{MRB07}. Moreover, the paper \cite{YCHL10} surveys and compares
various methods and available implementations for $\ell_1$-regularized
training of linear classifiers. Similarly to $\ell_1$-regularization in the
context of sparse signal recovery or sparse regression, the described
techniques stem from the whole range of applicable approaches, including
cyclic coordinate-descent methods, active-set and quasi-Newton schemes, and
projected (sub-)gradient algorithms.

It goes beyond the scope of this survey to further identify and remark on
possible extensions and applicable algorithms. Nevertheless, we note that 
recent modifications of the many algorithms summarized above may often be 
found simply by searching for citations of the respective original works 
referenced here. Moreover, implementations of many of the methods (usually 
in Matlab or Python) can also be found online, either prototyped directly 
by their authors or as part of more sophisticated larger software packages.
An important generalization of the $\ell_1$-norm approach (and the nuclear 
norm surrogate for matrix rank, cf.~\cite{RFP10}) is discussed in the following.

\subsection{Generalization: Atomic Norm Minimization}\label{sec:inexactopt:surrogates:atomic}
From a geometric perspective, the popular $\ell_1$-norm approach to
reconstructing sparse solutions from few linear measurements can also be
viewed as minimizing the so-called \emph{atomic norm} induced by the set of
unit one-sparse vectors; the convex hull of this \emph{atomic set}
coincides with the unit $\ell_1$-norm ball, i.e., the cross-polytope. As
laid out in~\cite{CRPW12}, this perspective yields a natural generalization
which gives rise to related convex heuristics for the recovery of sparse,
or simple'', solutions in a variety of applications: Provided the solution
in question is formed as a nonnegative linear combination of a few elements
of a (centrally symmetric, compact) atomic set~$\mathcal{A}\subset\R^n$,
the convex program $\min\{\norm{x}_\mathcal{A}:\norm{Ax-b}\leq\delta\}$,
can successfully recover it under certain assumptions, where
$\norm{x}_\mathcal{A}\define\inf\{\sum_{a\in\mathcal{A}}
c_a:x=\sum_{a\in\mathcal{A}} c_a a,~c_a\geq 0~\forall a\in\mathcal{A}\}$ is
the atomic norm. While the atomic norm may not be computable for an
arbitrary atomic set, in many cases of interest it does turn out to be
tractable or efficiently approximable; besides sparse vectors, the examples
detailed in~\cite{CRPW12} include the recovery of, e.g., low-rank matrices
(where the atomic norm reduces to the well-known nuclear norm~\cite{RFP10},
i.e., the sum of singular values), permutation or orthogonal matrices,
vectors from lists and low-rank tensors, with applications in machine
learning, (partial) ranking, or object tracking. Further applications of
the atomic norm framework cover, e.g., breast cancer prognosis from gene
expression data via a group-LASSO model with overlaps~\cite{OJV11}, linear
system identification~\cite{SBTR12}, sparse phase retrieval and sparse
PCA~\cite{MRD21}, image superresolution~\cite{CFDC20}, direction-of-arrival
estimation~\cite{YX16}, or speeding up neural network training by
sparsifying stochastic gradients~\cite{WSLCPW19}, to name but a
few. Similarly to the $\ell_1$-case, iterative reweighting can improve
solution sparsity for atomic norm minimization \cite{YX16}, and conditions
for exact or bounded-error approximate recovery from noiseless or noisy
linear measurements, respectively, can be formulated generally and for
special cases. For instance, \cite{CRPW12} provide probabilistic guarantees
in terms of the number of Gaussian linear measurements required for success
for several settings, and the very recent work~\cite{C22} gives
deterministic recovery conditions analogous to the nullspace property
(cf.~Section~\ref{sec:inexactopt:surrogates:theory}).

\subsection{Other Approximations for the Cardinality Objective}\label{sec:inexactopt:approximation_objective}
There are several works that consider replacing the cardinality objective
by other nonlinear approximations than the $\ell_1$-norm. The main reason
this is apparently less common is presumably the fact that such
approximations are almost exclusively nonconvex, yielding harder
optimization problems. For instance, the nonconvex $\ell_p$-quasinorms with
$0<p<1$ naturally tend to the $\ell_0$-norm for $p\searrow 0$ (so the
smaller~$p$, the better the approximation, generally), but while some
recoverability results similar to $\ell_1$-minimization problems can be
shown (see, e.g., \cite{C09,CS08}), the classic problem variants with such
nonconvex $\ell_p$-objectives are still (strongly)
\NP-hard~\cite{GJY11}. In both theory and practice, gains can be achieved
over $\ell_1$-minimization w.r.t.\ recoverable sparsity levels or number of
required measurements, and despite hardness and nonconvexity issues such as
the need to distinguish local from global optima (cf.~\cite{CG15} in the
present context), fast algorithms that work quite well have been developed.
For instance, \cite{MR10,BL11} describe IRLS-related or subgradient-based
descent schemes, respectively, \cite{GJY11} investigates an interior-point
potential-reduction method, and \cite{MFH11} proposes a coordinate-descent
algorithm for least-squares regression with nonconvex penalty
regularization targeting sparsity.

The paper~\cite{MBZJ09} proposes a method called SL0 (smoothed $\ell_0$)
that consists of an (inexact) projected gradient scheme applied to
maximizing the smooth functions
$F_\sigma(x)\define\sum_{i=1}^n e^{-x_i^2/(2\sigma^2)}$, for a decreasing
sequence of $\sigma$-values. Since for $\sigma\to 0$,
$e^{-x_i^2/(2\sigma^2)}\to 1-\norm{x_i}_0$, it follows that
$F_\sigma(x)\to n-\norm{x}_0$, so maximizing $F_\sigma(x)$ amounts to
approximately minimizing $\norm{x}_0$. Convergence is proven under certain
assumptions, and numerical experiments suggest superiority w.r.t.\ basis
pursuit in some settings.

The comparatively early work~\cite{M99}, published before the rise of
compressed sensing, treats the problem of finding minimum-support vertex
solutions of general polyhedral sets. In particular, it is demonstrated
under mild assumptions that $\ell_0$-regularized minimization of a concave
function over polyhedral constraints admits an optimal vertex solution, and
that there exists an exact smooth approximation of the cardinality penalty
term such that for certain finite choices of penalty parameters,
minimum-support solutions are retained. The suggested approximation is
$\norm{x}_0\approx n-\ones^\top e^{-\alpha y}$ for some (sufficiently
large) $\alpha>0$, where $e^q=(e^{q_1},\dots,e^{q_n})^\top$ for a vector
$q\in\R^n$ and $-y\leq x\leq y$.  With $X\subseteq\R^n$ describing the
polyhedral set and $f$ the concave original objective, the suggested
regularized problem thus reads
\[
  \min_{(x,y)} f(x) + \beta \ones^\top(\ones - e^{-\alpha
    y})\quad\text{s.t.}\quad x\in X,~-y\leq x\leq y,
\]
with regularization parameter $\beta\leq\beta_0$ for some $\beta_0>0$ and
penalty parameter $\alpha\geq\alpha_0(\beta)$ for some $\alpha_0(\beta)>0$.
(Note that $y$ effectively models the component-wise absolute value
of~$x$.) Special cases discussed explicitly are linear programs and linear
complementarity problems. The suggested algorithm based on this exact
penalty scheme is an application of a finitely-terminating fast successive
linearization algorithm. Adaptions of the approach from~\cite{M99} to the
problem \cardmin{$\norm{Ax-b}_\infty\leq\delta$} and a sparse portfolio
optimization problem are discussed in~\cite{JP08} and~\cite{LLRSS12},
respectively.

Based on ideas from~\cite{BB97}, \cite{JP08} discussed a heuristic for \cardmin{$\norm{Ax-b}_\infty\leq\delta$} that
builds on the equivalent bilinear reformulation
\[
  \min\ones^\top z\quad\text{s.t.}\quad b-\delta\ones\leq Ax-b\leq
  b+\delta\ones,~x_i(1-z_i)=0~~\forall i\in[n],~0\leq z\leq\ones,
\]
which is closely related to
the approach in \cite{FMPSW18}. To overcome the nonconvexity of the
bilinear (equilibrium or complementarity-type) constraints, one can move
the bilinear constraint into the objective and introduce an upper-bound
constraint for the cardinality; a sequence of subproblems can then be
solved efficiently for different objective bounds to obtain a final
solution, see \cite{JP08,BB97,BM93}. 

The connection to MaxFS/MinIISCover described in Section~\ref{sec:concreteproblems:miscellaneous:combopt} has also been exploited to derive a variety of (often LP-based) heuristics for cardinality minimization problems such as sparse signal reconstruction, subset selection, classifier hyperplane placement and others, see, e.g., \cite{C01,P08,C12,C19,FCR20,FCR22} and references therein. Numerical studies in these works suggest that such heuristics often yield better solutions than more common (e.g., greedy or $\ell_1$-norm-based) approaches, but still appear to be less widely known.

\subsection{Other Relaxations of Cardinality Constraints}\label{sec:inexactopt:relaxation_constraints}
Analogously to the reformulation of the cardinality minimization problem
mentioned in Section \ref{sec:exactopt:modelingcard}, one can reformulate
cardinality-constrained problems \cardcons{$f$}{$k$}{$X$} using
complementarity-type constraints as
\begin{equation}\label{eq:cardConstRelaxed}
  \min~f(x) \st  x \in X,~\ones^\top y \leq k,~x_i (1-y_i) = 0~\forall i\in [n],~0 \leq y \leq \ones,
\end{equation}
which was discussed in \cite{BKS16} as well as in \cite{FMPSW18, BLP14} for
cardinality minimization problems. The continuous-variable problem
\eqref{eq:cardConstRelaxed}, although being a relaxation (of~$y$ being
binary), still has the same global solutions as the original problem
\cardcons{$f$}{$k$}{$X$}. Note, however, that local solutions of
\eqref{eq:cardConstRelaxed} at which the cardinality constraint is not
active are not necessarily local solutions of \cardcons{$f$}{$k$}{$X$}.
This situation is specific to cardinality-constrained problems and does not
occur when the same type of reformulation is used for cardinality
minimization or regularization problems. We first focus on approaches that
tackle cardinality-constrained problems via the relaxed reformulation
\eqref{eq:cardConstRelaxed} with tools from nonlinear optimization.

Due to the complementarity-type constraints, the relaxed problem
\eqref{eq:cardConstRelaxed} is nonconvex and degenerate in the sense that
the feasible set does not have interior points and classical constraint
qualifications from nonlinear optimization are not satisfied.  Therefore,
it needs special care both in its theoretical analysis and in numerical
solution methods, see, e.g., \cite{CKS16, BS18} for tailored optimality
conditions. Due to the close relation of the relaxed problem to
mathematical programs with complementarity constraints (MPCCs), it is
possible to modify solution approaches for MPCCs, see, e.g., \cite{BKS16,
  BBCS17, LPR96, HKS13} and references therein. Since the
complementarity-type constraints in the relaxed problem
\eqref{eq:cardConstRelaxed} are linear\footnote{In the literature, complementarity constraints are usually of the form $0\leq g(x)\perp h(x)\geq 0$~and are called \emph{linear} if both $g$ and $h$ are affine-linear functions; the condition itself is always nonlinear.}, it is especially worth taking a
look at MPCCs with linear complementarity constraints, see
Section~\ref{sec:exactopt:modelingcard} for some references on linear
programs with complementarity constraints (LPCCs) and extensions to convex
QPs with complementarity constraints. Lately, augmented Lagrangian methods
have also become popular for degenerate problems such as MPCCs or the
relaxed problem \eqref{eq:cardConstRelaxed}, because they can be applied
directly without specialization, see \cite{ISU12, KRS2020, KRS2021}.

In \cite{XS20}, an ADMM was designed for the relaxed reformulation of the
cardinality regularization problem \cardreg{$\rho$}{$Ax \geq b$}, and the
authors of \cite{YG16} use the observation~\eqref{cardconsYG16}, i.e., that
\begin{equation*}
  \norm{x}_0 \leq k
  \quad \Leftrightarrow \quad  \norm{u}_1 \leq k,~\norm{x}_1 = x^\top u,~-\ones \leq u \leq \ones,
\end{equation*}
which is closely related to the reformulation used in
\eqref{eq:cardConstRelaxed}, as the basis for an alternating exact penalty
method and an alternating direction method. The central idea used in such
alternating methods (sometimes also called splitting or decomposition
methods) is to separate the considered problem into two (or more) parts
such that each individual problem is tractable. The precise methods then
differ with regards to which problem is considered, how exactly it is
split, how the resulting subproblems are coupled, and how they are solved
individually. To illustrate the basic idea for the relaxed problem
\eqref{eq:cardConstRelaxed}, let us assume that the objective function $f$
can be written as $f(x) = f_C(x) + f_N(x)$ with a convex function $f_C$ and
a nonconvex function $f_N$. Further, assume that the feasible set $X$ is
convex. Then, \eqref{eq:cardConstRelaxed} can be stated equivalently as
\begin{align*}
  \min_{(x,y), (v,w)}\quad &f_C(x) + f_N(v)\\
  \text{s.t.}\quad &x \in X,~\ones^\top y \leq k,~0 \leq y \leq \ones,\\
  &v_i (1-w_i) = 0 \quad \forall i\in [n],\\
  &(x,y) = (v,w).
\end{align*}
We can move the coupling condition $(x,y) = (v,w)$ to the objective using a
(say) least-squares penalty term and obtain
\begin{align*}
  \min_{(x,y), (v,w)}\quad &f_C(x) + f_N(v) + \alpha \norm{(x,y) - (v,w)}_2^2\\
  \text{s.t.}\quad &x \in X,~\ones^\top y \leq k,~0 \leq y \leq \ones,\\
  &v_i (1-w_i) = 0 \quad \forall i\in [n],
\end{align*}
with some penalty parameter $\alpha > 0$. For fixed values of $(v,w)$, this
problem is convex with respect to $(x,y)$. But for fixed values of $(x,y)$,
the problem is \emph{not} convex w.r.t.\ $(v,w)$ due to the
complementarity-type constraints and the potentially present nonconvex part
$f_N$ of the objective function. Nevertheless, in case $f$ is convex and
thus $f_N \equiv 0$, solving the minimization problem with regards to
$(v,w)$ reduces to projecting $(x,y)$ onto the set of points $(v,w)$ with
$v_i (1-w_i) = 0$ for all $i\in [n]$, for which a closed-form solution is
available. A closed-form solution for a nonconvex, but quadratic function
$f_N$ is given in \cite{XS20}. Moreover, recall that ADMM schemes are also
popular for convex $\ell_1$-based models,
cf. Section~\ref{sec:inexactopt:surrogates:algorithms}, or for nonconvex
tasks like dictionary learning, where the decomposed problem may not always
have closed-form solutions but can often be quickly solved approximately by
iterative schemes, see, e.g., \cite{TEM16,LTYEP21}. Thus, one can alternate
between solving the optimization problem over just $(x,y)$ and just
$(v,w)$, respectively. While such alternating minimization schemes often
work well in practice, it can be nontrivial to actually prove convergence.

Recently, one can also see some efforts to develop a unified theory for
several classes of complementarity-type constraints including those in the
relaxed problem \eqref{eq:cardConstRelaxed}, see for example \cite{BG18,
  BCH21}. In the future, those could give rise to new, more flexible
solution approaches. The basic idea here is to consider a more general
class of optimization problems with disjunctive constraints, i.e., where
the feasible set can be represented not only via intersections but also
unions of sets.  In fact, the resulting theory can be applied to the
relaxed problem \eqref{eq:cardConstRelaxed} but also directly to
cardinality-constrained problems \cardcons{$f$}{$k$}{$X$}, because the set
$\{x \in \R^n : \norm{x}_0 \leq k\}$ can be written as the union of
finitely many $k$-dimensional subspaces of $\R^n$.

Next, we describe some approaches that consider the cardinality-constrained
problem \cardcons{$f$}{$k$}{$X$} directly and employ methods from nonlinear
optimization. For the problem \cardcons{$f$}{$k$}{$\R^n$}, i.e.,
\begin{equation*}
  \min_x~f(x) \st \|x\|_0 \leq k
\end{equation*}
without additional constraints, several optimality conditions---such as
coordinate-wise optimality---are introduced in \cite{BE13} and then used to
analyze the convergence properties of an iterative hard thresholding
algorithm and an iterative greedy simplex-type method. In \cite{BH16}, this
approach is extended to allow closed convex feasible sets
$X \subseteq \R^n$ and efficient methods to compute the projection onto the
sparse feasible set\footnote{Note that, unlike projection onto the
  $k$-sparse set $\{x\in\R^n:\norm{x}_0\leq k\}$, projection onto the
  $k$-\emph{cosparse} set $\{x\in\R^n:\norm{Bx}_0\leq k\}$, with
  $B\in\R^{p\times n}$, is \NP-hard \cite{TGP14} (see also \cite{T19}).}
(i.e., $\{x\in X:\norm{x}_0\leq k\}$) are presented. Further
generalizations to regularized cardinality problems \cardreg{$\rho$}{$X$}
and to group sparsity can be found in \cite{BH18, BH19}. 

More optimality conditions based on various tangent cones, normal cones and
restricted normal cones can be found in \cite{PXZ15, BLPW14, BLPW13a,
  BLPW13b, LZ13}. In addition to developing these optimality conditions,
the authors also apply them to analyze the convergence properties of an
alternating projection method for \cardmin{$Ax=b$} and a penalty
decomposition method for \cardcons{$f$}{$k$}{$X$} and
\cardreg{$\rho$}{$X$}, see also
Section~\ref{sec:inexactopt:greedy_heuristics}. The authors of \cite{LLS21}
employ a similar penalty decomposition method for
\cardcons{$f$}{$k$}{$\R^n$} with an emphasis on possibly nonconvex
objective functions $f$, and in \cite{TYYS17} a penalty decomposition-type
algorithm is tailored to cardinality-constrained portfolio
problems. Similar optimality conditions also form the basis of \cite{HM20},
where the authors consider regularized linear regression problems and
combine a cyclic coordinate descent algorithm with local combinatorial
optimization to escape local minima. It is also worth mentioning that~\cite{AGH21} formulate an iterative convex relaxation method for \cardcons{$\norm{b-x}_2^2+\norm{Bx}_2^2$}{$k$}{$\R^n_{\geq 0}$}, where $B$ encodes adjacency relations of entries in~$x$ (e.g., neighboring pixels in an image), which is based on a non-relaxed complementarity-type MIQP model and perspective reformulation; this scheme is shown to significantly outperform standard $\ell_1$-techniques for such problems.

So-called \emph{difference of convex functions (DC)} approaches, see, e.g.,
\cite{ZSLS14, GTT18, LDLV15, LM14} and the many references therein, utilize
the fact that most nonconvex objective functions $f(x)$ occurring in
real-life applications can be written as a difference of two convex
functions $f(x) = g(x) - h(x)$. It should be noted that the DC formulation
of a function is generally not unique and that different formulations can
have different properties. The DC formulation can then be exploited
algorithmically, e.g., by replacing the function $h$ with an affine
approximation, which results in a convex objective function. For the
cardinality-constraint $\norm{x}_0 \leq k$, there exist several DC
formulations, e.g.,
\[
  \norm{x}_1 - \norm{x}_{1,k}  = 0
  \quad \text{or} \quad
  \norm{x}_2^2 - \norm{x}^2_{2,k} = 0,
\]
where $\norm{x}_{1,k}$ and $\norm{x}_{2,k}$ denote the largest-$k$ norms of
the vectors $x$, meaning the $\norm{\cdot}_1$- or $\norm{\cdot}_2$-norm
applied to the $k$ largest components (in absolute value) of $x$,
respectively. Such DC reformulations can be used in all classes of COPs; in
case of cardinality-constrained problems, a penalty formulation is often
used to move the cardinality term into the objective function. 

Finally, it is worth mentioning that \cardcons{$\tfrac{1}{2}\norm{Ax-b}_2^2$}{$k$}{$\R^n$} can be approximated by the so-called \emph{trimmed LASSO} (cf.~\cite{ABN21,BCM17})
\[
\min \tfrac{1}{2\lambda}\norm{Ax-b}_2^2+\norm{x}_{1,k}.
%\min\{\tfrac{1}{2\lambda}\norm{Ax-b}_2^2+\min\{\norm{\phi-x}_1:\norm{\phi}_0\leq k\}\}.
\]
This problem actually solves the cardinality-constrained least-squares problem exactly for sufficiently large~$\lambda$, is related to a variety of other LASSO-like problems, and it as well as closely related variants can be solved by several algorithmic techniques including ADMM and DC programming, see \cite{HG14,ABN21,BCM17,GTT18} and references therein.

\subsection{Greedy Methods and Other Heuristics}\label{sec:inexactopt:greedy_heuristics}
There is a large number of further algorithmic approaches that have been
adapted to cardinality minimization or cardinality-constrained
problems. Broadly speaking, these methods are mostly greedy schemes or
based on algorithmic frameworks originating in convex optimization. Other
broad families of heuristics such as evolutionary algorithms or randomized
search can also be found, but are apparently much less common in the
context of cardinality optimization problems. Since, moreover, such methods
are typically highly application-specific (for example, portfolio
optimization has been addressed by means of clustering and local relaxation
\cite{MS12}, particle swarm schemes \cite{DLL12}, genetic algorithms,
simulated annealing, and tabu search \cite{CMBS00}, and even neural
networks \cite{FG07}), we do not delve into the details in this paper.

\paragraph{\bf Hard Thresholding}\label{par:hardThresholding}
The papers \cite{BD08b,BD09a} introduce the \emph{iterative
  hard-threshold\-ing algorithm (IHT)} alluded to earlier, and prove
convergence of the algorithm iterates to local minima as well as error
bounds under certain conditions (e.g., the RIP).  For
\cardreg{$\tfrac{1}{\lambda}\norm{Ax-b}_2^2$}{$\R^n$}, the IHT iteration
(starting at $x^0\define 0$) reads
\begin{equation}\label{eq:IHTiteration}
  x^{k+1} \define \cH_{\sqrt{\lambda}}\left(x^k + A^\top\left(b-Ax^k\right)\right),
\end{equation}
where $\cH_{\sqrt{\lambda}}(\cdot)$ is the hard-thresholding operator,
defined component-wise as
\[
  \cH_{\epsilon}(z_i)\define\begin{cases} \hfill
    0,&\abs{z_i}\leq\epsilon\\ ~z_i,&\abs{z_i}>\epsilon.\end{cases}
\]
A similar iterative scheme is also proposed and analyzed in~\cite{BD08b}
for the cardinality-constrained $\ell_2$-mini\-mi\-zation problem
\cardcons{$\norm{Ax-b}_2^2$}{$k$}{$\R^n$}, called the \emph{$k$-sparse
  algorithm} there. The iterations are completely analogous
to~\eqref{eq:IHTiteration} except that $\cH_{\sqrt{\lambda}}(\cdot)$ is
replaced by the operator $\cH^0_{k}(\cdot)$, which retains the $k$ largest
absolute-value entries. %The IHT method also works for complex-valued
%variables. 
The papers~\cite{BD08b,BD09a} also discuss connections and
similarities of IHT and matching pursuit algorithms like OMP and CoSaMP
(outlined further below).

In~\cite{BTW15}, the IHT approach is combined with the conjugate gradient
principle to the \emph{CGIHT} algorithm. That work also provides
probabilistic recovery and stability guarantees for CGIHT variants (applied
to \cardcons{$\norm{Ax-b}_2^2$}{$k$}{$\R^n$}, joint-sparsity, and matrix
completion problems) as well as empirical phase transition diagrams, and
demonstrates computational advantages over IHT and several variations. For
brevity, we refer to \cite{BTW15} for an overview of these accelerated or
modified IHT variations, including the corresponding references.

Further closely related methods are (gradient) hard thresholding
pursuit~\cite{F11,YLZ18} and the general directional pursuit framework
from~\cite{BD08a} (although the latter is arguably more related to matching
pursuit schemes, which is the viewpoint in that paper).
Finally, the IHT method was extended to the case of group sparsity in~\cite{BKS21}.
    
\paragraph{\bf Matching Pursuit}\label{par:matchingPursuit}
Matching pursuit algorithms were originally developed to approximate a
signal (function) with relatively few atoms from a given overcomplete
dictionary, and can also be applied to obtain sparse approximate solutions
to underdetermined linear systems, i.e., to approximately solve
\cardcons{$\tfrac{1}{2}\norm{Ax-b}_2^2$}{$k$}{$\R^n$} or similar problems.
The basic \emph{matching pursuit} (MP) algorithm from~\cite{MZ93}
iteratively selects one column from~$A$ at a time, namely one that has
highest correlation with the residual $b-A\tilde{x}$, where $\tilde{x}$ is
zero except in the components corresponding to previously selected columns,
where the coefficients achieving the maximal residual-norm reduction in the
respective iteration are stored. The \emph{orthogonal matching pursuit}
(OMP) algorithm~\cite{PRK93} updates all coefficients of previously chosen
columns at each iteration rather than keeping them at their initial values
as in MP, thereby allowing for better approximations that potentially use
fewer columns.

There are many further variants that build on the general (O)MP greedy
principle and introduce different tweaks to improve the algorithmic
performance and/or achieve better solution quality and sparse recovery
guarantees under certain conditions: OMPR~\cite{JTD11} (\emph{OMP with
  replacement}) is an OMP variant that also allows for removal of
previously chosen columns from the solution support being constructed.  The
method is a special instantiation of a more general class of algorithms
that also generalizes, e.g., hard thresholding pursuit~\cite{F11}.  The
CoSaMP~\cite{NT09} algorithm (\emph{compressive sampling MP}) combines the
OMP idea with techniques from convex relaxation and other methods,
essentially iterating through residual updates with thresholding,
least-squares solution approximation on the estimated support, and further
thresholding. Similarly, StOMP~\cite{DTDS12} (\emph{stagewise OMP})
generalizes OMP by performing a fixed number of ``stages'' consisting of
obtaining support estimates by hard thresholding and updating the solution
estimate and residual based on the current support estimate in a
least-squares fashion; its analysis is focused on special choices of~$A$
with columns randomly generated from the unit sphere. \emph{Stagewise weak OMP} (SWOMP)~\cite{BD09b} uses thresholds based on the
maximal absolute value of entries in $A^\top r^k$ (w.r.t.\ the current
residual $r^k=b-Ax^k$) instead of the residual $\ell_2$-norm, and,
similarly to StOMP, proceeds in stages during which multiple elements are
added to the support estimate rather than one at a time (as in
OMP). Finally, ROMP~\cite{NV09} (\emph{regularized OMP}) groups the
elements of $A^\top r^k$ into sets of similar magnitude, then selecting the
set with largest $\ell_2$-norm and updating the signal estimate on the
corresponding support. Further papers on MP variants include \cite{JP08}
and \cite{WKS12} (generalized OMPs), \cite{CP19} (blended MP),
\cite{ASKAE16} (reduced-set MP), \cite{LKTJ17,LRKRSSJ18} (relating MP to
Frank-Wolfe and coordinate descent, respectively), and \cite{YGWX18}
(sparsity-adaptive MP), to name just a few. The algorithm from~\cite{N95}
can also be interpreted as a reduced-order MP method, cf.~\cite{P06}.

\paragraph{\bf Other Pursuit/Greedy Schemes}\label{par:otherPursuitGreedySchemes}    
The \emph{subspace pursuit} algorithm introduced in~\cite{DM09} (for
\cardmin{$\norm{Ax-b}_2\leq\delta$}, in concept if not directly) borrows
its idea from the so-called $A^*$ order-statistic algorithm known in coding
theory. In essence, it iteratively selects a fixed number of columns from
the measurement matrix $A$ that have high correlation with $b$ as the span
for a candidate subspace to contain the sought solution. The chosen column
subset is then updated/refined based on certain reliability criteria. The
method is similar to, but different in detail from, the matching pursuit
algorithms ROMP~\cite{NV09} and CoSaMP~\cite{NT09}. The paper~\cite{DM09}
also provides recovery and error guarantees based on RIP conditions, as
well as some simulation results comparing against OMP~\cite{PRK93}, ROMP,
and linear programming for BP($\R^n$).
    
The paper~\cite{BRB13} proposes and analyzes the \emph{gradient support
  pursuit (GraSP)} algorithm, which can be seen as a generalization of
CoSaMP~\cite{NT09} to the problem \cardcons{$f(x)$}{$k$}{$\R^n$} of finding
sparse solutions for generic cost functions $f:\R^n\to\R$. The GraSP method
iterates through computing gradients (or certain restricted subgradients,
in case~$f$ is nonsmooth) and thresholding their support, then
minimizing~$f$ over the joint support of the previous iterate and the
thresholded (sub)gradient, and finally retaining the~$k$ largest
absolute-value components of that solution to build the next iterate.
Several variants are discussed, including one that replaces the inner
minimization by a restricted Newton step. Reconstruction guarantees are
obtained w.r.t.\ newly introduced conditions (stable restricted Hessian or
stable restricted linearization, for smooth or nonsmooth~$f$,
respectively), and experimental results are provided for an application of
variable selection in logistic regression.

Finally, a greedy method based on relaxing an exact MIQP formulation of \cardcons{$\tfrac{1}{n}\norm{Ax-b}_2^2+\lambda\norm{x}_2^2$}{$k$}{$\R^n$} was put forth in~\cite{XD20}, along with approximation error bounds that do not require common conditions such as the RIP, and two further randomized variants. The paper also describes how to apply the algorithms to sparse inverse covariance estimation.

\paragraph{\bf Alternating Projections}
The papers \cite{HLN14,BLPW14} consider a reformulation of \cardmin{$Ax=b$}
as the feasibility problem
\[
  \text{find }x\in \{x:Ax=b\}\cap\{x:\norm{x}_0\leq k\}\eqqcolon X\cap \Sigma_k,
\]
noting that solutions coincide for the optimal choice of~$k$.  (In fact,
any solution for the feasibility problem is clearly an optimal solution of
\cardcons{$\norm{Ax-b}$}{$k$}{$\R^n$}, for any norm $\norm{\cdot}$.)  They
analyze the method of alternating projections, iterating by alternatingly
projecting onto $X$, a closed convex set with unique closed-form
(Euclidean) projection, and $\Sigma_k$, a non-convex set onto which one can
nevertheless project in a well-defined manner by means of
$\cH^0_{k}(\cdot)$. Thus, unlike the previously discussed work~\cite{CCG12}
(cf. Section~\ref{sec:inexactopt:surrogates:algorithms}), here, the focus
is on the actual sparsity instead of its $\ell_1$-norm surrogate. Local and
global convergence results are given for the alternating projection
algorithm in \cite{BLPW14} and \cite{HLN14}, respectively; the latter also
proposes a Douglas-Rachford splitting algorithm for the above feasibility
problem, establishing (local) convergence results for that method as well.

\paragraph{\bf Constrained Sparse Phase Retrieval}
For the cardinality-constrained sparse phase retrieval problem
\cardcons{$\norm{\abs{Ax}^2-b}_2^2$}{$k$}{$\R^n$}, a greedy method called
GESPAR is introduced in~\cite{SBE14}. It combines a local search (2-opt)
heuristic with an efficient damped Gauss-Newton method to minimize the
objective when variables are restricted to a candidate support. The GESPAR
algorithm invokes this scheme for different random initializations to
mitigate the impact of the local search getting stuck in local minima.

Other popular algorithms for
\cardcons{$\norm{\abs{Ax}^2-b}_2^2$}{$k$}{$\R^n$} are gradient-descent-like
methods based on the phase retrieval algorithm known as \emph{Wirtinger
  Flow} (WF) \cite{CLS15}; see its various variants such as truncated
WF \cite{CC15} or, for the amplitude-based formulation with $\abs{Ax}$
instead of $\abs{Ax}^2$, reshaped WF \cite{ZL16}. In particular, the
\emph{Thresholded WF}~\cite{CLM16} combines the basic WF iteration with
soft-thresholding w.r.t.\ adaptively defined threshold parameters; see also
the recent \emph{Sparse WF}~\cite{YWW19}, a hard-thresholding WF scheme.
The \emph{sparse truncated amplitude flow (SPARTA)} algorithm
from~\cite{WGCA17} is designed for the formulation
\cardcons{$\norm{\abs{Ax}-b}_2^2$}{$k$}{$\R^n$}, and resembles the method
from~\cite{YWW19} in that its iterations are also of an adaptive hard-thresholded
gradient-descent type (but differ due to the slightly different initial
problem formulation). %The SPARTA method combines these iterations with a
%support-restricted spectral initialization and adaptive threshold
%selection. Recovery guarantees for sufficiently sparse solutions are given
%for Gaussian measurements with high probability.
Typically, these WF-like algorithms require sophisticated (spectral)
initialization procedures to achieve certain convergence/success
guarantees, usually shown to hold with high probability in case $A$ is a
Gaussian random matrix.

Note that both \cite{YWW19} and \cite{WGCA17} write the cardinality
constraint in equality form, i.e., $\norm{x}_0=k$ rather than
$\norm{x}_0\leq k$. Nevertheless, owed to hard-thresholding, the
algorithms effectively do not distinguish between the equality and
inequality versions.
    
\paragraph{\bf Greedy Heuristic for Cardinality Minimization With
  Constant-Modu\-lus Constraints}
Besides the dedicated branch-and-cut MINLP solver briefly mentioned
earlier, the paper \cite{FHMPPT18} also introduces an effective randomized
greedy-like heuristic for the constant-modulus constrained cardinality
minimization problem\linebreak
\cardmin{$\norm{Ax-b}_2\leq\delta$,\,$\abs{x}\in\{0,1\}$,\,$x\in\C^n$}. It
proceeds by iteratively increasing the solution cardinality, randomly
initializing a vector with the current cardinality, and then (for each
sparsity level) evaluating a large but fixed number of random entry
modifications obeying the constant-modulus constraint, updating the
solution if the residual (w.r.t.\ the measurements) decreases, until
eventually the $\ell_2$-norm constraint bound is satisfied or the sparsity
level reaches~$n$. It seems conceivable that this heuristic idea could be
adapted to other cardinality minimization problems such as
\cardmin{$\norm{Ax-b}\leq\delta$} for various norms $\norm{\cdot}$, but to
the best of our knowledge, this has not yet been considered in the
literature.

\section{Scalability of Exact and Heuristic Algorithms}\label{sec:scalability}

A question that is both academically interesting and of practical importance is how well the different algorithms handle larger problem dimensions.
By design, exact solution algorithms as those presented in Section~\ref{sec:exactopt} strive not only to compute a solution but also to prove global optimality for this solution.
This quality assurance typically comes with a longer computation time, which can make exact solution algorithms infeasible for large-scale applications.
However, exact algorithms rooted in mixed-integer programming---like most of those we discussed---provide computational error bounds throughout the solution process.
Since MIP solvers often find very good solutions quickly (and spend most of the longer running time establishing, or proving, optima\-lity), it can thus be a viable strategy to terminate an exact algorithm prematurely, trading time for \emph{quantifiable} suboptimality. 
Heuristics (see Section~\ref{sec:inexactopt}), on the other hand, are \emph{designed} to compute good solutions fast and efficiently, which generally makes them more accessible for large-scale problems, albeit with the downside that the solution quality may fluctuate. The same can be said about the approach to solve (exactly or approximately) an easier surrogate problem or relaxation, e.g., the po\-pu\-lar $\ell_1$-norm methods also discussed in Section~\ref{sec:inexactopt}. Note that several such heuristics or model approximations also come with quality guarantees under certain conditions, but such conditions often may not hold in practice or are hard to verify.

Thus, when deciding which solution algorithm to use, one should not only take the problem dimension into consideration, but also how critical computation time is for the considered application, and how important it is to compute (provably) good solutions.
Note that these aspects are indeed highly application-specific: The same problem, in comparable dimensions, can come with completely different requirements on its solution, which may, in particular, forestall claiming any one method as ``the best'' for some problem. For instance, one may be interested in an actually optimal solution for a feature reduction task (say, \cardmin{$\norm{Ax-b}_2\leq\delta$}) and willing to spend significant computational resources to obtain it. In the context of dictionary learning for sparse coding, the same problem type may be encountered as a subproblem that needs to be solved repeatedly, preferably very quickly, with no strict requirements on the solution accuracy---then, it can already be satisfactory to merely take a single improvement step of some heuristic method or with respect to, e.g., the usual $\ell_1$-relaxation. 

Besides these fairly general observations, there are also several more technical points that can make it tricky to compare the performance and scalability of solution algorithms, especially when purely consulting published numerical results.
To illustrate some of these, let us again consider the well-known sparse regression problem \cardcons{$\norm{Ax-b}_2$}{$k$}{$\R^n$}, which depends on a matrix $A \in \R^{m \times n}$, a vector $b \in \R^m$ and the sparsity level $k \in \N$.
How hard it is to solve an instance of this problem depends not only on the number of variables $n$ (the ``primary'', ambient dimension), but, in fact, on \emph{all} problem size parameters $n,m,k$ and the relations between them. For example, in \cite{MBN19}, the problem is solved with exact algorithms for $n \in \{500, 1000\}$ and $k \in \{5,10,15\}$, and all considered methods show higher computation times and higher failure rates for larger values of $k$; see also \cite{BNCM16,BVP17} for more numerical experiments that also take into account the relation between $n,m,k$, and compare exact and heuristic solution algorithms. Additionally, the noise level encoded in the data $A,b$ (measurement noise or other data uncertainty) can affect the quality of the solutions computed with different solution algorithms, see, e.g., \cite{BKM16} for some empirical insight. The density of~$A$ may also be relevant, though it can only be controlled in certain applications; generally, sparser~$A$ allows for larger problems to be tackled due to enabling numerical speed-ups in, e.g., matrix-vector multiplication (often the computational bottleneck in first-order heuristic iterations) or linear programming (which is the backbone of modern MIP solvers). Moreover, the original problem is sometimes modified in order to be able to solve larger problem instances, e.g., by inserting an additional regularization term, see, e.g., \cite{BKM16,XD20}. For instance, in \cite{BVP17} the problem \cardcons{$\norm{Ax-b}_2^2 + \tfrac{1}{\lambda}\norm{x}_2^2$}{$k$}{$\R^n$} with an added Tikhonov regularization term is solved with an exact algorithm for $n \in \{50\,000, 100\,000, 200\,000\}$ and $k \in \{10, 20, 30\}$. It may be tempting to compare these scales to, e.g., those from~\cite{MBN19} mentioned earlier, but then one must keep in mind that the underlying model has been changed, so that the solutions not necessarily coincide. 

Arguably, a truly fair comparison of different methods also requires that the same test instances are used for the numerical evaluation. In some applications, widely-used benchmark data sets exist (e.g., for classification and other machine learning tasks many can be found in the online repository~\cite{DG17}) and implementation source code is often made publicly available, while in others, such a ``spirit of reproducibility'' may not be as commonly (or possibly, not as easily) adhered to, see, e.g., \cite{MHH18}, and references therein, for a broader discussion touching upon several disciplines. For random synthetic data, as is often encountered in, e.g., compressed sensing, comparison of numerical results across different works is still viable as long as the problem dimension parameters, probability distributions of the data, and noise levels are the same, or at least very comparable. However, regrettably often, numerical experiments employ (random) data with scale and sparsity parameters that may not allow a direct comparison to other works, consider only a selection of a few existing methods that may not reflect the state-of-the-art, and rarely test the scalability limits with respect to any of the relevant parameters or their relations. Moreover, algorithms are often prototyped by their respective authors for a few experiments that demonstrate their potential in some way, but are rarely tuned or implemented in a way that would allow them to reach their true potential. This may further complicate comparison and interpretation of numerical experiments from published literature, especially if the code is not made public and the actual implementation of some algorithm being discussed thus remains opaque. Even with published come, one may occasionally notice that elementary parts could be implemented much more efficiently, and generally has to deal with different programming languages as well. Finally, algorithmic parallelization capabilities should, in principle, also be taken into account (but introduce a host of potential new difficulties for comparisons), and of course, at least w.r.t.\ solution times, any fair comparison would require the respective algorithms to be run on the same machine. 

Thus, there are indeed many reasons for the apparent lack of ``ideal'' comparability of results in the literature, some more of which are the aspects discussed at the beginning of this section. Moreover, scalability may simply not be sufficiently relevant to an application context (e.g., if an application only ever yields problems with up to, say, a hundred variables, it does not really matter in that context whether problems with several thousand variables could also be solved, even though it might for other applications), or the sensitivity of a solution approach to algorithmic parameters might not be properly taken into account either w.r.t.\ different applications or when setting default values (so that performance may be unreliable on new data sets with different problem parameters). Finally, the term ``large-scale'' can also have very different meanings in different contexts. While in high-dimensional statistics or machine learning, the large-scale regime may encompass problems with several hundreds of thousands of variables, where even heuristic approaches may be slow or challenge memory limits on standard computers, in signal processing, a few hundred or a few thousand variables are often already considered large-scale, and even relatively generic exact methods may still work quite well. Similarly to a point made earlier, these contrasting meanings may even pertain to the same underlying problem formulation.

For the reasons laid out above, we do not include a list of ``problem sizes'' that can be solved with exact or heuristic methods for various problem classes in this survey. Generally, one can say that efficiently implemented heuristic or relaxation-based methods can ``often'' handle problem sizes with several thousand variables (see, e.g., \cite{LPT15} for various $\ell_1$-solvers), and up to hundreds of thousands of variables in extreme cases, e.g., \cite{JTD11,DMW19}, and that problem-specific exact mixed-integer programming algorithms are ``often'' efficient for problems with a few hundred variables up to a few thousand variables (e.g., \cite{MBN19,T19}), and can sometimes even be pushed to yield at least near-optimal solutions for problems with up to hundreds of thousands of variables in reasonable time (e.g., \cite{BVP17,HMR21b,MWZ19}). Also, generally, the sparser the solution (the smaller~$k$), the larger the problem size ($n$) that can be solved exactly, and problems that are convex (except for the cardinality part) are typically easier than closely related nonconvex ones. However, we emphasize that for specific problems in specific applications, it is hard to pinpoint any one method as the best, or the most scalable, and that this needs to be determined on a case by case basis, taking all the points mentioned above into consideration---at least as long as there is no truly comprehensive and fair computational study encompassing various applications and considering various problem size parameter combinations (which seems a daunting task to accomplish indeed).

\section{Conclusion and Final Remarks}\label{sec:conclusion}
In this paper, we have surveyed the vast literature that deals with
algorithmic approaches for optimization problems in which the cardinality
of a set of continuous variables has to be limited. This happens in a
variety of domains in the attempt of controlling the sparsity of the
solutions to those optimization problems because, for a number of reasons
that include, e.g., explainability, robustness, and easiness of
realization, sparse solutions are considered especially valuable.

More specifically, the paper attempted to discuss in a unified way
approaches that have been developed (and sometimes rediscovered with
different names) in several domains of applications. We gave particular
attention to three of those domains---namely, statistics and machine
learning, finance, and signal processing---but we also covered several
other connected areas (cf.\ Figure \ref{fig:applications_mindmap}), mainly
led by the types of models and algorithms we discussed.

We consider our effort as an initial but necessary and significant step in
the direction of consolidating and advancing the knowledge on formulations
and algorithms for solving this vast and fundamental class of optimization
problems. A further step in the same direction could come through a
comparison of software implementations of the algorithms surveyed in this
paper, so as to analyze and establish the difference in performances (accuracy and scalability) depending on data and contexts. This is a concrete major research goal,
though, admittedly, difficult to achieve (cf.~Section~\ref{sec:scalability}). Finally, throughout the paper, we
also pointed out several smaller ideas that, to the best of our knowledge,
have not been explored yet but seem worth investigating. We hope they may
provide viable research directions for the interested reader.

%\section*{Acknowledgements}
%We thank the anonymous referees for their useful remarks and suggestions, which helped improve the paper.
    
\bibliographystyle{siamplain}
% \bibliography{cardinality_opt_survey}

\end{document}